\documentclass[12pt]{amsart}
\usepackage[english]{babel}
 \usepackage[utf8x]{inputenc} 
\usepackage{amsmath}
\usepackage{cmll}
\usepackage{amssymb}
\usepackage{amsthm}
\usepackage{xcolor}
\usepackage{bbm}
\usepackage{mathrsfs}
\usepackage[all]{xy}
\usepackage{indentfirst}
\usepackage{fancyhdr}
\usepackage{newlfont}
\usepackage{setspace}
\usepackage{verbatim}
\usepackage{hyperref} 
\usepackage{bbm}

\usepackage{geometry}
\renewcommand{\H}{\mathbb{H}}
\newcommand{\B}{\mathbb{B}}

\newcommand{\G}{\mathbb{G}}

\newcommand{\N}{\mathbb{N}}

\newcommand{\R}{\mathbb{R}}

\newcommand{\cB}{\mathcal{B}}

\newcommand{\cV}{\mathcal{V}}

\renewcommand{\exp}{\mbox{\rm exp}\;\!}

\newcommand{\sgn}{\mbox{\rm sgn}\,}

\newcommand{\Lie}{\mathrm{Lie}}

\newcommand{\beqas}{\begin{eqnarray*}}
\newcommand{\eeqas}{\end{eqnarray*}}
\newcommand{\beqa}{\begin{eqnarray}}
\newcommand{\eeqa}{\end{eqnarray}}
\newcommand{\beq}{\begin{equation}}
\newcommand{\eeq}{\end{equation}}
\newcommand{\bce}{\begin{center}}
\newcommand{\ece}{\end{center}}

\newcommand{\pa}[1]{\left( #1 \right)}               
\newcommand{\ban}[1]{\left\langle  #1 \right\rangle}  

\newtheorem{The}{Theorem}[section]
\newtheorem{Lem}[The]{Lemma}
\newtheorem{Def}[The]{Definition}
\newtheorem{Rem}[The]{Remark}
\newtheorem{Pro}[The]{Proposition}
\newtheorem{Cor}[The]{Corollary}
\newtheorem{Exa}[The]{Example}
\newtheorem{Con}{Conjecure}

\newcommand{\bt}{\begin{The}}
\newcommand{\et}{\end{The}}
\newcommand{\bl}{\begin{Lem}}
\newcommand{\el}{\end{Lem}}
\newcommand{\bd}{\begin{Def}\rm}
\newcommand{\ed}{\end{Def}}
\newcommand{\br}{\begin{Rem}\rm}
\newcommand{\er}{\end{Rem}}
\newcommand{\bpr}{\begin{Pro}}
\newcommand{\epr}{\end{Pro}}
\newcommand{\bc}{\begin{Cor}}
\newcommand{\ec}{\end{Cor}}
\newcommand{\bj}{\begin{Con}}
\newcommand{\ej}{\end{Con}}
\newcommand{\bex}{\begin{Exa}}
\newcommand{\eex}{\end{Exa}}

\newtheorem{teo}{Theorem}[section]
\newtheorem{prop}[teo]{Proposition}
\newtheorem{cor}[teo]{Corollary}
\newtheorem{lem}[teo]{Lemma}

\theoremstyle{definition}

\newtheorem{deff}[teo]{Definition}

\newtheorem{Remark}[teo]{Remark}

\allowdisplaybreaks

\begin{document}
\begin{abstract}
		In this paper we construct the fractional powers of the sub-Laplacian in Carnot groups through an analytic continuation approach. In addition, we characterize the powers of the fractional sub-Laplacian in the Heisenberg group, and as a byproduct we compute the $k$-th order momenta with respect to the heat kernel. 
	\end{abstract}

\pagestyle{plain}


\title{The fractional powers of the sub-Laplacian in Carnot groups through an analytic continuation}

\author{Francesca Corni}
\address{Francesca Corni: Dipartimento di Matematica\\ Universit\`a di Bologna\\ Piazza di Porta S.Donato 5\\ 40126, Bologna-Italy}
\email{francesca.corni3@unibo.it }
\author{Fausto Ferrari}
\address{Fausto Ferrari: Dipartimento di Matematica\\ Universit\`a di Bologna\\ Piazza di Porta S.Donato 5\\ 40126, Bologna-Italy}
\email{fausto.ferrari@unibo.it }
\thanks{F.C.  and F.F. are partially supported by INdAM-GNAMPA 2023 project CUP E53C22001930001: {\it Equazioni completamente non lineari locali e non locali} and INdAM-GNAMPA 2024 project CUP E53C23001670001: {\it Free boundary
problems in noncommutative structures and degenerate operators}. F.F. is partially funded by  PRIN 2022 7HX33Z - CUP J53D23003610006, {\it Pattern formation in nonlinear phenomena}}
\thanks{AMS-2020, 35R03, 35R11}
\date{\today}

\maketitle

\tableofcontents
\section{Introduction}

The fractional Laplacian may be considered as the simplest example of a non-local fractional integral operator. A great interest about this object arises, since it allows to model various phenomena related to different fields. It is well known that fractional operators are connected to potential theory, probability, harmonic and fractional analysis, pseudo-differential operators and partial differential equations, see for instance \cite{Fol73}, \cite{Samko}, \cite{DPValdinoci} and \cite{Garofalo}.

On the other side, the setting of stratified groups represents a generalization of the Euclidean one, endowed with a rich structure, both from the algebraic and geometrical point of view. 

Actually, the goal of generalizing fractional operators to stratified groups gives rise to appealing questions and open problems.
The aim of this paper is to investigate the connection between the fractional sub-Laplacian and suitable generalized Riesz potentials in the context of Carnot groups. 

In particular, following an idea recalled in the book  by Landkof, \cite{Landkof}, based on the work of Marcel Riesz, \cite{MarcelRieszA}, \cite{MarcelRiesz},  we propose a natural extension of the definition of the real powers of the sub-Laplacian on a stratified group based on a suitable analytic continuation approach. We would like to represent an alternative face to the problem already dealt with by Folland \cite{Fol75}. 

Indeed, we think that this approach is very interesting for its simplicity and it is worth to be focused.

We start by considering a stratified group $\G$, i.e. a nilpotent simply connected Lie group whose Lie algebra admits a stratification. It is a standard fact that $\G$ can be represented as a vector space endowed with both a structure of stratified Lie algebra, with grading \begin{equation}\label{grading}
\G=V_1 \oplus \ldots V_{\iota},
\end{equation} and Lie group.

 We denote by $q$ and $Q=\sum_{i=1}^{\iota} i \dim(V_i)$ the topological and the homogeneous dimension of $\G$, respectively, and, following \cite{Fol75}, we suppose that $Q \geq 3$. We assume that a homogeneous distance $d$ on $\G$ is fixed, i.e. a distance $d$ such that $d(zx,zy)=d(x,y)$ and $d(\delta_t(x), \delta_t(y))=td(x,y)$ for every $x,y,z \in \G$ and $t>0$, where $\delta_t$ is a semigroup anisotropic dilation  associated with the grading \eqref{grading} of $\G$. 
 
We denote by $m_1$ the dimension of the first layer $V_1$ of the grading. Moreover, we assume that a scalar product on $\G$ is fixed and we consider an orthonormal basis $(e_1, \ldots, e_{m_1})$ of $V_1$. We set for every $i=1, \ldots, m_1$ the vector field $X_i \in \mathrm{Lie}(\G)$ such that $X_i(0)=e_i$. The sub-Laplacian associated with $(e_1, \ldots, e_{m_1})$ is the operator $$\mathcal{L}=-\sum_{i=1}^{m_1}X_i^2.$$ 
It is known that $\mathcal{L}$ is hypoelliptic,\cite{Hormander67},  left-invariant and $2$-homogeneous with respect to the anisotropic dilations $\{\delta_t \}_{t>0}$.

Fractional powers $\mathcal{L}^{-\frac{\alpha}{2}}$ of $\mathcal{L}$, for $\alpha \in (-2,0)$, have been introduced through the theory of semigroups, involving the fundamental solution $h$ of the heat operator $\mathcal{H}=\partial_t+\mathcal{L}$, that is the heat kernel on $\G$ of $\mathcal{H}$ (see Section \ref{sect:fraclap} for definitions). 

The operator $\mathcal{L}^{-\frac{\alpha}{2}}$, $\alpha \in (-2,0)$ takes the name of fractional sub-Laplacian and it has been studied by many authors. 

In \cite{FF15} the authors established a pointwise integral representation of $\mathcal{L}^{-\frac{\alpha}{2}}\phi(x)$ through a family of suitable generalized Riesz $(\alpha-Q)$-homogeneous kernels $P_{\alpha} \in C^{\infty}(\G \setminus \{0\})$.
They showed that for every $\alpha  \in (-2,0)$ if $u \in \mathcal{S}(\G)$, then $\mathcal{L}^{-\frac{\alpha}{2}}u \in L^2(\G)$ and 
\begin{equation}\label{onerep}\mathcal{L}^{-\frac{\alpha}{2}}u(x)= P.V. \int_{\G} (u(y)-u(x)) P_{\alpha}(x^{-1}y) dy.
\end{equation}
We start by considering the operator
\begin{equation}
\phi\star P_{\alpha}(x)= \int_{\G} \phi(xy) P_{\alpha}(y)dy,
\end{equation}
for $\phi \in \mathcal{S}(\G)$ and $x\in \G$, where the symbol $\star$ in \eqref{map} denotes the  convolution in the group, and we notice that it is well defined for every $\alpha \in (0,Q)$.

The main result of our paper relies in the construction of an analytic continuation of the map 
\begin{equation}\label{map}
(0, Q) \ni \alpha \to  \phi\star P_{\alpha}(x)= \int_{\G} \phi(xy) P_{\alpha}(y)dy,
\end{equation}
to the interval $(-2,Q)$ for any fixed $x \in \G$ and $\phi \in \mathcal{S}(\G)$. This allows to extend in a natural way the definition of the power $\mathcal{L}^{-\frac{\alpha}{2}}$ on the Schwartz class $\mathcal{S}(\G)$, for $\alpha \in (-2,Q)$. 

In our setting, we need to rely on various tools of geometric measure theory. Among these, in particular, we exploit the coarea formula proved in \cite{Mag6, Mag31}
and the Eikonal equation for the Carnot-Carath\'eodory distance proved in \cite{MSC2001}. 

On the other hand, we take into account various properties of the heat kernel $h$ associated with $\mathcal{H}$, developed in \cite{Fol75} and \cite{FMPPS18}. Actually, the main step in  our constuction is the estimate \cite[Theorem 20.3.3]{BLU07Stratified}, involving homogeneous Taylor polynomials on a stratified group. In fact, while in the Euclidean setting the main point of the continuation is the Pizzetti formula, in an arbitrary stratified group the Pizzetti formula is not available yet. More precisely, even if some results are available in the literature \cite{B2002}, \cite{B09}, they don't seem directly exploitable in our approach.

In the second part of the paper, we restrict ourselves to the simplest non-trivial case, given by the Heisenberg group $\H^n$. We show by induction the construction of the analytic continuation of the map in \eqref{map} up to 
$$ (-\infty, Q) \ni \alpha \to \phi \star P_{\alpha}(x),$$
for any fixed $x \in \G$ and $\phi \in \mathcal{S}(\G)$. In a natural way, this leads to an extension of the definition of $\mathcal{L}^{-\frac{\alpha}{2}}$ on the Schwartz class $\mathcal{S}(\H^n)$, for $\alpha \in (-\infty,Q)$. 

From our point of view, the Heisenberg group represents a privileged setting. In fact, in $\H^n$,  both the heat kernel $h$, the kernels $P_{\alpha}$ 
and the profile of the Carnot-Carath\'eodory ball $\partial \B_c(0,1)$
preserve various symmetries related to the structure of the group. 
In an arbitrary stratified group it is not clear how to proceed. 

These invariances allow to rely on the use of Taylor polynomials to extend the continuation of the map in \eqref{map} to every strip $(-2m-2, -2m]$, with $m \in \N$, through an inductive procedure.

Eventually, we provide simplified representations of the extended map $\alpha \to \psi(x,\alpha)$  on the intervals $(-4,Q)$, $(-6,Q)$ and $(-8,Q)$. By combining these representations with \cite[Proposition 1.5]{FMPPS18}, we obtain several geometric consequences. In particular, we compute the value of the integrals
\begin{equation}\label{geocons}
\begin{split}
\int_{\H^n} x_i^2 h(1, x) dx &\quad \int_{\H^n} x_i^4 h(1, x) dx  \quad \int_{\H^n} x_{2n+1}^2 h(1, x) dx\\
 \int_{\H^n} x_i^6 h(1, x) & dx \quad  \int_{\H^n} x_i^2 x_{2n+1}^2 h(1, x) dx,
\end{split}
\end{equation}
for $i \in \{1, \ldots, 2n\}$, that of course have an important probabilistic meaning.

Now, we briefly describe the content of the various sections. In Section \ref{sect:one} we introduce stratified groups and we collect many related definitions and results available in literature. Moreover, we deal with the sub-Laplacian $\mathcal{L}$ and fractional sub-Laplacian $\mathcal{L}^s$, with $s \in (0,1)$ and we introduce the family of the generalized Riesz potentials $P_{\alpha}$.

In Section \ref{sec-ext} we study the map
$$ (0,Q) \ni \alpha \to \psi(x,\alpha)=\phi \star P_{\alpha}(x),$$
and we build up its analytic continuation to the strip $(-2,0]$. Hence, we are in position to extend the definition of the fractional sub-Laplacian $\mathcal{L}^{s}$ to $s \in (-\frac{Q}{2},1)$.

In Section \ref{sect:three} we sketch the setting of the Heisenberg group $\H^n$ for fixing notation and we show some technical auxiliary results.

In Section \ref{sect-indu} we build up the analytic continuation of the function $\alpha \to \psi(x, \alpha)$ on the interval $(-\infty,Q)$. Then, we exploit the uniqueness of the analytic continuation to extend the notion of fractional sub-Laplacian $\mathcal{L}^s$ to every $s \in (-\frac{Q}{2}, \infty)$.

Finally, in Sections \ref{striscia2}, \ref{striscia3} and \ref{striscia4} we obtain simplified representations of $\alpha \to \psi(x,\alpha)$ on the intervals $(-4,Q)$, $(-6,Q)$ and $(-8,Q)$, respectively. We apply these representations to establish the value of the integrals in \eqref{geocons}.

\section{Preliminaries}
\label{sect:one}

\subsection{Stratified groups, Pansu differentiabily, coarea formula}
\label{sect:prel}
In this section we collect some definitions and more or less known results about stratified groups, that we state for fixing the notation.
\begin{deff}
A \textit{stratified group} or \textit{Carnot group} $\G$ of step $\iota$ is a connected, simply connected and nilpotent Lie group whose Lie algebra is stratified of step $\iota$, i.e. there exists a sequence of subspaces $\cV_j$ with $j\in\N$, such that $\cV_j= \{ 0 \}$ if $j>\iota$, $[\cV_1, \cV_i] = \cV_{i+1}$ for every $i,j\ge1$, $\cV_\iota \neq \{ 0 \}$ and
$ \mathrm{Lie}(\G)=\cV_1 \oplus \dots \oplus \cV_\iota$, 
where for $i,j \geq 1$ we have set
$$[\cV_i,\cV_j]=\text{span} \{ [X,Y]  :  X \in \cV_i, \ Y \in \cV_j\}.$$ 
\end{deff}

The exponential map $\mathrm{exp}: \mathrm{Lie}(\G) \to \G$ is a global diffeomorphism which allows to identify in a standard way $\G$ with $\mathrm{Lie}(\G)$ (see for instance \cite[Section 2.2]{BLU07Stratified}).
As a consequence, we can model a stratified group as a vector space $\G$ endowed with both a Lie bracket $[ \cdot, \cdot]$ with respect to which $\G$ is a stratified Lie algebra 
\begin{equation}\label{eq:cg}
\G=V_1 \oplus V_2 \ldots \oplus V_{\iota},
\end{equation} and a Lie group structure, whose product is given by the Baker-Campbell-Hausdorff operation associated with $[ \cdot, \cdot]$ (see \cite[Theorem 2.2.1, Theorem 2.2.13, Corollary 2.2.15, Theorem 2.2.24]{BLU07Stratified}). 

For every $x \in \G$ we denote the corresponding left translation by $$l_x:\G \to \G, \ l_x(y)=xy,$$ for every $y \in \G$. We denote by $H\G$ the distribution generated by $V_1$, it is called the \textit{horizontal distribution}.
A basis $(e_1,\ldots,e_q)$ of $\G$ is called {\em graded} if
the ordered family $ (e_{h_{i-1}+1},\ldots,e_{h_i}) $ is a basis of $V_i$ for every $i=1,\ldots,\iota$.
We call $(Z_1, \cdots Z_q)$ a \textit{Jacobian basis} of $\mathrm{Lie}(\G)$ if $(Z_1(0), \ldots Z_q(0))$ is a graded basis of $\G$.
A family of dilations is associated with the stratification of $\G$ in a natural way: for every $t>0$ we set $$\delta_t \left(\sum_{i=1}^\iota v_i \right)=\sum_{i=1}^{\iota}t^i v_i \qquad \mathrm{with} \ v_i \in V_i.$$
\begin{deff}(Degree)
Let $v \in \G$, we call $n \in \N$ the \textit{degree} of $v$, and we write $\mathrm{deg}(v)=n$, if 
$$ \delta_t(v)=t^n v,$$
hence if $v \in V_n$. Analogously, for a vector field $X \in \mathrm{Lie}(\G)$, we say that $\mathrm{deg}(X)=n$ if $X \in \mathrm{Lie}(V_n)$.
\end{deff}
We assume that a scalar product $ \langle \cdot, \cdot \rangle$ is fixed on $\G$
and we denote by $|\cdot|$ its associated norm. Moreover, we assume that $\langle \cdot, \cdot \rangle$ makes the $V_i$'s orthogonal. Taking into account the structure of vector space of $\G$, we may identify it with $T_0\G$. Thus, the scalar product $\langle \cdot, \cdot \rangle$ extends through the left translations to a left invariant Riemannian metric that we denote by $g$. 

A distance $d$ on $\G$ is called \textit{homogeneous} if $d(zx,zy)=d(x,y)$ and $d(\delta_t(x), \delta_t(y))=t d(x,y)$ for every $x,y,z \in \G$ and $t>0$.
Let us recall two concrete examples of homogeneous distances (see \cite[Definition 5.2.2]{BLU07Stratified} and \cite[Theorem 5.1]{FSSC5}).
First, we consider $x,y \in \G$ and we set $\Gamma_{x,y}$ the set of \textit{horizontal curves}, i.e. absolutely continuous curves $\gamma:[0,T] \to \G$ with $\dot{\gamma(t)} \in H\G_{\gamma(t)}$ for a.e. $t \in [0,T]$ for some $T>0$, such that $\gamma(0)=x, \ \gamma(T)=y$.
The \textit{Carnot-Carath\'eodory distance} or CC-distance between $x$ and $y$ is
$$ d_c(x,y)= \inf \left\{ \int^T_{0} \sqrt{g(\gamma(t))(\dot{\gamma}(t), \dot{\gamma}(t))} dt : \gamma \in \Gamma_{x,y} \right\}.$$ 
For every $x,y \in \G$, the Chow's theorem ensures that the set $\Gamma_{x,y}$ is non-empty hence $d_c$ is well defined. We denote by $\| x \|_c=d(x,0)$, for every $x \in \G$, the homogeneous norm associated with $d_c$. 
The second homogeneous distance we introduce is the distance $d_{\infty}$, defined for all $x,y \in \G$ as
 $$d_{\infty}(x,y)= \| y^{-1}  x \|_{\infty},$$ where the homogeneous norm $\| \cdot \|_{\infty}$ is 
$$ \| x \|_{\infty}:= \max\{ \varepsilon_j | x^j|^{1/j}, \ j=1, \dots, \kappa  \},$$for $x=\sum_{i=1}^\kappa x^j \in \G$, with $x^j \in V_j$, with $\varepsilon_1=1$, $\varepsilon_{j} \in (0,1]$ suitable positive constants depending on the group structure .
In this paper, we assume that a homogeneous distance $d$ is fixed on $\G$ and we denote by $\| x \|=d(x,0)$, for every $x \in \G$, the corresponding homogeneous norm. Moreover, we introduce the following closed balls 
\begin{align*}
\B(x,r)&= \{ x \in \G: \|x\| \leq r \},\\
\B_E(x,r)&=\{ x \in \G: |x| \leq r \},
\end{align*}
 for every $x \in \G$ and $r>0$. 
The closed ball centered at $x$ of radius $r$ corresponding to $d_c$ will be denoted as $\B_c(x,r)$. 
We recall that all homogeneous norms on $\G$ are equivalent (see for instance \cite[Theorem 5.1.4]{BLU07Stratified}), hence if $d_1$, $d_2$ are two homogeneous distances on $\G$, then there exist a positive constant $C>0$ such that
$$\frac{1}{C}d_1(x,y) \leq d_2(x,y) \leq C d_1(x,y),$$
 for every $x,y \in \G$.
We denote by $q$ the topological (or linear) dimension of $\G$ and we call \textit{homogeneous dimension} of $\G$ the integer $$Q=\sum_{i=1}^\iota i \dim(V_i),$$ which is also the Hausdorff dimension of $\G$ with respect to a homogeneous distance $d$. We assume that $Q\geq 3$.
Let us recall also the notion of Hausdorff measure in this context. Let $\mathcal{F} \subset \mathcal{P}(\G) $ be a non-empty family of closed subsets of a Carnot group $\G$, endowed with a homogeneous distance $d$. Let $\zeta: \mathcal{F} \to \R^+$ be a function such that $0 \leq \zeta(S)<\infty$ for any $S \in \mathcal{F}$. If $\delta>0$, and $A \subset \G$, we define
\begin{equation}
\phi_{\delta, \zeta}(A)= \inf \left\{ \sum_{j=0}^{\infty} \zeta(B_j) \ : \ A \subset \bigcup_{j=0}^\infty B_j , \ \mathrm{diam}(B_j) \leq \delta, \ B_j \in \mathcal{F} \right\}.
\end{equation}
If $\mathcal{F}$ coincides with the family of closed balls with respect to the distance $d$ and $\zeta(\B(x,r))=r^{\alpha}$ we call 
$$ \mathcal{S}^{\alpha}(A):= \sup_{\delta>0} \phi_{\delta, \zeta}(A)$$
the \textit{$\alpha$-spherical Hausdorff measure} of $A$.
If $\mathcal{F}$ coincides with the family of all
closed sets, and for $ k \in \{1, \dots, q \}$ we set 
$c_{k}=\omega_k$, the $k$-dimensional Euclidean measure of the unitary Euclidean ball in $\R^k$, and $\zeta(B)=c_{k} \left(\frac{\mathrm{diam}(B)}{2} \right)^{k},$ we call 
$$ \mathcal{H}^{k}_E (A):= \sup_{\delta>0} \phi_{\delta, \zeta}(A)$$
the \textit{$k$-Euclidean Hausdorff measure}.
In particular, we stress that $\mathcal{H}_E^k$ is the Hausdorff measure with respect to the distance generated by the Riemannian metric $g$ fixed on $\G$. Nevertheless, if we consider $\G$ in coordinates with respect to an orthonormal basis, then $\mathcal{H}_E^k$ coincides with the Euclidean Hausdorff measure and this justify our notation.
The $q$-dimensional Lebesgue measure on $\G$ is the Haar measure of $\G$. It is left invariant and $Q$-homogeneous.
On the other hand, for every $m>0$, for every measurable set $ A\subset \mathbb{G}$ and for all $x \in \G$ and $t>0$ we have
\begin{equation*}
 \mathcal{S}^m(l_x( A))=\mathcal{S}^m(A)  \quad \mathrm{and} \quad
\mathcal{S}^m(\delta_t(A))=t^m \mathcal{S}^m( A).
\end{equation*}
We refer the reader for example to \cite{SerraCassano2016} for more details about an introduction to stratified groups.
\subsubsection{Pansu differentiability}
A map $L:\G \to \R$ is a \textit{h-homomorphism} if it is a group homomorphism such that $L(\delta_t(x))=t(L(x))$, for every $x \in \G$ and $t>0$. It is not difficult to verify that $L$ is linear and $L|_{V_j}\equiv 0$ for every $j=2, \ldots, \iota$. We denote by $\mathcal{L}_h(\G, \R)$ the family of h-homomorphisms between $\G$ and $\R$.
Given $L, T \in \mathcal{L}_h(\G, \R)$, we define the distance $$d_{\mathcal{L}_h(\G, \R)}(L,T):= \sup_{\|x\| \leq 1} |T(x)^{-1}L(x) |.$$ This class of maps is the starting point to state a notion of differentiability fitting the setting of stratified groups.
Let $\Omega$ be an open set in $\G$, let $f : \Omega \to \R$ be a function and let $x \in \Omega$.
The map $f$ is said  \textit{Pansu differentiable} at $x$ if there exists a h-homomorphism $L:\G \to \R$ that satisfies 
\begin{equation}
\label{pansudiff}
|  f(y) - f(x)-  L(x^{-1} y) |= o ( \|x^{-1} y \|) \quad  \text{as} \quad \|x^{-1} y \|  \to 0.
\end{equation}
When condition \eqref{pansudiff} is verified, we call $L$ the \textit{Pansu differential} of $f$ at $x$ and we denote it by $D_hf(x)$. 
The unique vector $\nabla_{H}f(x)\in \G$ such that
\begin{equation}\label{def-grad}
D_hf(x)(z)=\ban{\nabla_{H}f(x),z}
\end{equation}for every $z\in \G$ is called the \textit{horizontal gradient} of $f$ at $x$.
Assuming that $(v_1, \dots, v_q)$ is a graded orthonormal basis of $\G$ and setting $X_j \in \mathrm{Lie}(\G)$ the left invariant vector field such that $X_j(0)=v_j$ for every $j=1, \dots, q$ \begin{equation}\label{forma-grad}
\nabla_Hf(x)=\sum_{i=1}^{m_1} X_if(x) v_i,
\end{equation}
where $m_1=\dim(V_1)$ (see \cite[Proposition 3.10]{SerraCassano2016} and the references therein).
We say that $f:\Omega\to\R$
is \textit{continuously Pansu differentiable} on $\Omega$ if $f$ is Pansu differentiable at every point of $\Omega$ and the function 
$D_hf: \Omega \to \mathcal{L}_h(\G, \R)$ is continuous with respect to the topology induced by $d_{\mathcal{L}_h(\G, \R)}$.
We denote by $C^1_h(\Omega)$ the family of continuously Pansu differentiable mappings from $\Omega$ to $\R$ and we observe that $C^1(\Omega) \subsetneq C^1_h(\Omega)$.
\subsubsection{Coarea formula in stratified groups}
In this subsection we recall a suitable coarea formula in Carnot groups. It is interesting to notice that the starting point of the proof of Theorem \ref{teoco} is a previous coarea formula for bounded variation maps (\cite{FSSC1, GN96, MSC2001}).
\begin{deff}
Let $v \in V_1 \setminus \{ 0 \}$ be a vector and consider its orthogonal hyperplane $\mathcal{L}(v) \subset \G$. We define 
$$ \theta^g_{Q-1}(v):= \max_{z \in \B(0,1)}\mathcal{H}_E^{q-1}(\mathcal{L}(v) \cap \B(z,1))$$
the \textit{metric factor} with respect to the direction $v$.
\end{deff}
We explicitly remark that $\theta^g_{Q-1}(v)$ depends both on the homogeneous distance $d$ and on the Riemannian metric $g$.

The following theorem is proved in \cite[Theorem 3.5]{Mag6} (see also \cite{Mag31}).
\begin{teo}\label{teoco}
Let $A \subset \G$ be a measurable set and consider a Lipschitz function $u: A \to \R$. For any measurable function $h:A \to [0, \infty)$ the equality
$$ \int_A h(x) | \nabla_Hu(x)| dx = \int_\R \int_{u^{-1}(r) \cap A} h(x) \theta^g_{Q-1}( \nabla_H u(x)) d \mathcal{S}^{Q-1}(x) dr$$
holds. Notice that $\theta^g_{Q-1}( \nabla_H u(x))$ and $\mathcal{S}^{Q-1} $
are both considered with respect to $d$.
\end{teo}

On the other hand, by \cite[Proposition 1.9]{Mag6} and \cite[Theorem 5.2]{Mag31} we know that the distance $d_{\infty}$ preserves enough symmetries to guarantee the constancy of the metric factor. This is the content of the following proposition.
\begin{prop}
\label{precedente}
If $d=d_{\infty}$, then there is a constant $\alpha_{Q-1}>0$ such that $$\theta^g_{Q-1}(v)=\alpha_{Q-1}, \qquad \mathrm{ for \ every } \ v \in V_1 \setminus \{0\}.$$
\end{prop}
Thus, in the setting of Proposition \ref{precedente}, we may set
$$\mathcal{S}^{Q-1}_{\infty}:= \alpha_{Q-1} \mathcal{S}^{Q-1}$$
to yield a simplified coarea formula, which follows from the combination of \cite[Corollary 3.6]{Mag6} and Proposition \ref{precedente}.
\begin{teo}
\label{coarea}
Let $d=d_{\infty}$, let $A \subset \G$ be a measurable set and let $u: A \to \R$
be a Lipschitz function. For any measurable function $h:A \to [0, \infty)$, the following equality holds
$$ \int_A h(x) | \nabla_Hu(x)| dx = \int_\R \int_{u^{-1}(r) \cap A} h(x) d \mathcal{S}_\infty^{Q-1} dr.$$
\end{teo}
Now, in order to obtain a further simplified expression of the coarea formula, we recall the Eikonal equation associates with the CC-distance $d_c$, which is a consequence of \cite[Theorem 3.1]{MSC2001}.
\begin{teo}
\label{Eik}
For a.e. $ x \in \G$ the equality
$ | \nabla_H d_c(x, 0)| =1$ holds.
\end{teo}

By combining Theorems \ref{coarea} and Theorem \ref{Eik}, setting $u(x):\G \to \R, \ u(x)=d_c(x,0)$ as the Lipschitz function in Theorem \ref{coarea}, we get the following formula, which retraces the Euclidean coarea formula.
\begin{cor}(Coarea formula)
\label{cor-coarea}
If $d=d_{\infty}$ and $A \subset \G$ is a measurable set,
then for any measurable function $h:A \to [0, \infty)$ we have
$$ \int_A h(x) dx = \int_\R \int_{\partial \B_c(0,r) \cap A} h(x) d \mathcal{S}_\infty^{Q-1} dr.$$
\end{cor}

\subsection{Taylor polynomial on stratified groups} 
In this subsection we introduce homogeneous Taylor polinomials on $\G$. We mainly follow \cite[Chapter 20]{BLU07Stratified}.
We choose $m \in \N$ and for every $m$-multi-index $\gamma =(\gamma_1, \ldots, \gamma_m)\in (\N \cup \{ 0 \})^m$ we set
$$|\gamma|=\gamma_1+\gamma_2+\dots \gamma_m \qquad \gamma!=\gamma_1! \gamma_2! \dots \gamma_m! \qquad \mathrm{and} \qquad x^{\gamma}=x_1^{\gamma_1}x_2^{\gamma_2}\dots x_m^{\gamma_m} \quad \mathrm{for \ every \ } x \in \R^m.$$
Moreover, we introduce
$$(Z_1, \cdots, Z_m)^{\gamma}=Z_1^{\gamma_1} \cdots Z_m^{\gamma_m}, \quad \mathrm{ for \ every \ } Z_1, \ldots Z_m \in \Lie(\G).$$
We say that an $m$-multi-index $\gamma =(\gamma_1, \ldots, \gamma_m)\in (\N \cup \{ 0 \})^m$ is \textit{even} if $\gamma_j$ is even (0 being even) for every $j \in \{1, \ldots, m \}$.\\
If $f \in C^{\infty}(\G,\R)$ and $(Z_1, \ldots, Z_q)$ is a Jacobian basis of $\mathrm{Lie}(\G)$, then for every $n \in \N \cup \{0 \}$ there exists one and only one polynomial $p$ homogeneous of degree at most $n$ such that
\begin{equation}\label{taipol}
(Z_1 \cdots Z_q)^{\beta}(p)(0)= (Z_1 \cdots Z_q)^{\beta}(f)(0),
\end{equation}
for every multi-index $\beta \in ( \N \cup \{0\})^q$ such that $|\beta|\leq n$.
The unique polynomial satisfying \eqref{taipol} is denoted by $\mathcal{T}_n(f,0)(x)$ and it is called the \textit{$Z$-Mac Laurin polynomial homogeneous of degree $n$} related to $f$.
Then, for every $x,x_0 \in \G$ we introduce the so-called \textit{$Z$-Taylor polynomial of $\G$-degree $n$} related to $f$ centered at $x_0$ as
\begin{equation}\label{transla}
\mathcal{T}_n(f,x_0)(x):=\mathcal{T}_n(f \circ l_{x_0},0)(x_0^{-1}x).
\end{equation}
Let us recall an important estimate involving $Z$-Taylor polynomials, obtained by \cite[Theorem 20.3.3]{BLU07Stratified}, setting $d=d_c$ (see also \cite[Theorem 1.37]{FS82}).
\begin{teo}(Stratified Taylor formula)
\label{Taylor}
Let $\G$ be a stratified group. There is a constant $b=b(\G)$ such that for every $n \in \N \cup \{ 0 \}$, there exists a constant $c=c(\G,n)>0$ such that
\begin{equation} \label{Taylorn}
|f(xy)-\mathcal{T}_n(f,x)(xy)| \leq c \|y\|_c^{n+1} \sup_{\{\|z\|_c \leq b^{n+1} \|y\|_c\}} \sup_{\{|\beta|=n+1\}} |(Z_1 \cdots Z_{m_1})^{\beta} f(xz)|,
\end{equation} 
for all $f \in C^{n+1}(\G,\R)$, for every $x,y \in \G$. Here $\beta$ is a $m_1$-multi-index and $m_1=\dim(V_1)$.
\end{teo}
An explicit form of the $Z$-Taylor polynomial homogeneous of $\G$-degree $n$ is known. We refer the reader to \cite[Corollary 20.3.15]{BLU07Stratified} (see also \cite[Theorem 2]{B2008} formula (43)).
\begin{teo}\label{explicit}
Let $n \in \N \cup \{0\}$, let $f \in C^{n+1}(\G,\R)$ and choose $x_0 \in \G$. Let $(Z_1, \dots, Z_q)$ be a Jacobian basis of $\mathrm{Lie}(\G)$, then
$$ \mathcal{T}_n(f,x_0)(x)=f(x_0)+ \sum_{h=1}^n \sum_{\substack{k=1, \ldots, n\\ i_1, \ldots,i_k \leq q, \\ \mathrm{deg}(Z_{i_1})+ \ldots \mathrm{deg}(Z_{i_k})=h}} \frac{(x_0^{-1}x)_{i_1} \ldots (x_0^{-1}x)_{i_k}}{k!} Z_{i_1} \ldots Z_{i_k} f(x_0),$$
for every $x \in \G$.
\end{teo}

\subsection{Fractional sub-Laplacian on stratified groups}
\label{sect:fraclap}
In this subsection, we introduce the sub-Laplacian on $\G$ and its powers. We choose a graded orthonormal basis $(v_1 , \dots, v_{m_1})$ of $V_1$ and we let $X_i \in \mathrm{Lie}(\G)$ be the left invariant vector field such that $X_i(0)=v_i$ for every $i=1, \dots, m_1$. 
We call \textit{sub-Laplacian} of $\G$ the second-order differential operator
\begin{equation}\label{sublaplacian}
\mathcal{L}:= - \sum_{i=1}^{m_1} X_i^2.
\end{equation}
It is a positive operator associated with the choice of the basis $(v_1, \ldots, v_{m_1})$.

\begin{Remark}
Since $\G$ is a stratified group, by Hormander's theorem the operator $\mathcal{L}$ is hypoelliptic. By \cite[Theorem 2.1]{Fol75} there exists a unique fundamental solution for $\mathcal{ L}$.
\end{Remark}
Let $\widehat{\G}$ be the homogeneous group $$\widehat{\G}=\R \times \G,$$ endowed with the product
$$(t,x) \cdot (t'.y)=(t+t', xy)$$ and the following stratification
$$ \widehat{\G}= V_1 \oplus \widehat{V_2}  \oplus \dots \oplus V_{\iota}, \quad \mathrm{ with  }\quad\mathrm{Lie}(\widehat{V_2})=\mathrm{span} \left\{ \frac{\partial}{\partial t}, \mathrm{Lie}(V_2) \right\}.$$  The group $\widehat{G}$ may be equipped with the dilations
\begin{equation}\label{dil}
\widehat{\delta}_s(t,x)=(s^2t, \delta_s(x)),
\end{equation}
for every $s>0$. We consider on $\widehat{\G}$ the heat operator $$\mathcal{H}= \mathcal{L} + \frac{\partial}{\partial t},$$ where
$\mathcal{L}$ is the sub-Laplacian in \eqref{sublaplacian}.
The operator $\mathcal{H}$ is invariant with respect to the left translations of $\widehat{\G}$ and homogeneous of degree 2 with respect to the dilations in \eqref{dil}, i.e. $\mathcal{H}(u \circ \widehat{\delta}_s)=s^2 \mathcal{H}u$ for $s>0$. Moreover, $\mathcal{H}$ is hypoelliptic and it has a unique fundamental solution $h$ called the \textit{heat kernel} on $\G$ (see \cite{Fol75}). We collect some properties of $h$ addressing the reader to \cite[Proposition 1.68]{FS82}, \cite[Theorem 1.1]{FMPPS18} and \cite[Proposition 1.74]{FS82} and to the references therein.
\begin{teo}
\label{T1}
Let $h:\R \times \G \to \R$ be the heat kernel associated with the sub-Laplacian $\mathcal{L}$ on $\G$, then the following conditions hold
\begin{itemize}
\item[(i)] $h \in C^{\infty}(\widehat{\G} \setminus \{(0,0)\})$.
\item[(ii)] $ \mathcal{H} h=0$ on $(0, \infty) \times \G $.
\item[(iii)] $\int_{\G} h(t,x) dx=1$, for every $t>0$; $h(t,x)=h(t,x^{-1})$, $h(t,x)\geq 0$, for every $x \in \G$ and $t \in \R$ and $h(t,x)=0$, when $t\leq 0$.
\item[(iv)] $h(t, \cdot ) \star h(\tau, \cdot)= h(t+\tau, \cdot)$, for every $t, \tau >0$.
\item[(v)] $h(\widehat{\delta}_s(t,x))=s^{-Q} h(t,x)$, for every $s, t>0$ and $x \in \G$.
\item[(vi)] There exists a constant $c>0$ such that 
$$ c^{-1} t^{-Q/2} \mathrm{exp}\Big(-\frac{ \| x \|^2}{c^{-1}t}\Big) \leq h(t,x) \leq c t^{-Q/2} \mathrm{exp}\Big(-\frac{\| x \|^2}{ct}\Big),$$
for every $x \in \G$ and $t>0$.
\item[(vii)] 
$h(t, \cdot) \in \mathcal{S}(\G)$, for every $t>0$.
\end{itemize}
\end{teo}
Now, we introduce the \textit{heat semigroup}:
\begin{equation} 
e^{-t \mathcal{L}}f(x):=(f \star h(t, \cdot))(x)=\int_\G h(t,y^{-1}x)f(y)dy, 
\end{equation}
for every $f \in L^1(\G)$, where $\star$ denotes the group convolution.
We define the fractional sub-Laplacian $\mathcal{L}^s$, $0<s<1$ as follows (see \cite{Yosida}). For every map $u \in W^{2s,2}(\G)$ and for every $x \in \G$
\begin{equation}\label{frac-lap}
\mathcal{L}^s u(x):= \frac{1}{\Gamma(-s)} \int_0^{\infty}t^{-s-1} (e^{-t\mathcal{L}} u(x)-u(x)) dt.
\end{equation}
See \cite[Section 1.2]{FMPPS18} and the references therein for the definition and an introduction to fractional Sobolev space on a stratified group. The operator $\mathcal{L}^s$ is non-local. \\
We recall the following result about the behaviour of $\mathcal{L}^s$, when $s$ is close to one, see also \cite[Proposition 1.5]{FMPPS18}. We remind the proof of the theorem to point out some technical aspects of the result. By $\mathcal{S}(\G)$ we denote the Schwartz class on $\G$ (see \cite[Section 1.D]{FS82}).
\begin{prop}\label{prop-limauno}
For every $\phi \in \mathcal{S}(\G)$ and for every $x \in \G$
$$\lim_{s \to 1^-} \mathcal{L}^{s}\phi(x)=\mathcal{L}\phi(x).$$
\end{prop}
\begin{proof}
Using the change of variables $\xi=t^{1-s}$ on \eqref{frac-lap}, we get
\begin{equation}\mathcal{L}^s\phi(x)=\frac{1}{(1-s)\Gamma(-s)} \int_0^{\infty} (e^{-\xi^{\frac{1}{1-s}\mathcal{L}}}\phi(x)-\phi(x))\xi^{-\frac{1}{1-s}} d \xi.
\end{equation}
Clearly, if $0\leq \xi <1$ then $\lim_{s \to 1^-}\xi^{\frac{1}{1-s}}=0$. Since $\phi \in \mathcal{S}(\G)$, by \cite[Theorem 3.1(ii)]{Fol75} (notice that $\mathcal{D(\G)}$ is $L^{\infty}$-dense in $\mathcal{S}(\G)$) we have
\begin{equation}\label{1.17}
\left\| \frac{e^{-h\mathcal{L}}\phi-\phi}{h}+\mathcal{L}\phi \right\|_{L^{\infty}(\G)} \to 0 \ \mathrm{as} \ h \to 0.
\end{equation}
Set $c(s)=\frac{1}{(1-s)\Gamma(-s)}=\frac{-s}{(1-s)\Gamma(1-s)}$. We claim that for every $x \in \G$
\begin{equation}\label{1.18}
\lim_{s \to 1^-} c(s) \int_0^1 (e^{-\xi^{\frac{1}{1-s}\mathcal{L}}}\phi(x)-\phi(x))\xi^{-\frac{1}{1-s}} d \xi =\frac{1}{2}\mathcal{L}(x).
\end{equation}
Indeed, denoting by $f_{s,x}(\xi):=((e^{-\xi^{\frac{1}{1-s}\mathcal{L}}}\phi(x)-\phi(x))\xi^{-\frac{1}{1-s}}$ and using \eqref{1.17} we get
\begin{equation}\label{1.19}
f_{s,x}(\xi) \to -\mathcal{L}(x)
\end{equation}
for every $x \in \G$. Moreover,
\begin{equation}\label{1.20}
\begin{split}
|f_{s,x}(\xi)|&=\left| ((e^{-\xi^{\frac{1}{1-s}\mathcal{L}}}\phi(x)-\phi(x))\xi^{-\frac{1}{1-s}} \right|=\left| \frac{1}{\xi^{\frac{1}{1-s}}} \int_0^{\xi^{\frac{1}{1-s}}} \frac{d}{d\tau} e^{-\tau \mathcal{L}}\phi(x) d \tau \right|\\
&= \left| \frac{1}{\xi^{\frac{1}{1-s}}} \int_0^{\xi^{\frac{1}{1-s}}}  e^{-\tau \mathcal{L}} \mathcal{L}\phi(x) d \tau \right| \leq \| \mathcal{L}\phi\|_{L^{\infty}(\G)},
\end{split}
\end{equation}
where in the last inequality we used that $\phi \in \mathcal{S}(\G)$ and the fact, proved in \cite[Proposition 3.2.1]{BH91}, that the heat semigroup preserves one-sided bounds, namely
\begin{equation}\label{1.6}
f \leq C \mathrm{ \ a.e. \ in \ }\G \Rightarrow e^{-t\mathcal{L}}f \leq C \mathrm{ \ a.e. \ in \ } \G \qquad \forall t \geq 0.
\end{equation}
Thus, our claim \eqref{1.18} follows using the definition of $c(s)$, hence from the fact that $\lim_{s\to 1^-}c(s)=-1$, and from \eqref{1.19} and \eqref{1.20}. Furthermore, using \eqref{1.6}, we obtain
\begin{align*}
c(s) \int_{1}^{\infty} |e^{-\xi^\frac{1}{1-s}\mathcal{L}}\phi(x)-\phi(x)| \xi^{-\frac{1}{1-s}} d \xi &\leq 2c(s) \|\phi\|_{L^{\infty}(\G)} \int_1^{\infty}\xi^{-\frac{1}{1-s}}d \xi\\
&= - 2  \|\phi\|_{L^{\infty}(\G)}\frac{-s}{(1-s)\Gamma(1-s)} \frac{1-s}{-s}\\
&= - 2 \frac{1}{\Gamma(1-s)}  \|\phi\|_{L^{\infty}(\G)} \to 0
\end{align*}
as $s \to 1^-$. Thus, letting $s \to 1^-$ we get the thesis.
\end{proof}
\subsection{Generalized Riesz potentials}
\label{section-kernel}
In this section we introduce a family of generalized Riesz potentials.
For every $ \alpha< Q$, $\alpha \neq 0, -2, -4, \cdots$ we set $P_{\alpha}:\G \to \R$ as 
$$ P_{\alpha}(x)=\frac{1}{\Gamma \left( \frac{\alpha}{2} \right)}\int_0^{\infty} t^{\frac{\alpha}{2}-1} h(t,x) dt.$$
We observe that $P_{\alpha}$ coincides with the map $R_{\alpha}$ defined in \cite{Fol75, FF15, FMPPS18} for $0<\alpha<Q$.
We collect below some properties of the $P_{\alpha}$ (see \cite[Proposition 3.17]{Fol75}).
\begin{enumerate}
\item \label{item1} By Theorem \ref{T1}(vi), for every $ x \neq 0$ $$h(t,x)=o(t^N) \ \mathrm{as} \ t \to 0^+$$ for every $N \in \N$. For every $x \in \G$, $$h(t,x) \leq c t^{-\frac{Q}{2}}$$ for some constant $c\geq 0$ and for every $t>0$. \\
Thus, $P_{\alpha}(x)$ converges absolutely if $\alpha<Q$ for every $x \neq 0$.
\item\label{item2} For every $\alpha<Q$, $P_{\alpha} \in C^{\infty}(\G \setminus \{0\})$, since $h(t,x) \in C^{\infty}(\widehat{\G }\setminus \{(0,0)\})$. 
\item \label{item-hom} For every $0<\alpha<Q$, $P_{\alpha}$ is non-negative. Moreover, for every $\alpha<Q$ it is homogeneous of degree $\alpha-Q$ with respect to the intrinsic dilations $\delta_s$, $s>0$. Let us show the homogeneity, for $s>0$ we have
\begin{align*}
P_{\alpha}(\delta_s(x))&= \frac{1}{\Gamma(\frac{\alpha}{2})} \int_0^{\infty} t^{\frac{\alpha}{2}-1} h(t,\delta_s(x)) dt\\
&= \frac{1}{\Gamma(\frac{\alpha}{2})} \int_0^{\infty} t^{\frac{\alpha}{2}-1} h\left(\frac{1}{s^2}s^2t,\delta_s(x)\right) dt\\
&= \frac{1}{\Gamma(\frac{\alpha}{2})} \int_0^{\infty} t^{\frac{\alpha}{2}-1} s^{-Q}h\left(\frac{1}{s^2}t,x \right) dt,
\end{align*}
where we exploited Theorem \ref{T1}(v). Let us now perform the change of variables $t=s^2\tau$ so that we get
\begin{align*}
P_{\alpha}(\delta_s(x))&= \frac{1}{\Gamma(\frac{\alpha}{2})} \int_0^{\infty} (s^2 \tau)^{\frac{\alpha}{2}-1} s^{-Q}h(\tau,x) s^2 d\tau\\
&=s^{\alpha-2-Q+2}\frac{1}{\Gamma(\frac{\alpha}{2})} \int_0^{\infty}  \tau^{\frac{\alpha}{2}-1} h(\tau,x) d\tau\\
&=s^{\alpha-Q}P_{\alpha}(x).
\end{align*}
\end{enumerate}
We introduce for every $ \alpha<Q$ and for every $x \in \G$, $x \neq 0$, the function
\begin{equation}
 \| x \|_{\alpha}:=\left( \int_{0}^{\infty} t^{\frac{\alpha}{2}-1} h(t,x) dt \right)^{\frac{1}{\alpha-Q}}.
 \end{equation}
Notice that for every $\alpha<Q$, $\alpha \neq 0,-2,-4, \ldots$ and $x \in \G$, $x \neq 0$, the equality $$P_{\alpha}(x)= \frac{1}{\Gamma(\frac{\alpha}{2})} \|x\|_{\alpha}^{\alpha-Q}$$
holds.
The map $\| \cdot \|_{\alpha}$ is a homogeneous norm on $\G$ smooth outside the identity element (see \cite[Section 1.1]{FMPPS18} and Item \ref{item-hom} in Section \ref{section-kernel}). Since on a Carnot group every homogeneous norms are equivalent, for every $\alpha<Q$ there exists a constant $c_{\alpha}>0$ such that
\begin{equation}\label{relcc-alfa}
\frac{1}{c_{\alpha}} \| x \|_c \leq \|x \|_{\alpha} \leq c_{\alpha} \|x\|_c,
\end{equation}
for every $x \in \G$.
\begin{Remark}\label{rem4}
Let us verify that if $0<\alpha<Q$, then $P_{\alpha} \in L^1_{loc}(\G)$. In fact, by combining the coarea formula of Corollary \ref{cor-coarea}, the homogeneity of the measure $\mathcal{S}_{\infty}^{Q-1}$ and the homogeneity of the norm $\| \cdot \|_{\alpha}$, for every $R>0$ we get 
\begin{align*}
\frac{1}{\Gamma(\frac{\alpha}{2}) }\int_{\B_c(0,R)} \|x \|_{\alpha}^{\alpha-Q} dx &= \frac{1}{\Gamma(\frac{\alpha}{2}) }\int_0^R \int_{\partial \B_c(0,r)} \|x \|_{\alpha}^{\alpha-Q} d\mathcal{S}^{Q-1}_{\infty}(x) dt\\
&= \frac{1}{\Gamma(\frac{\alpha}{2}) }\int_0^R \int_{\partial \B_c(0,1)} r^{\alpha-Q} \|z \|_{\alpha}^{\alpha-Q} r^{Q-1} d\mathcal{S}^{Q-1}_{\infty}(z) dt \\
&\leq \frac{1}{\Gamma(\frac{\alpha}{2}) } \int_{0}^{R} r^{\alpha-1} dr \  \mathcal{S}_{\infty}^{Q-1}( \partial \B_c(0,1)) \max_{\partial \B_c(0,1)} \|z\|_{\alpha}^{\alpha-Q},  
\end{align*}
which is finite since $\alpha>0$. 
\end{Remark}
Taking into account the previous Items \ref{item2}, \ref{item-hom} and Remark \ref{rem4}, it holds that if $0<\alpha<Q$, then $P_{\alpha}$ is a kernel of type $\alpha$, according to the definition in \cite{Fol75}. We list below various results known in the literature that relate the maps $P_{\alpha}$ with $\mathcal{L}$ and $\mathcal{L}^s$, for $0<s<1$. We give the details of their own proofs only in those cases that are useful for the sequel of the paper.
\begin{itemize}
\item[(i)] The map $P_2$ is the fundamental solution of $\mathcal{L}$. Indeed for $x \in \G$ $$\mathcal{L} P_2(x)= \int_0^{\infty}  \mathcal{L}h(t,x) dt= -\int_0^{\infty} \frac{\partial}{\partial t}h(t,x) dt =- \lim_{t \to \infty} h(t,x)+h(0,x)=-\lim_{t \to \infty} h(t,x)=0, $$
where we exploited Theorem \ref{T1}(iii),(vi) and the known estimates about the derivatives of the heat kernel (see for example the proof of \cite[Proposition 1.75]{FS82})
\item[(ii)] By \cite[Proposition 4.1]{GT1}, for every $s \in (0,1)$ the kernel $P_{2s}$ is the fundamental solution of $\mathcal{L}^s$.
\item[(iii)]  By \cite[Theorem 3.15(iii), Proposition 3.18]{Fol75} and \cite[Lemma 8.5]{Gar18} we can deduce that for every $s \in (0,1)$ and $u \in C^{\infty}_0(\G)$ the equality $$ \mathcal{L}^su= \mathcal{L} u \star P_{2-2s}$$
holds.
\item[(iv)]\label{quattro} For every $s \in (0,1)$ the following relation holds outside the origin
$$ \mathcal{L} P_{2-2s}=P_{-2s}.$$
In fact, let $x \in \G$, $x \neq 0$
\begin{align*}
\mathcal{L}P_{2-2s}(x)&= \mathcal{L} \left( \frac{1}{\Gamma \left( 1-s \right)} \int_0^{\infty} t^{\frac{2-2s}{2}-1} h(t,x) dt \right)\\
& =  \frac{1}{-s \Gamma(-s)} \int_0^{\infty} t^{-s} \mathcal{L} h(t,x) dt \\
& = - \left( \frac{1}{-s \Gamma(-s)} \int_0^{\infty} t^{-s} \frac{\partial}{\partial t} h(t,x) dt \right)\\
&=  \frac{1}{-s \Gamma(-s)} \int_0^{\infty}(-s) t^{-s-1}  h(t,x) dt \\
&=   \frac{1}{ \Gamma(-s)} \int_0^{\infty} t^{-s-1}  h(t,x) dt =P_{-2s}(x),
\end{align*}
where we have considered as well the known estimates about the derivatives of the heat kernel (see \cite[Proposition 1.75]{FS82}) and the integration by parts.
\item[(v)] Reading \cite[Theorem 3.11]{FF15} in light of item (iv), for every $s \in (0,1)$ if $u \in \mathcal{S}(\G)$, then $\mathcal{L}^su \in L^2(\G)$ and 
\begin{equation}\label{3.11} \mathcal{L}^su(x)= P.V. \int_{\G} (u(y)-u(x)) P_{-2s}(x^{-1}y) dy.
\end{equation}
We have taken into account the equivalence proved in \cite[Lemma 8.5]{Gar18}.

\item[(vi)]For every $ 0< \alpha, \beta<Q$ such that $\alpha+\beta<Q$ and for every $x \in \G$, $x \neq 0$ the following relation holds
\begin{equation}
\label{convolution}
P_{\alpha+ \beta}(x)= P_{\alpha} \star P_{\beta}(x).
\end{equation}
Let us show \eqref{convolution}:
\begin{align*}
P_{\alpha}\star P_{\beta}(x)&= \frac{1}{\Gamma(\frac{\alpha}{2})\Gamma(\frac{\beta}{2})} \int_{\G} \int_0^{\infty} t^{\frac{\alpha}{2}-1}h(t, y)dt \int_0^{\infty} \tau^{\frac{\beta}{2}-1} h(\tau, y^{-1}x) d \tau  dy\\
&= \frac{1}{\Gamma(\frac{\alpha}{2})\Gamma(\frac{\beta}{2})} \int_{\G} \int_{(0, \infty)\times(0, \infty)} t^{\frac{\alpha}{2}-1}\tau^{\frac{\beta}{2}-1} h(t,y) h(\tau, y^{-1}x) dt d \tau  dy\\
&=\frac{1}{\Gamma(\frac{\alpha}{2})\Gamma(\frac{\beta}{2})}  \int_{(0, \infty)\times(0, \infty)} t^{\frac{\alpha}{2}-1}\tau^{\frac{\beta}{2}-1} \int_{\G} h(t, y) h(\tau, y^{-1}x) dy dt d \tau  \\
&= \frac{1}{\Gamma(\frac{\alpha}{2})\Gamma(\frac{\beta}{2})}  \int_{(0, \infty)\times(0, \infty)} t^{\frac{\alpha}{2}-1}\tau^{\frac{\beta}{2}-1} h(t + \tau, x) dt d \tau ,
\end{align*}
where we used Fubini's theorem and Theorem \ref{T1}(iv). Then, we consider the change of variables $$ \begin{cases} t=s(1-u)\\
\tau=us, \end{cases}$$ getting
\begin{align*}
P_{\alpha} \star P_{\beta}(x)&=\frac{1}{\Gamma(\frac{\alpha}{2})\Gamma(\frac{\beta}{2})} \int_{(0, \infty) \times (0,1)} (s(1-u))^{\frac{\alpha}{2}-1}(us)^{\frac{\beta}{2}-1} h(s, x) s  \ dsdu\\
&= \frac{1}{\Gamma(\frac{\alpha}{2})\Gamma(\frac{\beta}{2})} \int_0^{\infty} s^{\frac{\alpha}{2}-1}s^{\frac{\beta}{2}-1} s h(s, x)  ds \int_0^1(1-u)^{\frac{\alpha}{2}-1} u^{\frac{\beta}{2}-1} du\\
&= \frac{1}{\Gamma(\frac{\alpha}{2})\Gamma(\frac{\beta}{2})} \int_0^{\infty} s^{\frac{\alpha+ \beta}{2}-1} h(s, x)  ds \ B\left( \frac{\alpha}{2}, \frac{\beta}{2}\right),
\end{align*}
where $B(\cdot, \cdot )$ denotes the Beta Euler function. Exploiting the relation between the Beta and the Gamma Euler functions, we have
\begin{align*}
P_{\alpha} \star P_{\beta}(x)&= \frac{1}{\Gamma(\frac{\alpha}{2})\Gamma(\frac{\beta}{2})} \int_0^{\infty} s^{\frac{\alpha+ \beta}{2}-1} h(s, x)  ds \frac{\Gamma(\frac{\alpha}{2}) \Gamma(\frac{\beta}{2})}{\Gamma(\frac{\alpha+\beta}{2})}\\
&= \frac{1}{\Gamma(\frac{\alpha+\beta}{2})}\int_0^{\infty} s^{\frac{\alpha+ \beta}{2}-1} h(s, x)  ds= P_{\alpha+\beta}(x),
\end{align*}
and \eqref{convolution} is proved.
\end{itemize}

\subsubsection{Regularity}
Here we deal with the regularity of the map $\alpha \to P_{\alpha}(x)$, for $ x \in \G \setminus \{0\}$.
\begin{prop}\label{analiti2}
For every $x \in \G, \ x \neq 0$, the map
\begin{align*}
(-\infty, Q) \ni \alpha \to \| x \|_{\alpha}^{\alpha-Q}&= \int_{0}^{\infty} t^{\frac{\alpha}{2}-1} h(t,x) dt 
\end{align*} is analytic.
\end{prop}
\begin{proof}
We fix $x \in \G$, $x \neq 0$, $t>0$ and $k \in \N$ and we compute the $k$-th derivative of the map $\alpha \to t^{\frac{\alpha}{2}-1} h(t,x)$ that is
$$ \frac{\partial^k}{\partial \alpha^k} \Big( t^{\frac{\alpha}{2}-1} h(t,x) \Big)= t^{\frac{\alpha}{2}-1} \left( \frac{1}{2} \right)^k \ln^k(t) h(t,x).$$
We notice that from the estimates in Item \ref{item1} of Section \ref{section-kernel} we have $$\int_0^{\infty}t^{\frac{\alpha}{2}-1} \left( \frac{1}{2} \right)^k \ln^k(t) h(t,x)dt < \infty,$$ for every $\alpha<Q$, then we can deduce that for every $\alpha<Q$ we have
$$ \frac{\partial^k}{\partial \alpha^k} \int_0^{\infty} t^{\frac{\alpha}{2}-1} h(t,x)dt=  \int_0^{\infty} t^{\frac{\alpha}{2}-1} \left( \frac{1}{2} \right)^k \ln^k(t) h(t,x)dt.$$
Now, we claim that for every closed interval $[a,b] \subset (-\infty, Q)$ there are two positive constants $C$ and $M$ such that
\begin{equation}
\label{claim}
\left| \frac{\partial^k}{\partial \alpha^k} \int_0^{\infty} t^{\frac{\alpha}{2}-1} h(t,x)dt  \right| \leq C M^{k+1} k!,
\end{equation}
for every $\alpha \in [a,b]$. 
This would conclude the proof. Hence, we need to prove \eqref{claim}. First, we consider
\begin{equation}
\label{stima1}
\left| \frac{\partial^k}{\partial \alpha^k} \int_0^{\infty} t^{\frac{\alpha}{2}-1} h(t,x)dt  \right| \leq \left| \int_0^{1} t^{\frac{\alpha}{2}-1} \left( \frac{1}{2} \right)^k \ln^k(t) h(t,x)dt \right| + \left| \int_1^{\infty} t^{\frac{\alpha}{2}-1} \left( \frac{1}{2} \right)^k \ln^k(t) h(t,x)dt \right|.
\end{equation}
By combining the homogeneity of $h$ and Theorem \ref{T1}(vii), which ensures that $h(1, \cdot ) \in \mathcal{S}(\G)$, we can estimate the second term of the right-hand side of \eqref{stima1} getting
\begin{align*}
\left|\int_1^{\infty} t^{\frac{\alpha}{2}-1} \left( \frac{1}{2} \right)^k \ln^k(t) h(t,x)dt \right| &=\int_1^{\infty} t^{\frac{\alpha}{2}-1} \left( \frac{1}{2} \right)^k \ln^k(t) h(t,x)dt\\ &= \int_1^{\infty} t^{\frac{\alpha}{2}-1-\frac{Q}{2}} \left( \frac{1}{2} \right)^k \ln^k(t) h(1,\delta_{\frac{1}{\sqrt{t}}}(x))dt\\
& \leq \|h(1, \cdot) \|_{L^{\infty}(\G)}\int_1^{\infty} t^{\frac{\alpha}{2}-1-\frac{Q}{2}} \left( \frac{1}{2} \right)^k \ln^k(t)dt.
\end{align*}
Now, for $k \in \N$ we set 
\begin{align}
\label{ak}
a_k:&= \int_1^{\infty} t^{\frac{\alpha-Q}{2}-1} \left( \frac{1}{2} \right)^k \ln^k(t) dt,
\end{align}
and we notice that
\begin{align*}
a_k&=\lim_{a \to \infty} \left[  \frac{2}{\alpha-Q} t^{\frac{\alpha-Q}{2}} \left( \frac{1}{2} \right)^k \ln^k(t) \right]_{t=1}^{a}+ \int_{1}^{\infty} \frac{2}{Q-\alpha} t^{\frac{\alpha}{2}-1-\frac{Q}{2}}  \left( \frac{1}{2} \right)^k k \ln^{k-1}(t)dt\\
&= \frac{1}{Q-\alpha} k a_{k-1},  \notag  
\end{align*}
where we have exploited the integration by parts.
Then, recursively we get that 
\begin{equation}
\label{ric1}
a_k=k \frac{1}{Q-\alpha} a_{k-1}= k(k-1)\frac{1}{(Q-\alpha)^2} a_{k-2} = \ldots = k! \frac{1}{(Q-\alpha)^{k}}a_0=k! \frac{2}{(Q-\alpha)^{k+1}}.
\end{equation}
Let us now focus on the first contribute of the right-hand side of \eqref{stima1}.
We exploit the fact that $h(t,x)=o(t^N)$ as $t \to 0$ for every $x \neq 0$ and for every $N \in \N$ (hence for every $N \in \N$, there is $c>0$ such that $h(t,x)\leq ct^N$ as $t<1$) and we can state that
$$\left| \int_0^{1} t^{\frac{\alpha}{2}-1} \left( \frac{1}{2} \right)^k \ln^k(t) h(t,x)dt \right| \leq c \left| \int_0^{1} t^{\frac{\alpha}{2}-1+N} \left( \frac{1}{2} \right)^k \ln^k(t) dt  \right| \leq c \int_0^{1} t^{\frac{\alpha}{2}-1+N} \left( \frac{1}{2} \right)^k |\ln(t)|^k dt. $$
Let us now set for $k \in \N$ and $N \in \N$ the term
\begin{align}
\label{bk}
b_k:&= \int_0^{1} t^{\frac{\alpha}{2}-1+N} \left( \frac{1}{2} \right)^k |\ln(t)|^k dt.
\end{align}
We have that
\begin{align*}
b_k&=\lim_{a \to 0} \left[ t^{\frac{\alpha}{2}+N} \frac{2}{\alpha+2N} \left( \frac{1}{2} \right)^k |\ln^k (t)| \right]_{t=a}^1 - \int_0^1 \frac{1}{\alpha+2N}  t^{\frac{\alpha}{2}-1+N} \left( \frac{1}{2} \right)^{k-1} k |\ln (t)|^{k-1} \sgn(\ln(t))dt. \notag
\end{align*}
If $N>-\frac{\alpha}{2}$, we can go on getting 
\begin{equation}
\label{ric2}
\begin{split}
b_k &=  \int_0^1 \frac{1}{\alpha+2N}  t^{\frac{\alpha}{2}-1+N} \left( \frac{1}{2} \right)^{k-1} k |\ln (t)|^{k-1} dt\\
&=\frac{1}{\alpha+2N} k b_{k-1}= \frac{1}{(\alpha+2N)^2}k(k-1)b_{k-2}= \ldots = \frac{1}{(\alpha+2N)^k} k! b_0=\frac{2}{(\alpha+2N)^{k+1}} k!. \\
\end{split}
\end{equation}
By combining \eqref{stima1}, \eqref{ric1} and \eqref{ric2}, we get that
\begin{equation}
\label{stima-ana}
\begin{split}
\left| \frac{\partial^k}{\partial \alpha^k} \int_0^{\infty} t^{\frac{\alpha}{2}-1} h(t,x)dt  \right| \leq  k! \| h(1, \cdot)\|_{L^{\infty}(\G)}\frac{2}{(Q-\alpha)^{k+1}} +  \frac{2c}{(\alpha+2N)^{k+1}} k!\\
=2(1+c)(1+\|h(1, \cdot)\|_{L^{\infty}(\G)}) k! \max \left\{ \frac{1}{Q-\alpha}, \frac{1}{\alpha+2N} \right\}^{k+1}.
\end{split}
\end{equation}
Since $N$ can be chosen arbitrarily large (once fixed $\alpha$), our claim \eqref{claim} is proved and this completes the proof.
\end{proof}
As a consequence of Proposition \ref{analiti2}, the map $\alpha \to P_{\alpha}(x)$ is analytic for $\alpha<Q$, for every $x \in \G, \ x \neq 0$, since it is a product of analytic functions.\\
Now we are in position to introduce the following map
\begin{equation}\label{sigmaalfa}
(-\infty,Q) \ni \alpha \to \sigma(\alpha):= \int_{\partial \B_c(0,1)} \| z \|_{\alpha}^{\alpha-Q} d \mathcal{S}^{Q-1}_{\infty}(z).
\end{equation}

\subsection{Main properties of the map $\sigma(\alpha)$}
\label{sectioncalfa}
In this subsection we collect and prove some properties of the map $\sigma$ defined in \eqref{sigmaalfa}.
\begin{enumerate}
\item The map $\sigma(\alpha)$ is well defined for every $\alpha<Q$. In fact, if we exploit the estimates in Item \ref{item1} of Section \ref{section-kernel}, we have
\begin{align*}
\sigma(\alpha)&= \int_{\partial \B_c(0,1)} \int_0^{\infty} t^{\frac{\alpha}{2}-1} h(t,x) dt d \mathcal{S}^{Q-1}_{\infty}(x)\\
&=  \int_{\partial \B_c(0,1)} \int_0^1  t^{\frac{\alpha}{2}-1} h(t,x) dt d \mathcal{S}^{Q-1}_{\infty}(x) +   \int_{\partial \B_c(0,1)} \int_1^{\infty} t^{\frac{\alpha}{2}-1} h(t,x) dt d \mathcal{S}^{Q-1}_{\infty}(x)  \\
& \leq c_1  \int_{\partial \B_c(0,1)} \int_0^1 t^{\frac{\alpha}{2}-1} t^Nd \mathcal{S}^{Q-1}_{\infty}(x) dt + c_2 \int_{\partial \B_c(0,1)} \int_1^{\infty} t^{\frac{\alpha}{2}-1-\frac{Q}{2}}  dt d \mathcal{S}^{Q-1}_{\infty}(x) \\
&= \mathcal{S}^{Q-1}_{\infty}(\partial \B_c(0,1)) \left( c_1\int_0^1   t^{\frac{\alpha}{2}-1+N} dt + c_2 \int_1^{\infty} t^{\frac{\alpha}{2}-1-\frac{Q}{2}} dt\right),
\end{align*}
for some positive constants $c_1$ and $c_2$ for every arbitrarily fixed $N \in \N$. Therefore, $\sigma(\alpha)$ converges if $\alpha <Q$. 
\item \label{itemcalfa} We show that $\alpha \to \sigma(\alpha)$ is analytic if $\alpha<Q$. For every $k \in \N$ we have
\begin{equation}
\label{der-c-alfa}
\begin{split}
\left| \frac{\partial}{\partial \alpha^k} \sigma(\alpha) \right|&=\left| \frac{\partial^k}{\partial \alpha^k} \int_{\partial \B_c(0,1)} \int_0^{\infty} t^{\frac{\alpha}{2}-1} h(t,x) dt d\mathcal{S}^{Q-1}_{\infty}(x)\right| \\
&=\left| \int_{\partial \B_c(0,1)} \int_0^{\infty}\frac{\partial^k}{\partial \alpha^k} \left( t^{\frac{\alpha}{2}-1} h(t,x) \right) dt d\mathcal{S}^{Q-1}_{\infty}(x)\right|\\
&=\left| \int_{\partial \B_c(0,1)} \int_0^{\infty} t^{\frac{\alpha}{2}-1} \left( \frac{1}{2}\right)^k \ln^k(t) h(t,x)  dt d\mathcal{S}^{Q-1}_{\infty}(x)\right|,
\end{split}
\end{equation}
since for every $\alpha<Q$ and $k=0, 1, \dots$ we have that
\begin{align}\label{servez}
&\int_{\partial \B_c(0,1)} \int_0^{\infty} t^{\frac{\alpha}{2}-1} \left( \frac{1}{2}\right)^k \ln^k(t) h(t,x)  dt d\mathcal{S}^{Q-1}_{\infty}(x) \leq \notag \\ 
& \mathcal{S}^{Q-1}_{\infty}(\partial \B_c(0,1)) \int_0^{\infty} t^{\frac{\alpha}{2}-1} \left( \frac{1}{2}\right)^k \ln^k(t) h(t,x_M)  dt < \infty,
\end{align}
where $x_M \in \partial \B_c(0,1)$ is such that $h(t,x_M)=\max_{\partial \B_c(0,1)} \left( \int_0^{\infty} t^{\frac{\alpha}{2}-1} \left( \frac{1}{2}\right)^k \ln^k(t) h(t,x)  dt\right)$. By combining \eqref{der-c-alfa}, \eqref{servez} and \eqref{stima-ana}, we get that
\begin{align*}
\left| \frac{\partial^k}{\partial \alpha^k} \sigma(\alpha) \right| &\leq \mathcal{S}^{Q-1}_{\infty}(\partial \B_c(0,1)) \left| \frac{\partial^k}{\partial \alpha^k} \int_0^{\infty} t^{\frac{\alpha}{2}-1} h(t,x_M)dt  \right| \\
&= \mathcal{S}^{Q-1}_{\infty}(\partial \B_c(0,1))  2(1+c)(1+\|h(1, \cdot)\|_{L^{\infty}(\G)}) k! \max \left\{ \frac{1}{Q-\alpha}, \frac{1}{\alpha+2N} \right\}^{k+1},
\end{align*}
so that $\sigma(\alpha)$ is analytic with respect to $\alpha$ for $\alpha<Q$.
\item \label{fa2!} Let us prove that 
$$\lim_{\alpha \to 0} \sigma(\alpha)=2,$$
hence that
$$ \lim_{\alpha \to 0} \int_{\partial \B_c(0,1)} \| z \|_{\alpha}^{\alpha-Q} d \mathcal{S}^{Q-1}_{\infty}(z) =2.$$
Let us start by considering Theorem \ref{T1}(iii). By exploiting the coarea formula and the homogeneity of $\mathcal{S}^{Q-1}_{\infty}$, for every $t>0$ we get
\begin{equation}
\label{fa1}
\begin{split}
1=\int_\G h(t,x) dx&= \int_0^{\infty} \int_{\partial \B_c(0,r)} h(t,x) d \mathcal{S}^{Q-1}_{\infty}(x) dt\\
&=\int_0^{\infty}r^{Q-1} \int_{\partial \B_c(0,1)} h(t, \delta_r(z)) d \mathcal{S}^{Q-1}_{\infty}(z) dr.
\end{split}
\end{equation}
Now, let us consider 
\begin{align*}
\sigma(\alpha) &= \int_{\partial \B_c(0,1) } \| z \|_{\alpha}^{\alpha-Q} d \mathcal{S}^{Q-1}_{\infty}(z)\\
&= \int_{\partial \B_c(0,1)} \int_0^{\infty} t^{\frac{\alpha}{2}-1}h(t,z) dt d \mathcal{S}^{Q-1}_{\infty}(z)\\
&= \int_0^{\infty}t^{\frac{\alpha}{2}-1} \int_{\partial \B_c(0,1)} h(t,z) d \mathcal{S}^{Q-1}_{\infty}(z) dt\\
&= \int_0^{\infty}t^{\frac{\alpha}{2}-1} \int_{\partial \B_c(0,1)} (\sqrt{t})^{-Q} h(1,\delta_{\frac{1}{\sqrt{t}}}(z)) d \mathcal{S}^{Q-1}_{\infty}(z) dt\\
&= \int_0^{\infty}t^{\frac{\alpha}{2}-1- \frac{Q}{2}} \int_{\partial \B_c(0,1)} h(1,\delta_{\frac{1}{\sqrt{t}}}(z)) d \mathcal{S}^{Q-1}_{\infty}(z) dt,
\end{align*}
where we used the homogeneity of $h$. Let us change variable setting $r= \frac{1}{\sqrt{t}}$, hence $t=r^{-2}$ ($dt=-2r^{-3}$) and we get
\begin{align*}
\sigma(\alpha) &= -2 \int_{\infty}^{0}r^{-\alpha+2+Q} \int_{\partial \B_c(0,1)} h(1,\delta_{r}(z)) d \mathcal{S}^{Q-1}_{\infty}(z) r^{-3} dr\\
&= 2 \int_{0}^{\infty}r^{-\alpha-1+Q} \int_{\partial \B_c(0,1)} h(1,\delta_{r}(z)) d \mathcal{S}^{Q-1}_{\infty}(z) dr,
\end{align*}
which converges to 
\begin{align*}
 2 \int_{0}^{\infty}r^{Q-1} \int_{\partial \B_c(0,1)} h(1,\delta_{r}(z)) d \mathcal{S}^{Q-1}_{\infty}(z) dr,
\end{align*}
as $\alpha \to 0$. Hence the proof of Item \ref{fa2!} is now complete taking into account \eqref{fa1}.

\end{enumerate}

\section{Analytic continuation in $\G$}
\label{sec-ext}
Let $\phi \in \mathcal{S}(\G)$ and $x \in \G$.
For every $r>0$ and $0<\alpha<Q$ we denote the closed ball $$\B_{\alpha}(x,r):= \{ y \in \G : \| x^{-1} y \|_{\alpha} \leq r \}.$$  
For every $ \alpha>0 $ we introduce the map
\begin{equation}\label{psi}\psi(x,\alpha)= \int_{\G} \phi(xy)P_{\alpha}(y) dy= \frac{1}{\Gamma(\frac{\alpha}{2})} \int_\G \phi(xy) \|y \|_{\alpha}^{\alpha-Q} dy.
\end{equation}

\begin{prop}\label{prop31}
For every $x \in \G$ the map $(0, Q) \ni \alpha \to \psi(x, \alpha)$ is well defined.
\end{prop}
\begin{proof}
By exploiting the estimate in \eqref{relcc-alfa}, the fact that $\phi$ is in the Schwartz class $ \mathcal{S}(\G)$ on $\G$ and the coarea formula of Corollary \ref{cor-coarea}, we have
\begin{equation}
\label{converge}
\begin{split}
 \int_{\G} \phi(xy) P_{\alpha}(y) dy & \leq  \int_{\G} |\phi(xy)| |P_{\alpha}(y) |dy = \frac{1}{  \Gamma(\frac{\alpha}{2})} \int_{\G} |\phi(z)| \| x^{-1} z \|_{\alpha}^{\alpha-Q} dz \\
&= \frac{1}{ \Gamma(\frac{\alpha}{2})} \int_{\B_{\alpha}(x,1)} |\phi(z)| \| x^{-1} z \|_{\alpha}^{\alpha-Q} dz  +  \frac{1}{ \Gamma(\frac{\alpha}{2})} \int_{\G \setminus \B_{\alpha}(x,1)} |\phi(z)| \| x^{-1} z \|_{\alpha}^{\alpha-Q} dz\\
& \leq  \frac{1}{ \Gamma(\frac{\alpha}{2})} \|\phi \|_{L^{\infty}(\G)} \int_{\B_c(x,c_{\alpha})} \| x^{-1} z \|_{\alpha}^{\alpha-Q} dz  +  \frac{1}{ \Gamma(\frac{\alpha}{2})} \int_{\G \setminus \B_{\alpha}(x,1)}|\phi(z)| \frac{(1+\| z \|_c^k)}{(1+\| z \|_c^k)}  dz,
\end{split}
\end{equation}
thus, taking into account the homogeneity of the norm $\| \cdot \|_{\alpha}$ and of the measure $\mathcal{S}^{Q-1}_{\infty}$, and the definition of $\sigma(\alpha)$ we can continue from \eqref{converge} getting
\begin{align*}
 \int_{\G} \phi(xy) P_{\alpha}(y) dy &\leq  \frac{1}{\Gamma(\frac{\alpha}{2})} \sigma ( \alpha) \|\phi\|_{L^{\infty}(\G)} \int_0^{c_{\alpha}} r^{\alpha-1} dr   +  \frac{1}{\Gamma(\frac{\alpha}{2})} c_k \int_{\G } \frac{1}{(1+\| z \|_c^k)} dz < \infty \\
&\leq  \frac{1}{\Gamma(\frac{\alpha}{2})} \frac{c_{\alpha}}{\alpha} \sigma ( \alpha) \|\phi  \|_{L^{\infty}(\G)}+  \frac{1}{\Gamma(\frac{\alpha}{2})} c_k \int_{\G } \frac{1}{(1+\| z \|_c^k)} dz,
\end{align*}
where, choosing $k>Q$, the last integral converges. 
\end{proof}

The map $\alpha \to \psi(x,\alpha)$ is well defined also if we weaken the hypothesis on $\phi$. 
\begin{prop}
\label{prop-psi}
If $\phi \in C^{\infty}(\G)$ and $|\phi(z)| \leq K\|z\|_c^{-L}$ as $\|z\|_c \geq S$ for some $L, K, S>0$ then the map $$(0,\min \{L,Q \}) \ni \alpha \to \psi(x, \alpha)$$ is well defined.
\end{prop}
\begin{proof}
Keeping in mind that $x$ is fixed, we consider $R(x):= \max \{1+ \|x\|_c, S \}$ and we set $\widetilde{R}(x)>0$ such that $\B_c(0,R(x)) \subset \B_c(x,\widetilde{R}(x))$. For $\alpha>0$, by exploiting the hypotheses on $\phi$ and the coarea formula in Corollary \ref{cor-coarea}, we get \begin{equation}
\label{continue4}
\begin{split}
 \int_{\G} \phi(xy) P_{\alpha}(y) dy &\leq  \int_{\G}| \phi(xy) P_{\alpha}(y)| dy= \frac{1}{\Gamma(\frac{\alpha}{2})} \int_{\G}| \phi(z)| \| x^{-1} z \|_{\alpha}^{\alpha-Q} dz \\
= \frac{1}{\Gamma(\frac{\alpha}{2})} &\int_{\B_{c}(0,R(x))}| \phi(z) |\| x^{-1} z \|_{\alpha}^{\alpha-Q} dz  +  \frac{1}{\Gamma(\frac{\alpha}{2})} \int_{\G \setminus \B_c(0,R(x))} |\phi(z)| \| x^{-1} z \|_{\alpha}^{\alpha-Q} dz\\
 \leq  \frac{1}{\Gamma(\frac{\alpha}{2})} &  \|\phi\|_{L^{\infty}(\G)} \int_{\B_c(0,R(x))} \| x^{-1} z \|_{\alpha}^{\alpha-Q} dz  +  K\frac{1}{\Gamma(\frac{\alpha}{2})} \int_{\G \setminus \B_c(0,R(x))}\frac{1}{\|z\|_c^{L}}  \| x^{-1} z \|_{\alpha}^{\alpha-Q}  dz.
\end{split}
\end{equation}
Thus, taking into account the homogeneity of the norm $\| \cdot \|_{\alpha}$ and of the measure $\mathcal{S}^{Q-1}_{\infty}$, and the definition of $\sigma(\alpha)$, we can continue from \eqref{converge} getting
\begin{align}\label{serve2}
 \int_{\G} \phi(xy) P_{\alpha}(y) dy &\leq  \frac{1}{\Gamma(\frac{\alpha}{2})} \sigma(\alpha)  \|\phi\|_{L^{\infty}(\G)} \int_0^{\widetilde{R}(x)} r^{\alpha-1} dr   + K \frac{1}{ \Gamma(\frac{\alpha}{2})}  \int_{\G \setminus \B_c(0,R(x))}\frac{1}{\|z\|_c^{L+Q-\alpha}}  \frac{\|z\|_c^{Q-\alpha}}{\| x^{-1} z \|_{\alpha}^{Q-\alpha}}   dz  \notag \\
&\leq  \frac{1}{\Gamma(\frac{\alpha}{2})} \sigma(\alpha)  \|\phi\|_{L^{\infty}(\G)} \int_0^{\widetilde{R}(x)} r^{\alpha-1} dr   +  K\frac{1}{\Gamma(\frac{\alpha}{2})} (c_{\alpha}(\|x\|_c+1))^{Q-\alpha}\int_{\G \setminus \B_c(0,R(x))}\frac{1}{\|z\|_c^{L+Q-\alpha}}  dz, 
\end{align}
since for every $z \notin \B_c(0,R(x))$, combining the definition of $R(x)$, the estimate \eqref{relcc-alfa} and some properties of homogeneous norms we have
\begin{align}\label{serve3}
& \frac{\| x^{-1}z \|_{\alpha}}{\| z \|_c} \geq \frac{1}{c_{\alpha}} \frac{\| x^{-1}z \|_c}{\| z \|_c} = \frac{1}{c_{\alpha}}  \| \delta_{\frac{1}{\| z \|_c}}(x)^{-1}\delta_{\frac{1}{\| z \|_c}}(z) \|_c \notag \\
& \geq \frac{1}{c_{\alpha}} | \| \delta_{\frac{1}{\| z \|_c}}(x)\|_c - \| \delta_{\frac{1}{\| z \|_c}}(z)  \|_c |\\
& \geq \frac{1}{c_{\alpha}}\bigg(1-\frac{\|x\|_c}{\|x\|_c+1} \bigg) =\frac{1}{c_{\alpha}}\frac{1}{\|x\|_c+1}. \notag
\end{align}
Therefore, we can continue from \eqref{serve2} getting 
\begin{align*}
 \int_{\G} \phi(xy) P_{\alpha}(y) dy  \leq \  & \frac{1}{\Gamma(\frac{\alpha}{2})}  \sigma(\alpha)  \|\phi\|_{L^{\infty}(\G)} \int_0^{\widetilde{R}(x)} r^{\alpha-1} dr\\
&   + K \frac{1}{\Gamma(\frac{\alpha}{2})} (c_{\alpha}(\|x\|_c+1))^{Q-\alpha}\mathcal{S}^{Q-1}_{\infty}(\partial \B_c(0,1)) \int_{R(x)}^{\infty} r^{\alpha-L-1} dr. 
\end{align*}
The last term of previous equation converges for every $\alpha>0$ such that $L+1-\alpha>1$, hence when $\alpha<L$. This concludes the proof.
\end{proof}

From now on, in the current section we show that the map $\alpha \to \psi(x, \alpha)$ is analytic for $0 < \alpha <Q$ and we construct an analytic continuation of $\alpha \to \psi(x, \alpha)$ on $-2<\alpha\leq 0 $. This construction allows to introduce a consistent suitable definition of negative powers of the sub-Laplacian $\mathcal{L}$. We mainly follow the structure of \cite[Chapter 1]{Landkof}, where an analogous construction is carried out in Euclidean spaces.
\subsection{The map $\alpha \to \psi(x, \alpha)$ is analytic on $(0,Q)$}
In this subsection, we show that the map $\alpha \to \psi(x, \alpha)$ is analytic on $(0,Q)$. First, we consider the following remark.
\begin{Remark}
\label{uniform-ogrande}
For every $x \neq 0$, we know that $h(t,x)=o(t^N)$ for every $N \in \N$. This property is uniform for $x \in \G \setminus \B(0,r)$, for every $r>0$. In fact, by Theorem \ref{T1}(vi), there exists a constant $c>0$ such that for every $x \in \G$ and $t>0$
$$ c^{-1} t^{-Q/2} \exp\left(-\frac{ \| x \|^2}{c^{-1}t}\right) \leq h(t,x) \leq c t^{-Q/2} \exp\left(-\frac{\| x \|^2}{ct}\right),$$
hence for every $y \in \G \setminus \B_c(0,r)$ and $t>0$ the estimate 
$$h(t,y) \leq c t^{-Q/2} e^{-\frac{r}{ct}}$$
holds, so that for every $N \in \N$
$$t^{-N} h(t,y) \leq c t^{-Q/2-N} e^{-\frac{r}{ct}} \to 0$$
as $t \to 0$ independently of $y$. We are ready to treat the regularity of $\psi(x,\alpha)$.
\end{Remark}
\begin{prop}
\label{prop-schw}
For every $\phi \in \mathcal{S}(\G)$ and $x \in \G$ the function
$$ \alpha \to  \psi(x, \alpha)=\int_{\G} \phi(xy) P_{\alpha}(y) dy$$ is analytic on $(0,Q).$
\end{prop}
\begin{proof}
Since the map $\alpha \to \frac{1}{\Gamma(\frac{\alpha}{2})}$ is analytic on $\R$ and analytic functions well behave with respect to the product, it is enough to prove that
\begin{equation*}
\alpha \to \int_{\G} \phi(xy) \| y \|_{\alpha}^{\alpha-Q}dy= \int_{\G} \phi(xy) \int_{0}^{\infty} t^{\frac{\alpha}{2}-1} h(t,y)dt dy
\end{equation*}
is analytic on $(0,Q).$ Let us consider that for every $k \in \N$ for every $t>0$ and $x \in \G$
$$ \frac{\partial^k}{\partial \alpha^k} t^{\frac{\alpha}{2}-1} h(t,x)= t^{\frac{\alpha}{2}-1} \left( \frac{1}{2} \right)^k \ln^k(t) h(t,x).$$
We obtain the needed estimate on the $k$-th derivative of the map $\psi(x, \alpha)$, showing at the same time that it is summable on $ (0, \infty) \times \G $ 
\begin{align}
\label{zeropezzo}
\left| \int_{\G} \phi(xy) \int_{0}^{\infty} t^{\frac{\alpha}{2}-1} \left( \frac{1}{2} \right)^k \ln^k(t) h(t,y) dt dy \right| & \leq  \int_{\G} |\phi(xy)| \int_{0}^{\infty} t^{\frac{\alpha}{2}-1} \left( \frac{1}{2} \right)^k |\ln(t)|^k h(t,y) dt dy \notag \\
=  \int_{\G \setminus \B_c(0,1)} |\phi(xy)| & \int_{0}^{\infty} t^{\frac{\alpha}{2}-1} \left( \frac{1}{2} \right)^k |\ln(t)|^k h(t,y) dt dy \\
+   \int_{\B_c(0,1)} |\phi(xy)| &  \int_{0}^{\infty} t^{\frac{\alpha}{2}-1} \left( \frac{1}{2} \right)^k |\ln(t)|^k h(t,y) dt dy . \notag
\end{align}
Let us split the second term of \eqref{zeropezzo} as follows
\begin{align}
\label{zeropezzo1}
 &\int_{\B_c(0,1)} |\phi(xy)| \int_{0}^{\infty} t^{\frac{\alpha}{2}-1} \left( \frac{1}{2} \right)^k |\ln(t)|^k h(t,y) dt dy \notag \\
=&  \int_{\B_c(0,1)} |\phi(xy)| \int_{0}^{1} t^{\frac{\alpha}{2}-1} \left( \frac{1}{2} \right)^k |\ln(t)|^k h(t,y) dt dy \\
&+  \int_{\B_c(0,1)} |\phi(xy)| \int_{1}^{\infty} t^{\frac{\alpha}{2}-1} \left( \frac{1}{2} \right)^k \ln^k(t) h(t,y)dt dy. \notag
\end{align}
Let us then consider the first term of the sum in \eqref{zeropezzo1}, getting the estimate
\begin{align}
\label{primopezzo}
 & \int_{\B_c(0,1)} |\phi(xy)| \int_{0}^{1} t^{\frac{\alpha}{2}-1} \left( \frac{1}{2} \right)^k |\ln(t)|^k h(t,y) dy  \notag \\
  \leq & \sup_{\B_c(0,1)} | \phi(x \cdot) | \int_0^1  t^{\frac{\alpha}{2}-1} \left( \frac{1}{2} \right)^k |\ln(t)|^k \int_{\B(0,1)} h(t,y) dy dt  \notag \\
    \leq & \sup_{\B_c(0,1)} | \phi(x \cdot) | \int_0^1  t^{\frac{\alpha}{2}-1} \left( \frac{1}{2} \right)^k |\ln(t)|^k \int_{\G} h(t,y) dy dt\\
\leq & \sup_{\B_c(0,1)}| \phi(x \cdot) | \int_0^1  t^{\frac{\alpha}{2}-1} \left( \frac{1}{2} \right)^k |\ln(t)|^k dt.  \notag
\end{align}
Now, consider that $\int_0^1  t^{\frac{\alpha}{2}-1} \left( \frac{1}{2} \right)^k |\ln(t)|^k dt$ equals the $b_k$ defined in \eqref{bk} for $N=0>- \frac{\alpha}{2}$, hence we can continue \eqref{primopezzo} exploiting the estimate in \eqref{ric2} and we get that
\begin{equation}
\label{first}
\begin{split}
\int_{\B_c(0,1)} |\phi(xy)| \int_{0}^{1} t^{\frac{\alpha}{2}-1} \left( \frac{1}{2} \right)^k |\ln(t)|^k h(t,y) dy \leq \sup_{\B_c(0,1)} | \phi(x \cdot) |  \ \frac{2}{\alpha^{k+1}} k!.
\end{split}
\end{equation}
Let us bound the second term of the sum in \eqref{zeropezzo1}
\begin{align}
\label{secondopezzo}
&\int_{\B_c(0,1)} |\phi(xy)| \int_{1}^{\infty} t^{\frac{\alpha}{2}-1} \left( \frac{1}{2} \right)^k \ln^k(t) h(t,y)dt dy \notag \\
& = \int_{\B_c(0,1)} |\phi(xy)| \int_{1}^{\infty} t^{\frac{\alpha}{2}-1-\frac{Q}{2}} \left( \frac{1}{2} \right)^k \ln^k(t) h(1,\delta_{\frac{1}{\sqrt{t}}}(y))dt dy  \\
& \leq \sup_{\B_c(0,1)}| \phi(x \cdot) | \sup_{\B_c(0,1)} h(1, \cdot) \mu(\B_c(0,1))  \int_{1}^{\infty} t^{\frac{\alpha}{2}-1-\frac{Q}{2}} \left( \frac{1}{2} \right)^k \ln^k(t) dt. \notag
\end{align}
Now, consider that $\int_1^\infty  t^{\frac{\alpha}{2}-1-\frac{Q}{2}} \left( \frac{1}{2} \right)^k \ln^k(t) dt$ equals the $a_k$ introduced in \eqref{ak}, hence we can continue from \eqref{secondopezzo} exploiting the estimate in \eqref{ric1} and we get that
\begin{equation}
\label{second}
\begin{split}
&\int_{\B_c(0,1)} |\phi(xy)| \int_{1}^{\infty} t^{\frac{\alpha}{2}-1} \left( \frac{1}{2} \right)^k \ln^k(t) h(t,y)dt dy \\
&\leq \sup_{\B_c(0,1)}| \phi(x \cdot) | \sup_{\B_c(0,1)} h(1, \cdot) \mu(\B_c(0,1)) \frac{2}{(Q-\alpha)^{k+1}}   k!.
\end{split}
\end{equation}
Let us split the first term of the sum in \eqref{zeropezzo} as follows
\begin{align}
\label{zeropezzo2}
&\int_{\G \setminus \B_c(0,1)} |\phi(xy)| \int_{0}^{\infty} t^{\frac{\alpha}{2}-1} \left( \frac{1}{2} \right)^k |\ln(t)|^k h(t,y) dt dy  \notag \\
=&\int_{\G \setminus \B_c(0,1)} |\phi(xy)| \int_{0}^{1} t^{\frac{\alpha}{2}-1} \left( \frac{1}{2} \right)^k |\ln(t)|^k h(t,y) dt dy \\
&+ \int_{\G \setminus \B_c(0,1)} |\phi(xy)| \int_{1}^{\infty} t^{\frac{\alpha}{2}-1} \left( \frac{1}{2} \right)^k \ln^k(t) h(t,y) dt dy.  \notag
\end{align}
Keeping in mind Remark \ref{uniform-ogrande}, we can estimate the first term of the sum in \eqref{zeropezzo2} as follows for $N \in \N$
\begin{align}
\label{terzopezzo}
&\int_{\G \setminus \B_c(0,1)} |\phi(xy)| \int_{0}^{1} t^{\frac{\alpha}{2}-1} \left( \frac{1}{2} \right)^k |\ln(t)|^k h(t,y) dt dy  \notag \\
\leq \ & c  \int_{0}^{1} t^{\frac{\alpha}{2}-1+N} \left( \frac{1}{2} \right)^k |\ln(t)|^k \int_{\G \setminus \B_c(0,1)} | \phi(xy)| dy dt  \notag \\
\leq \ & c \int_{0}^{1} t^{\frac{\alpha}{2}-1+N} \left( \frac{1}{2} \right)^k |\ln(t)|^k \int_{\G } |\phi(z)| \frac{(1+\|z\|_c^\ell)}{(1+\|z\|_c^\ell)} dz dt \\
\leq \ & c c_{\ell} \int_{\G} \frac{1}{1+\|z\|_c^\ell} dz  \int_{0}^{1} t^{\frac{\alpha}{2}-1+N} \left( \frac{1}{2} \right)^k |\ln(t)|^k dt   \notag \\
\leq \ & c \widetilde{c}_{\ell} \int_{0}^{1} t^{\frac{\alpha}{2}-1+N} \left( \frac{1}{2} \right)^k |\ln(t)|^k dt,  \notag 
\end{align}
where $c$, $c_{\ell}$ and $\widetilde{c}_{\ell}$ are suitable positive constants and $\ell$ is understood such that $\ell>Q$.
Analogously to \eqref{primopezzo}, when $N> -\frac{\alpha}{2}$ we can bound \eqref{terzopezzo} as follows 
\begin{equation}
\label{third}
\begin{split}
&\int_{\G \setminus \B_c(0,1)} |\phi(xy)| \int_{0}^{1} t^{\frac{\alpha}{2}-1+N} \left( \frac{1}{2} \right)^k |\ln(t)|^k h(t,y) dt dy \leq \tilde{c}_\ell  \frac{2}{(\alpha+2N)^{k+1}} k!.
\end{split}
\end{equation}
Let us finally consider the second term of the sum in \eqref{zeropezzo2}
\begin{align}
\label{quartopezzo}
&\int_{\G \setminus \B_c(0,1)} |\phi(xy)| \int_{1}^{\infty} t^{\frac{\alpha}{2}-1} \left( \frac{1}{2} \right)^k \ln(t)^k h(t,y) dt dy  \notag \\
&= \int_{\G \setminus \B_c(0,1)} |\phi(xy)| \int_{1}^{\infty} t^{\frac{\alpha}{2}-1-\frac{Q}{2}} \left( \frac{1}{2} \right)^k \ln(t)^k h(1,\delta_{\frac{1}{\sqrt{t}}}(y)) dt dy  \notag \\
&=   \int_{1}^{\infty} t^{\frac{\alpha}{2}-1-\frac{Q}{2}} \left( \frac{1}{2} \right)^k \ln(t)^k \int_{\G \setminus \B_c(0,1)} |\phi(xy)|h(1,\delta_{\frac{1}{\sqrt{t}}}(y)) dy dt  \notag \\
&\leq  \|h(1,\cdot)\|_{L^{\infty}(\G)} \int_{1}^{\infty} t^{\frac{\alpha}{2}-1-\frac{Q}{2}} \left( \frac{1}{2} \right)^k \ln^k(t) dt \int_{\G } |\phi(h)| dh  \\
& =  \|h(1,\cdot)\|_{L^{\infty}(\G)} \int_{1}^{\infty} t^{\frac{\alpha}{2}-1-\frac{Q}{2}} \left( \frac{1}{2} \right)^k \ln^k(t) dt \int_{\G } |\phi(h)| \frac{(1+ \|h \|_c^\ell)}{(1+ \|h \|_c^\ell)}dh  \notag \\
& \leq  c_\ell  \|h(1,\cdot)\|_{L^{\infty}(\G)} \int_{1}^{\infty} t^{\frac{\alpha}{2}-1-\frac{Q}{2}} \left( \frac{1}{2} \right)^k \ln^k(t) \int_{\G } \frac{1}{(1+ \|h \|_c^\ell)}dh  dt   \notag \\
& \leq  \widetilde{c}_\ell  \|h(1,\cdot)\|_{L^{\infty}(\G)} \int_{1}^{\infty} t^{\frac{\alpha}{2}-1-\frac{Q}{2}} \left( \frac{1}{2} \right)^k \ln^k(t)  dt,  \notag 
\end{align} 
where $c_{\ell}$ and $\widetilde{c}_\ell$ are suitable positive constants depending on $\ell$, and $\ell$ is understood to be strictly larger than $Q$. Now, analogously to the reasoning used in \eqref{secondopezzo}, exploiting the estimate in \eqref{ric1}, we get
\begin{equation}
\label{fourth}
\int_{\G \setminus \B_c(0,1)} |\phi(xy)| \int_{1}^{\infty} t^{\frac{\alpha}{2}-1} \left( \frac{1}{2} \right)^k |\ln(t)|^k h(t,y) dt dy  \leq c_\ell  \|h(1,\cdot)\|_{L^{\infty}(\G)} \frac{2}{(Q-\alpha)^{k+1}}   k!.
\end{equation}

Now, by combining \eqref{zeropezzo}, \eqref{zeropezzo1}, \eqref{zeropezzo2}, \eqref{first}, \eqref{second}, \eqref{third} and \eqref{fourth}, we get that for every $k \in \N$
\begin{align}
\label{conclusione}
&\left| \int_{\G} \phi(xy) \int_{0}^{\infty} t^{\frac{\alpha}{2}-1} \left( \frac{1}{2} \right)^k \ln^k(t) h(t,y) dt dy \right|  \notag \\
 \leq & \sup_{\G} | \phi(x \cdot) |  \ \frac{2}{\alpha^{k+1}} k! + \sup_{\G}| \phi(x \cdot) | \sup_{\B_c(0,1)} h(1, \cdot) \mu(\B_c(0,1)) \frac{2}{(Q-\alpha)^{k+1}}   k!  \notag \\
& + \widetilde{c}_\ell  \frac{2}{(\alpha+2N)^{k+1}} k!+ \widetilde{c}_\ell   \|h(1,\cdot)\|_{L^{\infty}(\G)} \frac{2}{(Q-\alpha)^{k+1}}   k!\\
 \leq & C(\phi, \ell)   \max \left\{\frac{1}{\alpha+2N}, \frac{1}{Q-\alpha}, \frac{1}{\alpha} \right\}^{k+1} k!,  \notag 
\end{align}
where $c(\phi,\ell)$ is a suitable positive constant.
Thus, $\psi(x, \alpha)$ is analytic on $0 < \alpha < Q$ with respect to $\alpha$.
\end{proof}

\begin{Remark}
\label{partefuori}
Following the scheme of the previous proof, exploiting the estimates in \eqref{third} and \eqref{fourth}, it holds that the map $$\alpha \to \int_{\G \setminus \B_c(0,1)} \phi(xy) \| y \|_{\alpha}^{\alpha-Q} dt$$ is analytic for $\alpha<Q$ (where it surely converges by direct computations similar to the ones in Proposition \ref{prop31}).
\end{Remark}

\subsection{Further properties of the map $\alpha \to \psi(x,\alpha)$}
\label{section-psi}
Here we generalize Proposition \ref{prop-schw}. We assume that $\phi \in C^{\infty}(\G)$ and $|\phi(z)| \leq K \|z\|_c^{-L}$ when $\|z\|_c \geq S$ for some $L, S,K>0$.
\begin{prop} \label{prop:anzola}
For any $\phi \in C^{\infty}(\G)$ such that $|\phi(z)| \leq K \|z\|_c^{-L}$ when $\|z\|_c \geq S$ for some $L, S, K>0$, for every $x \in \G$ the map
$$ \alpha \to \int_{\G} \phi(xy) P_{\alpha}(y) dy$$ is analytic on $(0,\min\{L,Q\})$.
\end{prop}
\begin{proof}
It is enough to prove that
\begin{equation*}
\alpha \to \int_{\G} \phi(xy) \| y \|_{\alpha}^{\alpha-Q}dy= \int_{\G} \phi(xy) \int_{0}^{\infty} t^{\frac{\alpha}{2}-1} h(t,y)dt dy
\end{equation*}
is analytic on $(0,\min\{L,Q\})$.
Let $x \in \G$ and set $R(x)= \max \{ 1+\|x\|_c, S \}$. For $k \in \N$ we consider
\begin{align}
\label{zeropezzo23}
&\left| \int_{\G} \phi(xy) \int_{0}^{\infty} t^{\frac{\alpha}{2}-1} \left( \frac{1}{2} \right)^k \ln^k(t) h(t,y) dt dy \right|  \notag \\
  \leq & \int_{\G} |\phi(xy)| \int_{0}^{\infty} t^{\frac{\alpha}{2}-1} \left( \frac{1}{2} \right)^k |\ln(t)|^k h(t,y) dt dy  \notag \\
= & \int_{\G \setminus \B_c(x^{-1},R(x))} |\phi(xy)| \int_{0}^{\infty} t^{\frac{\alpha}{2}-1} \left( \frac{1}{2} \right)^k |\ln(t)|^k h(t,y) dt dy \\
&+   \int_{\B_c(x^{-1},R(x))} |\phi(xy)| \int_{0}^{\infty} t^{\frac{\alpha}{2}-1} \left( \frac{1}{2} \right)^k |\ln(t)|^k h(t,y) dt dy.  \notag
\end{align}
By computations analogous to the ones  used in the previous Proposition \ref{prop-schw}, one verifies that
\begin{align}
\label{zeropezzo3}
& \int_{\B_c(x^{-1},R(x))} |\phi(xy)| \int_{0}^{\infty} t^{\frac{\alpha}{2}-1} \left( \frac{1}{2} \right)^k |\ln(t)|^k h(t,y) dt dy  \notag \\
= & \int_{\B_c(x^{-1},R(x))} |\phi(xy)| \int_{0}^{1} t^{\frac{\alpha}{2}-1} \left( \frac{1}{2} \right)^k |\ln(t)|^k h(t,y) dt dy  \notag \\
&+  \int_{\B_c(x^{-1},R(x))} |\phi(xy)| \int_{1}^{\infty} t^{\frac{\alpha}{2}-1} \left( \frac{1}{2} \right)^k |\ln(t)|^k h(t,y)dt dy\\
 \leq & \sup_{\B_c(x^{-1},R(x))} | \phi(x \cdot) | \sup_{\B_c(x^{-1},R(x))}|h(1,\cdot)|  2 \max \left\{\frac{1}{\alpha}, \frac{1}{Q-\alpha} \right\}^{k+1}k!.  \notag
\end{align}
Let us split the first term of the sum in \eqref{zeropezzo23} 
\begin{align}
\label{zeropezzo21}
&\int_{\G \setminus \B_c(x^{-1},R(x))} |\phi(xy)| \int_{0}^{\infty} t^{\frac{\alpha}{2}-1} \left( \frac{1}{2} \right)^k |\ln(t)|^k h(t,y) dt dy  \notag \\
=&\int_{\G \setminus \B_c(x^{-1},R(x))} |\phi(xy)| \int_{0}^{1} t^{\frac{\alpha}{2}-1} \left( \frac{1}{2} \right)^k |\ln(t)|^k h(t,y) dt dy \\
&+ \int_{\G \setminus \B_c(x^{-1},R(x))} |\phi(xy)| \int_{1}^{\infty} t^{\frac{\alpha}{2}-1} \left( \frac{1}{2} \right)^k |\ln(t)|^k h(t,y) dt dy.  \notag
\end{align}
Let us consider the first term of the sum in \eqref{zeropezzo21} as follows
\begin{align}
\label{terzopezzo1}
&\int_{\G \setminus \B_c(x^{-1},R(x))} |\phi(xy)| \int_0^1 t^{\frac{\alpha}{2}-1} \left( \frac{1}{2} \right)^k |\ln(t)|^k h(t,y) dt dy  \notag \\
&= \int_{\G \setminus \B_c(x^{-1},R(x))} |\phi(xy)| \int_0^1 t^{\frac{\alpha}{2}-1-\frac{Q}{2}} \left( \frac{1}{2} \right)^k |\ln(t)|^k h(1,\delta_{\frac{1}{\sqrt{t}}}(y)) dt dy \\
&=   \int_0^1 t^{\frac{\alpha}{2}-1-\frac{Q}{2}} \left( \frac{1}{2} \right)^k |\ln(t)|^k \int_{\G \setminus \B_c(x^{-1},R(x))} |\phi(xy)|h(1,\delta_{\frac{1}{\sqrt{t}}}(y)) dy dt . \notag
\end{align}
Let us change the variables as $z= \delta_{\frac{1}{\sqrt{t}}}(y)$, so that, since the Lebesgue measure on $\G$ is $Q$-homogeneous, we get the estimate
\begin{align*}
 &\int_0^1 t^{\frac{\alpha}{2}-1-\frac{Q}{2}} \left( \frac{1}{2} \right)^k |\ln(t)|^k \int_{\G \setminus \B_c(x^{-1},R(x))} |\phi(xy)|h(1,\delta_{\frac{1}{\sqrt{t}}}(y)) dy dt \\
  &=  \int_0^1 t^{\frac{\alpha}{2}-1-\frac{Q}{2}} \left( \frac{1}{2} \right)^k |\ln(t)|^k \int_{\G \setminus \delta_{\frac{1}{\sqrt{t}}}(\B_c(x^{-1},R(x)))} |\phi(x\delta_{\sqrt{t}}(z))|h(1,z) dz  \ t^{\frac{Q}{2}} dt \\
    &=  \int_0^1 t^{\frac{\alpha}{2}-1} \left( \frac{1}{2} \right)^k |\ln(t)|^k \int_{\G \setminus \delta_{\frac{1}{\sqrt{t}}}(\B_c(x^{-1},R(x)))} |\phi(x\delta_{\sqrt{t}}(z))|h(1,z) dz   dt \\
    & \leq K \int_0^1 t^{\frac{\alpha}{2}-1} \left( \frac{1}{2} \right)^k |\ln(t)|^k \int_{\G \setminus \delta_{\frac{1}{\sqrt{t}}}(\B_c(x^{-1},R(x)))} \frac{1}{\|x\delta_{\sqrt{t}}(z)\|_c^L} h(1,z) \frac{\|z\|_c^\ell}{\|z\|_c^{\ell}}dz  dt \\
    &  \leq K c_\ell \int_0^1 t^{\frac{\alpha}{2}-1} \left( \frac{1}{2} \right)^k |\ln(t)|^k \int_{\G \setminus \delta_{\frac{1}{\sqrt{t}}}(\B_c(x^{-1},R(x)))} \frac{1}{\|x\delta_{\sqrt{t}}(z)\|_c^L}  \frac{1}{\|z\|_c^{\ell}}dz  dt, 
\end{align*}
where $c_{\ell}$ is a suitable positive constant obtained exploiting the fact that the map $h(1, \cdot) \in \mathcal{S}(\G)$, for an arbitrary $\ell \in \N$. Let us change again variables, setting $h= x\delta_{\sqrt{t}}(z)$ and considering that the Lebesgue measure is left-invariant and $Q$-homogeneous, so that we get the estimate
\begin{align}\label{serve4}
 & K c_\ell \int_0^1 t^{\frac{\alpha}{2}-1} \left( \frac{1}{2} \right)^k |\ln(t)|^k \int_{\G \setminus \delta_{\frac{1}{\sqrt{t}}}(\B_c(x^{-1},R(x)))} \frac{1}{\|x\delta_{\sqrt{t}}(z)\|_c^L}  \frac{1}{\|z\|_c^{\ell}}dz  dt \notag \\
    &  \leq K c_\ell \int_0^1 t^{\frac{\alpha}{2}-1-\frac{Q}{2}} \left( \frac{1}{2} \right)^k |\ln(t)|^k \int_{\G \setminus \B_c(0,R(x))} \frac{1}{\|h \|_c^L}  \frac{1}{\|\delta_{\frac{1}{\sqrt{t}}}(x^{-1}h)\|_c^{\ell}}dh  dt \notag \\
    &  \leq K c_\ell \int_0^1 t^{\frac{\alpha}{2}-1-\frac{Q}{2}} \left( \frac{1}{2} \right)^k |\ln(t)|^k \int_{\G \setminus \B_c(0,R(x))} \frac{1}{\|h \|_c^L}  \frac{t^{\frac{\ell}{2}}}{\|x^{-1}h\|_c^{\ell}}dh  dt \\
        &  \leq K c_\ell\int_0^1 t^{\frac{\alpha}{2}-1-\frac{Q}{2}+\frac{\ell}{2}} \left( \frac{1}{2} \right)^k |\ln(t)|^k \int_{\G \setminus \B_c(0,R(x))} \frac{1}{\|h \|_c^L}  \frac{\|h\|^\ell}{\|x^{-1}h\|_c^{\ell}} \frac{1}{\|h\|_c^{\ell} }dh  dt \notag \\
                &  \leq K c_\ell \int_0^1t^{\frac{\alpha}{2}-1-\frac{Q}{2}+\frac{\ell}{2}} \left( \frac{1}{2} \right)^k |\ln(t)|^k (1+\|x\|_c)^{\ell}\int_{\G \setminus \B_c(0,R(x))} \frac{1}{\|h \|_c^{L+\ell}}  dh  dt. \notag 
\end{align} 
We have used an argument analogous to the one used in \eqref{serve3}.
The integral in the last line of \eqref{serve4} converges for every $\ell>Q-L$. If we choose some $\ell>Q-L$ large enough that also $\ell-Q > -\alpha$ holds, then we can choose also $\ell>Q$ and we can introduce $N=[\frac{\ell-Q}{2}]+1$ and we can repeat the computations leading to \eqref{ric2} getting
\begin{equation}
\label{third1}
\int_{\G \setminus \B_c(x^{-1},1)} |\phi(xy)|\int_0^1 t^{\frac{\alpha}{2}-1} \left( \frac{1}{2} \right)^k |\ln(t)|^k h(t,y) dt dy  \leq K \widetilde{c}_\ell \frac{2}{(\alpha+(\ell-Q))^{k+1}} k! .
\end{equation}
Let us then consider the second term of the sum in \eqref{zeropezzo21}
\begin{align}
\label{quartopezzo1}
&\int_{\G \setminus \B_c(x^{-1},R(x))} |\phi(xy)| \int_{1}^{\infty} t^{\frac{\alpha}{2}-1} \left( \frac{1}{2} \right)^k \ln^k(t) h(t,y) dt dy  \notag \\
&= \int_{\G \setminus \B_c(x^{-1},R(x))} |\phi(xy)| \int_{1}^{\infty} t^{\frac{\alpha}{2}-1-\frac{Q}{2}} \left( \frac{1}{2} \right)^k \ln^k(t) h(1,\delta_{\frac{1}{\sqrt{t}}}(y)) dt dy \\
&=   \int_{1}^{\infty} t^{\frac{\alpha}{2}-1-\frac{Q}{2}} \left( \frac{1}{2} \right)^k \ln^k(t) \int_{\G \setminus \B_c(x^{-1},R(x))} |\phi(xy)|h(1,\delta_{\frac{1}{\sqrt{t}}}(y)) dy dt .  \notag
\end{align}
Let us change the variables as $z= \delta_{\frac{1}{\sqrt{t}}}(y)$, so that, since the Lebesgue measure is $Q$-homogeneous we get
\begin{align*}
 & \int_{1}^{\infty} t^{\frac{\alpha}{2}-1-\frac{Q}{2}} \left( \frac{1}{2} \right)^k \ln^k(t) \int_{\G \setminus \B_c(x^{-1},R(x))} |\phi(xy)|h(1,\delta_{\frac{1}{\sqrt{t}}}(y)) dy dt \\
  &= \int_{1}^{\infty} t^{\frac{\alpha}{2}-1-\frac{Q}{2}} \left( \frac{1}{2} \right)^k \ln^k(t) \int_{\G \setminus \delta_{\frac{1}{\sqrt{t}}}(\B_c(x^{-1},R(x)))} |\phi(x\delta_{\sqrt{t}}(z))|h(1,z) dz  \ t^{\frac{Q}{2}} dt \\
    &= \int_{1}^{\infty} t^{\frac{\alpha}{2}-1} \left( \frac{1}{2} \right)^k \ln^k(t) \int_{\G \setminus \delta_{\frac{1}{\sqrt{t}}}(\B_c(x^{-1},R(x)))} |\phi(x\delta_{\sqrt{t}}(z))|h(1,z) dz   dt \\
    &  \leq K \int_{1}^{\infty} t^{\frac{\alpha}{2}-1} \left( \frac{1}{2} \right)^k \ln^k(t) \int_{\G \setminus \delta_{\frac{1}{\sqrt{t}}}(\B_c(x^{-1},R(x)))} \frac{1}{\|x\delta_{\sqrt{t}}(z)\|_c^L} h(1,z) \frac{\|z\|_c^\ell}{\|z\|^{\ell}}dz  dt \\
    &  \leq K c_\ell \int_{1}^{\infty} t^{\frac{\alpha}{2}-1} \left( \frac{1}{2} \right)^k \ln^k(t) \int_{\G \setminus \delta_{\frac{1}{\sqrt{t}}}(\B_c(x^{-1},R(x)))} \frac{1}{\|x\delta_{\sqrt{t}}(z)\|_c^L}  \frac{1}{\|z\|^{\ell}}dz  dt,
\end{align*}
where $c_{\ell}$ denotes a suitable positive constant depending on $\ell$, obtained considering that $h(1, \cdot) \in \mathcal{S}(\G)$.
Let us change again variables, setting $h= x\delta_{\sqrt{t}}(z)$ and recalling that the Lebesgue measure on $\G$ is left-invariant and $Q$-homogeneous, so that we get
\begin{align*}
 & K  c_\ell \int_{1}^{\infty} t^{\frac{\alpha}{2}-1} \left( \frac{1}{2} \right)^k |\ln(t)|^k \int_{\G \setminus \delta_{\frac{1}{\sqrt{t}}}(\B_c(x^{-1},R(x)))} \frac{1}{\|x\delta_{\sqrt{t}}(z)\|_c^L}  \frac{1}{\|z\|^{\ell}}dz  dt \\
    &  \leq K c_\ell \int_{1}^{\infty} t^{\frac{\alpha}{2}-1-\frac{Q}{2}} \left( \frac{1}{2} \right)^k \ln^k(t) \int_{\G \setminus \B_c(0,R(x))} \frac{1}{\|h \|_c^L}  \frac{1}{\|\delta_{\frac{1}{\sqrt{t}}}(x^{-1}h)\|^{\ell}}dz  dt \\
    &  \leq K c_\ell \int_{1}^{\infty} t^{\frac{\alpha}{2}-1-\frac{Q}{2}} \left( \frac{1}{2} \right)^k \ln^k(t) \int_{\G \setminus \B_c(0,R(x))} \frac{1}{\|h \|_c^L}  \frac{t^{\frac{\ell}{2}}}{\|x^{-1}h\|^{\ell}}dz  dt \\
        &  \leq K c_\ell \int_{1}^{\infty} t^{\frac{\alpha}{2}-1-\frac{Q}{2}+\frac{\ell}{2}} \left( \frac{1}{2} \right)^k \ln^k(t) dt\int_{\G \setminus \B_c(0,R(x))} \frac{1}{\|h \|_c^L}  \frac{\|h\|^\ell}{\|x^{-1}h\|^{\ell}} \frac{1}{\|h\|^{\ell} }dz   \\
                &  \leq K c_\ell \int_{1}^{\infty} t^{\frac{\alpha}{2}-1-\frac{Q}{2}+\frac{\ell}{2}} \left( \frac{1}{2} \right)^k \ln^k(t) (1+\|x\|_c)^{\ell} dt \int_{\G \setminus \B_c(0,R(x))} \frac{1}{\|h \|_c^{L+\ell}}  dz\\
                 &  \leq K c_\ell \int_{1}^{\infty} t^{\frac{\alpha}{2}-1-\frac{L}{2}} \left( \frac{1}{2} \right)^k \ln^k(t) (1+\|x\|_c)^{\ell} dt \int_{\G \setminus \B_c(0,R(x))} \frac{1}{\|h \|_c^{L+\ell}}  dz.
\end{align*}
The integral in the last line is justified if we consider $\ell>Q-L$, hence when $L>Q-\ell$. Now, similarly to the procedure applied to deal with \eqref{secondopezzo}, exploiting an estimate analogous to \eqref{ric1}, we can conclude that
\begin{equation}
\label{fourth1}
\begin{split}
\int_{\G \setminus \B_c(x^{-1},R)} |\phi(xy)| \int_{1}^{\infty} t^{\frac{\alpha}{2}-1} \left( \frac{1}{2} \right)^k \ln^k(t) h(t,y) dt dy  &\leq K c_\ell \frac{2}{(L-\alpha)^{k+1}}   k!.
\end{split}
\end{equation}
Now, by combining \eqref{zeropezzo23}, \eqref{zeropezzo3},  \eqref{third1} and \eqref{fourth1}, we get that for every $k \in \N$, and for every $\ell>\max \{Q-L, Q-\alpha, 2Q \}$
\begin{equation}
\label{conclusione1}
\begin{split}
&\left| \int_{\G} \phi(xy) \int_{0}^{\infty} t^{\frac{\alpha}{2}-1} \left( \frac{1}{2} \right)^k \ln^k(t) h(t,y) dt dy \right| \\
& \leq C \max \left\{ \frac{1}{Q-\alpha},\frac{1}{L-\alpha} ,\frac{1}{\alpha}, \frac{1}{\alpha+(\ell-Q)} \right\}^{k+1} k!,
\end{split}
\end{equation}
for some suitable constant $C>0$. Therefore, $\psi(x, \alpha)$ is analytic on $0 < \alpha < \min\{Q,L\}$.
\end{proof}

\begin{Remark}
\label{partefuori1}
Following the previous proof, in the same hypotheses of Proposition \ref{prop:anzola}, we obtain that the map $$\alpha \to \int_{\G \setminus \B_c(x^{-1},R(x))} \phi(xy) \| y \|_{\alpha}^{\alpha-Q} dt$$ is analytic on $(-\infty, \min\{Q,L\})$. It is enough to consider \eqref{zeropezzo21} and the estimates in \eqref{third1} and \eqref{fourth1} that are valid for every $k \in \N$, and for every $\ell>\max\{ Q-L, Q-\alpha, 2Q\}$. The analyticity is guaranteed on $(-\infty, \min\{Q,L\})$ also for the map $$\alpha \to \int_{\G \setminus \B_c(0,1)} \phi(xy) \| y \|_{\alpha}^{\alpha-Q} dt,$$ since, in order to estimate the $k$-th derivative of $$ \alpha \to \int_{\B_c(x^{-1}, R) \setminus \B_c(0,1)} \phi(xy) \| y \|_{\alpha}^{\alpha-Q} dt,$$ one can simplify the arguments carried out in the proof of Proposition \ref{prop-schw}.
\end{Remark}
\subsection{The map $\alpha \to \psi(x, \alpha)$ is analytically continued on (-2,0]}
\label{primastriscia}

In this subsection, we build up the analytic continuation of $\alpha \to \psi(x,\alpha)$ on the strip $(-2,0]$. 
\begin{teo}\label{teo:samoggia}
For every $\phi \in \mathcal{S}(\G)$ and $x \in \G$, the map 
$$(0,Q) \to \psi(x,\alpha)$$ 
can be analytically continued to the interval $(-2,Q)$. Moreover, the representation
\begin{align}\label{peralfamag-2}
\psi(x,\alpha)= \int_{\G \setminus \B_c(0,1)} \phi(xy)P_{\alpha}(y) dy + \int_{\B_c(0,1)} (\phi(xy)-\phi(x))P_{\alpha}(y) dy + \phi(x) \frac{1}{\Gamma(\frac{\alpha}{2})} \frac{1}{\alpha} \sigma(\alpha)
\end{align}
holds for $\alpha \in (-2,Q)$.
\end{teo}
Before proving Theorem \ref{teo:samoggia}, we prove two technical results concerning the convergence and the regularity of the second term of the sum in \eqref{peralfamag-2}, respectively.
\begin{lem}\label{rem-conv}
For every $\phi \in \mathcal{S}(\G)$ and for every $x \in \G$ the integral
$$\int_{\B_c(0,1)}(\phi(xy)-\phi(x))P_{\alpha}(y) dy $$ converges for every $-2<\alpha<Q$. 
\end{lem}

\begin{proof} By the coarea formula of Corollary \ref{cor-coarea} and the homogeneity of $P_{\alpha}$ and of $\mathcal{S}^{Q-1}_{\infty}$ we get
\begin{align*}
\int_{\B_c(0,1)}(\phi(xy)-\phi(x))P_{\alpha}(y) dy &= \int_0^1 \int_{\partial \B_c(0,r)} (\phi(xy)-\phi(x))P_{\alpha}(y) d \mathcal{S}^{Q-1}_{\infty}(y) dr\\
&= \int_0^1 r^{\alpha-1} \int_{\partial \B_c(0,1)} (\phi(x \delta_r(z))-\phi(x))P_{\alpha}(z) d \mathcal{S}^{Q-1}_{\infty}(z) dr\\
&= \int_0^1 r^{\alpha-1} \int_{\partial \B_c(0,1)} (\phi(x \delta_r(z))-\phi(x)-D_h \phi(x)(\delta_r(z)))P_{\alpha}(z) d \mathcal{S}^{Q-1}_{\infty}(z) dr.
\end{align*}
In order to obtain the last equality, we have exploited the equality 
\begin{equation}\label{rel-palla-diffe}
\int_{\B_c(0,1)} D_h\phi(x)(y) P_{\alpha}(y) dy=0,
\end{equation}
that is justified by the combination of few results. In fact, first we know that $\phi \in \mathcal{S}(\G) \subset C^{\infty}(\G) \subset C^1(\G) \subset C^1_h(\G)$, then by the definition of h-homomorphism we know that for every $x, y \in \G$ it holds that
$$D_h\phi(x)(y)=\langle \nabla_H\phi(x) ,y\rangle = -\langle \nabla_H\phi(x) ,-y\rangle =- \langle \nabla_H\phi(x) ,y^{-1}\rangle=-D_h\phi(x)(y^{-1}). $$
Thus \eqref{rel-palla-diffe} follows by the following remarks
\begin{itemize}
\item by the homogeneity of $d_c$, if $y \in \B_c(0,1)$, then $y^{-1} \in \B_c(0,1)$. 
\item by the property of the heat kernel $h$ stated in Theorem \ref{T1}(iii), it can be verified that $$\|y\|_{\alpha}^{\alpha-Q}=\|y^{-1}\|_{\alpha}^{\alpha-Q}$$ for every $y \in \G$, $y \neq 0$.
Hence, also $P_{\alpha}(y)=P_{\alpha}(y^{-1})$ holds.
\end{itemize}
Now, by Theorem \ref{Taylor} (in light of \eqref{def-grad} and \eqref{forma-grad}), we know that $$|\phi(x \delta_r(z))-\phi(x)-D_h \phi(x)(\delta_r(z)))| =O(r^2) \quad \mathrm{ \ as  \ } r \to 0,$$ as a consequence we know that
\begin{align*}
\int_{\B_c(0,1)}|\phi(xy)-\phi(x)||P_{\alpha}(y)| dy & \leq  c \int_0^1 r^{\alpha+1} dr \int_{\partial \B_c(0,1)} |P_{\alpha}(z)| d \mathcal{S}^{Q-1}_{\infty}(z) \\
 & = c \ \frac{1}{| \Gamma(\frac{\alpha}{2})|} \frac{1}{\alpha+2}\sigma(\alpha),
\end{align*}
for some suitable constant $c>0$. The integral converges for every $\alpha$ such that $\alpha >-2$, since 
$$\frac{1}{| \Gamma(\frac{\alpha}{2})|} \frac{1}{\alpha+2}\sigma(\alpha)=\frac{1}{| \Gamma(\frac{\alpha}{2})|\alpha} \frac{\alpha}{\alpha+2}\sigma(\alpha)$$ and
$\lim_{\alpha \to 0} \Gamma(\frac{\alpha}{2}) \alpha=2$.
\end{proof}

\begin{lem} \label{claim-partedentro}
For every $\phi \in \mathcal{S}(\G)$ and for every $x \in \G$ the map
\begin{equation}
\label{partedentro}
\alpha \to \int_{\B_c(0,1)} (\phi(xy)-\phi(x))\|y \|_{\alpha}^{\alpha-Q}dy
\end{equation}
is analytic on $(-2,Q)$.
\end{lem}
\begin{proof}
Let us consider the $k$-th derivative of \eqref{partedentro} with respect to $\alpha$, for $k \in \N$. In particular, we will show that derivative and integral can be exchanged by estimating the term
\begin{equation}
\left| \int_{\B_c(0,1)} (\phi(xy)-\phi(x))\int_0^{\infty} t^{\frac{\alpha}{2}-1} (\ln(t))^k \left( \frac{1}{2} \right)^k h(t,y)dt dy \right|.
\end{equation}
We split the integral of \eqref{partedentro} as
\begin{align}
&\left| \int_{\B_c(0,1)} (\phi(xy)-\phi(x))\int_0^{\infty} t^{\frac{\alpha}{2}-1} (\ln(t))^k \left( \frac{1}{2} \right)^k h(t,y)dt dy \right| \leq \notag\\
& \left| \int_{\B_c(0,1)} (\phi(xy)-\phi(x))\int_0^{1} t^{\frac{\alpha}{2}-1} (\ln(t))^k \left( \frac{1}{2} \right)^k h(t,y)dt dy \right|\\
&+\left| \int_{\B_c(0,1)} (\phi(xy)-\phi(x))\int_1^{\infty} t^{\frac{\alpha}{2}-1} (\ln(t))^k \left( \frac{1}{2} \right)^k h(t,y)dt dy \right|. \notag
\end{align}
The second term can be dealt with repeating the steps in \eqref{secondopezzo}
and we get that
\begin{equation}\label{zero}
\begin{split}
&\left| \int_{\B_c(0,1)} ( \phi(xy)-\phi(x)) \int_{1}^{\infty} t^{\frac{\alpha}{2}-1} \left( \frac{1}{2} \right)^k \ln^k(t) h(t,y)dt dy  \right|\\
&\leq \sup_{\B_c(0,1)}| \phi(x \cdot)-\phi(x) | \sup_{\B_c(0,1)} h(1, \cdot) \mu(\B_c(0,1)) \frac{2}{(Q-\alpha)^{k+1}}   k!.
\end{split}
\end{equation}
Then we need to deal with 
$$\left| \int_{\B_c(0,1)} (\phi(xy)-\phi(x))\int_0^{1} t^{\frac{\alpha}{2}-1} (\ln(t))^k \left( \frac{1}{2} \right)^k h(t,y)dt dy \right|.$$
By exploiting the equality \eqref{rel-palla-diffe} justified in Lemma \ref{rem-conv}, we get
\begin{align}
\label{un}
&\left| \int_{\B_c(0,1)} (\phi(xy)-\phi(x))\int_0^{1} t^{\frac{\alpha}{2}-1} (\ln(t))^k \left( \frac{1}{2} \right)^k h(t,y)dt dy \right|  \notag \\
=&\left| \int_{\B_c(0,1)} (\phi(xy)-\phi(x)-D_h\phi(x)(y) )\int_0^{1} t^{\frac{\alpha}{2}-1} (\ln(t))^k \left( \frac{1}{2} \right)^k h(t,y)dt dy \right|  \notag \\
 \leq & \int_{\B_c(0,1)} |\phi(xy)-\phi(x)-D_h\phi(x)(y) |\int_0^{1} t^{\frac{\alpha}{2}-1} |\ln(t)|^k \left( \frac{1}{2} \right)^k h(t,y)dt dy \\
 =  & \int_0^{1} t^{\frac{\alpha}{2}-1} |\ln(t)|^k \left( \frac{1}{2} \right)^k  \int_{\B_c(0,1)} |\phi(xy)-\phi(x)-D_h\phi(x)(y) | h(t,y)dy dt. \notag
\end{align}
Now, by Theorem \ref{Taylor} (in light of \eqref{def-grad} and \eqref{forma-grad}), we know that $$|\phi(x y)-\phi(x)-D_h \phi(x)(y))| =O(\|y \|_c^2) \quad \mathrm{ \  as \ } \|y \|_c \to 0.$$ Hence, for some $C\geq 0$ we have
\begin{align*}
& \int_0^{1} t^{\frac{\alpha}{2}-1} |\ln(t)|^k \left( \frac{1}{2} \right)^k  \int_{\B_c(0,1)} |\phi(xy)-\phi(x)-D_h\phi(x)(y)| h(t,y)dy dt \\
\leq & \ C \int_0^{1} t^{\frac{\alpha}{2}-1} |\ln(t)|^k \left( \frac{1}{2} \right)^k  \int_{\B_c(0,1)} \|y \|_c^2 h(t,y)dy dt \\
=& \ C \int_0^{1} t^{\frac{\alpha}{2}-1} |\ln(t)|^k \left( \frac{1}{2} \right)^k t^{-\frac{Q}{2}} \int_{\B_c(0,1)} \|y \|_c^2 h(1,\delta_{\frac{1}{\sqrt{t}}}(y))dy dt,
\end{align*}
due to the homogeneity of $h$. Now, by a change of variables $z= \delta_{\frac{1}{\sqrt{t}}}(y)$ and by the homogeneity of the Lebesgue measure on $\G$ we get
\begin{align}
\label{dos}
& C\int_0^{1} t^{\frac{\alpha}{2}-1} |\ln(t)|^k \left( \frac{1}{2} \right)^k t^{-\frac{Q}{2}} \int_{\B_c(0,1)} \|y \|_c^2 h(1,\delta_{\frac{1}{\sqrt{t}}}(y))dy dt  \notag \\
=& \ C\int_0^{1} t^{\frac{\alpha}{2}-1} |\ln(t)|^k \left( \frac{1}{2} \right)^k  \int_{\B_c(0,\frac{1}{\sqrt{t}})} \|\delta_{\sqrt{t}}(z) \|_c^2 h(1,z)dz dt  \notag \\
=& \ C\int_0^{1} t^{\frac{\alpha}{2}-1} |\ln(t)|^k \left( \frac{1}{2} \right)^k  \int_{\B_c(0,\frac{1}{\sqrt{t}})} t\| z\|_c^2 h(1,z)dz dt \\
\leq &  \ C\int_0^{1} t^{\frac{\alpha}{2}} |\ln(t)|^k \left( \frac{1}{2} \right)^k  dt\int_{\G} \|z\|_c^2h(1,z)dz.   \notag
\end{align}
Notice that $\int_{\G} \|z\|_c^2 h(1,z) dz < \infty$ since $h(1,z) \in \mathcal{S}(\G)$. Now, observe that, according to \eqref{bk}, $$ b_k= \int_0^{1} t^{\frac{\alpha}{2}} |\ln(t)|^k \left( \frac{1}{2} \right)^k   dt,$$ when $N=1$, recalling that $N=1 >- \frac{\alpha}{2}$ if $\alpha>-2$. Hence, we can repeat the argument involving the estimate in \eqref{ric2} obtaining that for every $k \in \N$
\begin{equation}
\label{tres}
\begin{split}
 \int_0^{1} t^{\frac{\alpha}{2}} |\ln(t)|^k \left( \frac{1}{2} \right)^k   dt
\leq C \frac{2}{(\alpha+2)^{k+1}} k!.
\end{split}
\end{equation} 
By combining \eqref{zero}, \eqref{un}, \eqref{dos} and \eqref{tres} the proof is concluded.
\end{proof}

Now, we are in position to prove Theorem \ref{teo:samoggia}.

\begin{proof}[Proof of Theorem \ref{teo:samoggia}]
Let $x \in \G$ and $\phi \in \mathcal{S}(\G)$. We rearrange the map $(0,Q) \ni \alpha \to \psi(x, \alpha)$ as 
\begin{align}
\label{cavametti}
\psi(x,\alpha)&= \int_{\G} \phi(xy)P_{\alpha}(y) dy  \notag \\
&= \int_{\G \setminus \B_c(0,1)} \phi(xy)P_{\alpha}(y) dy +  \int_{\B_c(0,1)} \phi(xy)P_{\alpha}(y) dy  \notag \\
&= \int_{\G \setminus \B_c(0,1)} \phi(xy)P_{\alpha}(y) dy \\
&+ \int_{\B_c(0,1)} (\phi(xy)-\phi(x))P_{\alpha}(y) dy + \int_{\B_c(0,1)} \phi(x)P_{\alpha}(y) dy.  \notag
\end{align}
Since $\alpha \in (0,Q)$, we can compute the third integral of the sum in \eqref{cavametti} exploiting the coarea formula of Corollary \ref{cor-coarea} and homogeneity properties, getting
\begin{align}\label{eq:dentropalla}
 \int_{\B_c(0,1)} \phi(x)P_{\alpha}(y) dy &= \phi(x) \frac{1}{\Gamma(\frac{\alpha}{2})} \int_{\B_c(0,1)} \| y \|_{\alpha}^{\alpha-Q} dy \notag\\ &= \phi(x) \frac{1}{\Gamma(\frac{\alpha}{2})} \int_{0}^{1} \int_{\partial \B_c(0,r)} \| y \|_{\alpha}^{\alpha-Q} d\mathcal{S}^{Q-1}_{\infty}(y)  dr \notag\\
&= \phi(x) \frac{1}{\Gamma(\frac{\alpha}{2})} \int_{0}^{1} \int_{\partial \B_c(0,r)} \| \delta_r(z) \|_{\alpha}^{\alpha-Q} d\mathcal{S}_{\infty}^{Q-1}(\delta_r(z))  dr \notag\\
&= \phi(x) \frac{1}{\Gamma(\frac{\alpha}{2})} \int_{0}^{1} \int_{\partial \B_c(0,1)} r^{\alpha-Q} \| z \|_{\alpha}^{\alpha-Q} r^{Q-1} d\mathcal{S}_{\infty}^{Q-1}(z)  dr \notag\\
&= \phi(x) \frac{1}{\Gamma(\frac{\alpha}{2})}  \int_{0}^{1} r^{\alpha-1} \int_{\partial \B_c(0,1)} \| z \|_{\alpha}^{\alpha-Q}  d\mathcal{S}_{\infty}^{Q-1}(z)  dr\\
&= \phi(x) \frac{1}{\Gamma(\frac{\alpha}{2})} \sigma(\alpha) \int_{0}^{1} r^{\alpha-1} dr \sigma(\alpha)\notag \\
&= \phi(x) \frac{1}{\Gamma(\frac{\alpha}{2})} \sigma(\alpha) \lim_{a \to 0}\bigg[ \frac{r^{\alpha}}{\alpha} \bigg]_{r=a}^1 \sigma(\alpha) \notag \\
&= \phi(x) \frac{1}{\Gamma(\frac{\alpha}{2})} \frac{1}{\alpha} \sigma(\alpha). \notag
\end{align}
Hence, for every $\alpha \in (0,Q)$ the equality
\begin{align}\label{patrono}
\psi(x,\alpha)= \int_{\G \setminus \B_c(0,1)} \phi(xy)P_{\alpha}(y) dy + \int_{\B_c(0,1)} (\phi(xy)-\phi(x))P_{\alpha}(y) dy + \phi(x) \frac{1}{\Gamma(\frac{\alpha}{2})} \frac{1}{\alpha} \sigma(\alpha)
\end{align}
holds. \\
By Remark \ref{partefuori}, the first term of the sum in \eqref{patrono} is analytic with respect to $\alpha$ on $(-\infty,Q)$. 
On the other hand, the third term of the sum in \eqref{patrono} is analytic with respect to $\alpha$ on $(-\infty,Q)$, since by Item \ref{itemcalfa} in Section \ref{sectioncalfa}, $\alpha \to \sigma(\alpha)$ is analytic on $(-\infty,Q)$ and $\lim_{\alpha \to 0} \Gamma(\frac{\alpha}{2}) \alpha=2$.\\
We point out that by Lemma \ref{rem-conv}, the expression in \eqref{patrono} is well-defined for every $\alpha \in (-2,Q)$. Eventually, by Lemma \ref{claim-partedentro}, the second term of the sum in \eqref{patrono} is analytic with respect to $\alpha$ on $(-2,Q)$ and this concludes the proof.
\end{proof}

\begin{Remark}
Let us show that $\psi(x,0)=\phi(x)$. 
Since
\begin{enumerate}
\item $\lim_{\alpha \to 0} \frac{1}{\alpha \Gamma(\frac{\alpha}{2})}=\frac{1}{2},$
\item by Item \ref{fa2!} of Section \ref{sectioncalfa}, $\lim_{\alpha \to 0} \sigma(\alpha)=2,$
\end{enumerate}
then
$$\lim_{\alpha \to 0} \frac{1}{\alpha \Gamma(\frac{\alpha}{2})}\sigma(\alpha)=1$$
holds. Since $\lim_{\alpha \to 0} \frac{1}{\Gamma(\frac{\alpha}{2})}=0,$ then $\psi(x,0)=\lim_{\alpha \to 0} \psi(x, \alpha)$ is \begin{small}
\begin{align*}
 \lim_{\alpha \to 0} \left( \frac{1}{\Gamma(\frac{\alpha}{2})}\int_{\G \setminus \B_c(0,1)} \phi(xy)\|y \|_{\alpha}^{\alpha-Q} dy + \frac{1}{\Gamma(\frac{\alpha}{2})}  \int_{\B_c(0,1)} (\phi(xy)-\phi(x))\|y \|_{\alpha}^{\alpha-Q}dy +\phi(x) \frac{1}{\Gamma(\frac{\alpha}{2})} \frac{1}{\alpha} \sigma(\alpha) \right)=\phi(x).
\end{align*}
\end{small}

\end{Remark}
In the following proposition, we deduce an useful representation of $\psi(x,\alpha)$ valid for $\alpha \in (-2,0)$.
\begin{prop}\label{rem-repre}
For every $x \in \G$, for every $\phi \in \mathcal{S}(\G)$ the following representation holds
\begin{equation}\label{representation-2}
 \psi(x, \alpha)= \int_{\G}(\phi(xy)-\phi(x))P_{\alpha}(y) dy,
\end{equation}
for every $\alpha \in (-2,0)$.
\end{prop}
\begin{proof}
Let us compute for $-2<\alpha<0$ the following integral
\begin{align}\label{eq:fuoripalla}
 \int_{\G \setminus \B_c(0,1)} \phi(x)P_{\alpha}(y) dy &= \phi(x) \frac{1}{\Gamma(\frac{\alpha}{2})} \int_{\G \setminus \B_c(0,1)} \| y \|_{\alpha}^{\alpha-Q} dy \notag \\ &= \phi(x) \frac{1}{\Gamma(\frac{\alpha}{2})} \int_{1}^{\infty} \int_{\partial \B_c(0,r)} \| y \|_{\alpha}^{\alpha-Q} d\mathcal{S}^{Q-1}_{\infty}(y)  dr \notag \\
&= \phi(x) \frac{1}{\Gamma(\frac{\alpha}{2})} \int_{1}^{\infty} \int_{\partial \B_c(0,r)} \| \delta_r(z) \|_{\alpha}^{\alpha-Q} d\mathcal{S}_{\infty}^{Q-1}(\delta_r(z))  dr\\
&= \phi(x) \frac{1}{\Gamma(\frac{\alpha}{2})} \lim_{a \to \infty} \bigg[ \frac{r^{\alpha}}{\alpha} \bigg]_{r=1}^a \int_{\partial \B_c(0,1)} \| z \|_{\alpha}^{\alpha-Q}  d\mathcal{S}_{\infty}^{Q-1}(z) \notag \\
&= - \phi(x) \frac{1}{\Gamma(\frac{\alpha}{2})} \frac{1}{\alpha} \sigma(\alpha). \notag
\end{align}
As usual, we have exploited the coarea formula of Corollary \ref{cor-coarea} and the homogeneity properties. The proof follows by combining Theorem \ref{teo:samoggia} and \eqref{eq:fuoripalla}.
\end{proof}

\begin{Remark}
The results of the current subsection hold as well being true if $\phi \in C^{\infty}(\G)$ and $|\phi(z)| \leq K \|z\|_c^{-L}$ as $\|z\|_c>S$ for some $L, S, K>0$. In particular, under these hypotheses one can check that $\psi(x,\alpha)$ can be analytically continued on the interval $(-2,\min\{Q,L\})$.
\end{Remark}
\subsection{Extension of the definition of $\mathcal{L}^s$ to $-\frac{Q}{2}<s<1$}
\label{sec-introddistr}
In this section, we introduce a further notation. For every fixed $x \in \G$ and $\phi \in \mathcal{S}(\G)$ it is useful to denote the map $\psi(x,\alpha)$ also by $\psi_{\phi}(x,\alpha)$. This notation appears in the paper only when we wish to emphasize the dependence of $\psi(x,\alpha)$ on the choice of the map $\phi$.
By combining \cite[Theorem 3.11]{FF15}, which states the pointwise integral representation of $\mathcal{L}^s\phi(x)$ in \eqref{3.11}, and the representation \eqref{representation-2} of $\psi_{\phi}(x,\alpha)$ showed in Proposition \ref{rem-repre}, we notice that for every $ s \in (0,1)$, $x \in \G $, and $\phi \in \mathcal{S}(\G)$ the equality
\begin{equation}
\label{fiuguale}
 \psi_{\phi}(x, -2s)= \mathcal{L}^s\phi(x)
 \end{equation}
holds. Since for every $x \in \G$ and $\phi \in \mathcal{S}(\G)$ we have analytically continued the map $\alpha \to \psi_{\phi}(x,\alpha)$ on $ (-2,Q)$, we can introduce for every $x \in \G$  a family of tempered distributions $\{ \widetilde{P}_{\alpha}^x \}_{\alpha \in (-2,Q)}$ by re-reading the map $\psi_{\phi}(x, \alpha)$ as a the action of $\widetilde{P}^x_{\alpha}$ on $\phi \in \mathcal{S}(\G)$, so that for every $\alpha \in (-2,Q)$ we have
$$\widetilde{P}_{\alpha}^x (\phi ):= \psi_{\phi}(x, \alpha).$$ 
Thus, for every $s \in (0,1)$ and $x \in \G$ we have
$$ \widetilde{P}_{-2s}^x (\phi  )= \psi_{\phi}(x, -2s)=\mathcal{L}^s\phi(x),$$
for every $\phi \in \mathcal{S}(\G)$.
We notice that, when $0<\alpha<Q$, $P_{\alpha}$ is a distribution as well and $\widetilde{P}^x_{\alpha}$ and $P_{\alpha}$ satisfy the following relation
\begin{equation}\label{peralfamag0}
\widetilde{P}_{\alpha}^x(\phi)=P_{\alpha}(\phi \circ l_x)= \int_{\G}\phi(xy)P_{\alpha}(y)dy,
\end{equation}
for every $x \in \G$ and $\phi \in \mathcal{S}(\G)$. It is natural to introduce the following notion.
\begin{deff}\label{def:fralap}
For every $s \in (-\frac{Q}{2},1)$,
$$\mathcal{L}^s \phi(x):=  \psi_{\phi}(x,-2s)=\widetilde{P}_{-2s}^x(\phi),$$
for every $\phi \in \mathcal{S}(\G)$ and for every $x \in \G$.
\end{deff}
In the following two subsections we study the properties of the operator $\mathcal{L}^s$ introduced in Definition \ref{def:fralap}.
\subsection{An estimate of $x \to \psi(x, \alpha)$}
\label{estimates-psi}
First, we observe that $ \G \ni x \to \psi(x,\alpha)$ is a smooth map. With the following result, later we wish to prove the consistence of the extension of the operator $\mathcal{L}^s$ provided by Definition \ref{def:fralap}.

\begin{prop}\label{stima32}
For every $\alpha>0$ and $\phi \in \mathcal{S}(\G)$ $$ \psi(x, \alpha) =O(\|x \|_c^{\alpha-Q})$$ as $\| x \|_c \to \infty$.
\end{prop}
\begin{proof}
Recalling the estimate in \eqref{relcc-alfa} and the fact that $\phi \in \mathcal{S}(\G)$, we obtain
\begin{align*}
&\int_{\G}\phi(xy) P_{\alpha}(y) dy = \ \frac{1}{\Gamma(\frac{\alpha}{2}) } \int_{\G} \phi(xy) \| y \|_{\alpha}^{\alpha-Q} dy 
  \leq  \  \frac{1}{\Gamma(\frac{\alpha}{2}) } C_{\alpha} \int_{\G} |\phi(xy) | \| y \|_{c}^{\alpha-Q} dy \\
  = & \ \frac{1}{\Gamma(\frac{\alpha}{2}) } C_{\alpha} \int_{\G} |\phi(z)| \| x^{-1} z \|_{c}^{\alpha-Q} dz 
  =  \ \frac{1}{\Gamma(\frac{\alpha}{2}) } C_{\alpha} \int_{\G} |\phi(z)| \frac{(1+\| z \|_c^k)}{(1+\|z\|_c^k)} \| x^{-1} z \|_{c}^{\alpha-Q} dz \\
  \leq & \ \frac{1}{\Gamma(\frac{\alpha}{2}) } C_{\alpha} c_k \int_{\G} \frac{1}{(1+\|z\|_c^k)} \| x^{-1} z \|_{c}^{\alpha-Q} dz 
  =   \ \frac{1}{\Gamma(\frac{\alpha}{2}) } C_{\alpha} c_k \int_{\G} \frac{1}{(1+\|z\|_c^k)}\frac{1}{ \| x^{-1} z \|_{c}^{Q-\alpha}} dz \\
  = & \ \frac{1}{\Gamma(\frac{\alpha}{2}) } C_{\alpha} c_k \int_{\G \setminus \B_c(x, \frac{\|x\|_c}{2})} \frac{1}{(1+\|z\|_c^k)}\frac{1}{ \| x^{-1} z \|_{c}^{Q-\alpha}} dz 
+ \frac{1}{\Gamma(\frac{\alpha}{2}) } C_{\alpha} c_k \int_{\B_c(x, \frac{\|x\|_c}{2})} \frac{1}{(1+\|z\|_c^k)}\frac{1}{ \| x^{-1} z \|_{c}^{Q-\alpha}} dz,
\end{align*}
where $C_{\alpha}$ and $c_k$ are suitable positive constants.
Let us discuss separately the two terms of the last line. By exploiting homogeneity properties and the coarea formula of Corollary \ref{cor-coarea} we get
\begin{align}
\label{same-comp}
 \frac{1}{\Gamma(\frac{\alpha}{2}) } &C_{\alpha} c_k \int_{\G \setminus \B_c(x, \frac{\|x\|_c}{2})} \frac{1}{(1+\|z\|_c^k)}\frac{1}{ \| x^{-1} z \|_{c}^{Q-\alpha}} dy  \leq  \frac{1}{\Gamma(\frac{\alpha}{2}) } C_{\alpha} c_k \frac{{2}^{Q-\alpha}}{\| x \|_c^{Q-\alpha}} \int_{\G \setminus \B_c(x, \frac{\|x\|_c}{2})} \frac{1}{(1+\|z\|_c^k)} dy  \notag \\
& \leq  \frac{1}{\Gamma(\frac{\alpha}{2}) } C_{\alpha} c_k \frac{{2}^{Q-\alpha}}{\| x \|_c^{Q-\alpha}} \int_{\G } \frac{1}{(1+\|z\|_c^k)} dy\\
 & \leq  \frac{1}{\Gamma(\frac{\alpha}{2}) } C_{\alpha} c_k \frac{{2}^{Q-\alpha}}{\| x \|_c^{Q-\alpha}} \mathcal{S}^{Q-1}_{\infty}(\partial \B_c(0,1)) \int_{0}^{\infty} \frac{1}{(1+r^k)} r^{Q-1}dy . \notag
\end{align}
The integral converges for every $k>Q$. Hence, we proved that for every $k>Q$ the estimate
\begin{equation}
\label{uno}
 \frac{1}{\Gamma(\frac{\alpha}{2}) } C_{\alpha} c_k \int_{\G \setminus \B_c(x, \frac{\|x\|_c}{2})} \frac{1}{(1+\|z\|_c^k)}\frac{1}{ \| x^{-1} z \|_{c}^{Q-\alpha}} dy  \leq c(k,\alpha) \| x \|_c^{\alpha-Q}
\end{equation}
holds for some suitable positive constant $c(k,\alpha)$. 

Concerning the second term of the sum, we have
\begin{align*}
&\frac{1}{\Gamma(\frac{\alpha}{2}) }  C_{\alpha} c_k \int_{\B_c(x, \frac{\|x\|_c}{2})} \frac{1}{(1+\|z\|_c^k)}\frac{1}{ \| x^{-1} z \|_{c}^{Q-\alpha}} dz \\
& =\frac{1}{\Gamma(\frac{\alpha}{2}) } C_{\alpha} c_k \int_{\B_c(x, \frac{\|x\|_c}{2})} \frac{1}{(1+\|xx^{-1}z\|_c^k)}\frac{1}{ \| x^{-1} z \|_{c}^{Q-\alpha}} dz \\
& \leq \frac{1}{\Gamma(\frac{\alpha}{2}) } C_{\alpha} c_k \int_{\B_c(x, \frac{\|x\|_c}{2})} \frac{1}{(1+(\|x\|_c-\|x^{-1}z\|_c)^k)}\frac{1}{ \| x^{-1} z \|_{c}^{Q-\alpha}} dz\\
& \leq \frac{1}{\Gamma(\frac{\alpha}{2}) } C_{\alpha} c_k \frac{1}{(1+(\frac{\|x\|_c}{2})^k)} \int_{\B_c(x, \frac{\|x\|_c}{2})} \frac{1}{ \| x^{-1} z \|_{c}^{Q-\alpha}} dz \\
& \leq \frac{1}{\Gamma(\frac{\alpha}{2}) } C_{\alpha} c_k \frac{1}{(1+(\frac{\|x\|_c}{2})^k)} \mathcal{S}^{Q-1}_{\infty}(\partial \B_c(0,1))\int_{0}^{\frac{\|x\|_c}{2}} r^{\alpha-Q} r^{Q-1} dz \\
& = \frac{1}{\Gamma(\frac{\alpha}{2}) } C_{\alpha} c_k \frac{1}{(1+(\frac{\|x\|_c}{2})^k)} \mathcal{S}^{Q-1}_{\infty}(\partial \B_c(0,1)) \frac{(\frac{\|x\|_c}{2})^{\alpha}}{\alpha}\\
& \leq \frac{1}{\Gamma(\frac{\alpha}{2}) } C_{\alpha} c_k 2^k \|x\|_c^{-k} \mathcal{S}^{Q-1}_{\infty}(\partial \B_c(0,1)) \frac{\|x\|_c^{\alpha}}{2^{\alpha}\alpha}\\
& \leq \tilde{c}(k, \alpha) \|x \|_c^{\alpha-k},
\end{align*}
where $\tilde{c}(k,\alpha)$ is a suitable positive constant. Hence, we have proved that for every $k$ we have
\begin{equation}
\label{due}
\frac{1}{\Gamma(\frac{\alpha}{2}) }  C_{\alpha} c_k \int_{\B_c(x, \frac{\|x\|_c}{2})} \frac{1}{(1+\|z\|_c^k)}\frac{1}{ \| x^{-1} z \|_{c}^{Q-\alpha}} dz \leq \tilde{c}(k, \alpha) \|x \|_c^{\alpha-k}.
\end{equation}
By combining \eqref{uno} and \eqref{due}, we get that for every $k>Q$ the estimate
\begin{equation*}
 \frac{1}{\Gamma(\frac{\alpha}{2}) } \int_{\G} \phi(xy) \| y \|_{\alpha}^{\alpha-Q} dy \leq c(k,\alpha) \| x \|_c^{\alpha-Q}+ \tilde{c}(k, \alpha) \|x \|_c^{\alpha-k} \leq C(\alpha, k) \| x \|_c^{\alpha-Q} 
\end{equation*}
holds, where $C(k,\alpha)$ is a suitable positive constant, concluding the proof.
\end{proof}
\begin{prop}
\label{est-2}
For every $-2 < \alpha <0$ and $\phi \in \mathcal{S}(\G)$
\begin{equation}
\label{thesis}
\psi(x, \alpha) =O(\|x \|_c^{\alpha-Q})
\end{equation}
as $\| x \|_c \to \infty$.
\end{prop}
\begin{proof}
We can use the representation in \eqref{representation-2}, i.e. for $\alpha \in (-2,0)$
$$ \psi(x, \alpha) = \int_{\G} (\phi(xy)-\phi(x)) P_{\alpha}(y) dy.$$
Let $b=b(\G)>0$ the constant which appears in Theorem \ref{Taylor} and consider
\begin{align}
\label{threeterms}
&\int_{\G} (\phi(xy)-\phi(x)) P_{\alpha}(y) dy  \notag \\
=& \int_{\G \setminus \B_c(0, \frac{\| x\|_c}{2b^2})} (\phi(xy)-\phi(x)) P_{\alpha}(y) dy+\int_{\B_c(0, \frac{\| x\|_c}{2b^2})} (\phi(xy)-\phi(x)) P_{\alpha}(y) dy\\
=&\int_{\G \setminus \B_c(0, \frac{\| x\|_c}{2b^2})} \phi(xy) P_{\alpha}(y) dy -\int_{\G \setminus \B_c(0, \frac{\| x\|_c}{2b^2})} \phi(x) P_{\alpha}(y) dy +\int_{\B_c(0, \frac{\| x\|_c}{2b^2})} (\phi(xy)-\phi(x)) P_{\alpha}(y) dy.  \notag
\end{align}
We deal with the three terms of the sum in \eqref{threeterms}, separately. Let us consider the first one, in particular
\begin{align*}
\int_{\G \setminus \B_c(0, \frac{\| x\|_c}{2b^2})} \phi(xy) P_{\alpha}(y) dy &= \frac{1}{\Gamma(\frac{\alpha}{2})} \int_{\G \setminus \B_c(0, \frac{\| x\|_c}{2b^2})} \phi(xy) \| y \|_{\alpha}^{\alpha-Q} dy\\
&= \frac{1}{\Gamma(\frac{\alpha}{2})} \int_{\G \setminus \B_c(x, \frac{\| x\|_c}{2b^2})} \phi(z) \| x^{-1}z \|_{\alpha}^{\alpha-Q} dz \\
&= \frac{1}{\Gamma(\frac{\alpha}{2})} \int_{\G \setminus \B_c(x, \frac{\| x\|_c}{2b^2})} \phi(z) \frac{(1+\|z \|_c^k)}{(1+\|z \|_c^k)} \| x^{-1}z \|_{\alpha}^{\alpha-Q} dz \\
&\leq \frac{1}{|\Gamma(\frac{\alpha}{2})|} c_k C_{\alpha} \int_{\G \setminus \B_c(x, \frac{\| x\|_c}{2b^2})}  \frac{1}{(1+\|z \|_c^k)} \frac{1}{\| x^{-1}z \|_{c}^{Q-\alpha}} dz ,
\end{align*}
holds for some positive constant $C_{\alpha}$.
Now, as in \eqref{same-comp} (since $\alpha<Q$), we get the estimate in \eqref{uno}, i.e. for every $k>Q$ the estimate
\begin{equation}
\label{an}
\begin{split}
\int_{\G \setminus \B_c(0, \frac{\| x\|_c}{2b^2})} \phi(xy) P_{\alpha}(y) dy\leq \frac{1}{|\Gamma(\frac{\alpha}{2})|} c_k c_{\alpha} \int_{\G \setminus \B_c(x, \frac{\| x\|_c}{2b^2})}  \frac{1}{(1+\|z \|_c^k)} \frac{1}{\| x^{-1}z \|_{c}^{Q-\alpha}} dz \leq  c(k,\alpha) \| x \|_c^{\alpha-Q}
\end{split}
\end{equation}
holds.
Let us consider the second term of \eqref{threeterms}. By the coarea formula of Corollary \ref{cor-coarea} and homogeneity properties we obtain
\begin{align*}
-\int_{\G \setminus \B_c(0, \frac{\| x\|_c}{2b^2})} \phi(x) P_{\alpha}(y) dy &=- \phi(x) \int_{\frac{\|x \|_c}{2b^2}}^{\infty} \int_{\partial \B_c(0,r)} P_{\alpha}(y) d\mathcal{S}^{Q-1}_{\infty}(y) dr \\
&=- \phi(x) \int_{\frac{\|x \|_c}{2b^2}}^{\infty} r^{\alpha-Q} r^{Q-1} \int_{\partial \B_c(0,1)} P_{\alpha}(z) d\mathcal{S}^{Q-1}_{\infty}(z) dr \\
&= \phi(x) \frac{1}{\Gamma(\frac{\alpha}{2})} \frac{1}{\alpha} \sigma(\alpha) \frac{\|x\|_c^{\alpha}}{(2b^2)^{\alpha}}.
\end{align*}
We point out that we used that $\alpha \in (-2,0)$. Now, since $\phi(x)=O( \|x\|_c^N )$ as $\|x \|_c \to \infty$ for every negative integer $N $, fixing $N=-Q$ we get that for $\|x\|_c>1$ the estimate
\begin{equation}
\label{do}
\begin{split}
-\int_{\G \setminus \B_c(0, \frac{\| x\|_c}{2b^2})} \phi(x) P_{\alpha}(y) dy \leq \widehat{c}(\alpha)\|x \|_c^{\alpha-Q} \phi(x)
\end{split}
\end{equation}
holds. By a change of variables in the third term of the sum in \eqref{threeterms}, we obtain
\begin{equation}\label{continue2}
\begin{split}
\int_{\B_c(0, \frac{\| x\|_c}{2b^2})} (\phi(xy)-\phi(x)) P_{\alpha}(y) dy = \int_{\B_c(x, \frac{\| x\|_c}{2b^2})} (\phi(z)-\phi(x)) P_{\alpha}(x^{-1}z) dz.
\end{split}
\end{equation}
Recalling the stratified mean value inequality of Theorem \ref{Taylor} and the fact that $\phi$ is smooth, we have the following inequality
\begin{equation}
| \phi(xy)-\phi(x)-D_h\phi(x)(y) |=| \phi(xy)-\mathcal{T}_1(\phi,x)(xy) | \leq c \|y\|_c^2 \sup_{p \in \B_c(x,b^2\|y\|_c)} \mathfrak{d}^2_h\phi(p),
\end{equation}
where $ \mathfrak{d}^2_h\phi$ is $\mathfrak{d}_h^2\phi(z):\G \to \R, z \to \max_{|\alpha|=2} |(Z_1 \ldots Z_{m_1})^{\alpha}\phi(z)|$ for every $z \in \G$. We get
\begin{equation}\label{oss1}
| \phi(z)-\phi(x)-D_h\phi(x)(x^{-1}z) | \leq c \|x^{-1}z\|_c^2 \sup_{p \in \B_c(x,b^2\|x^{-1}z\|_c)} \mathfrak{d}^2_h\phi(p)=c \|x^{-1}z\|_c^2  \mathfrak{d}^2_h\phi(\xi),
\end{equation}
for some $\xi=\xi(z) \in \B_c(x,b^2\|x^{-1}z\|_c ) \subset \B_c(x, \frac{\|x\|_c}{2})$.
Now, considering the Pansu differential $D_h\phi(x)$ of $\phi$ at $x$, we notice that the following equality
\begin{equation}\label{oss2}
\begin{split}
&\int_{\B_c(x, \frac{\| x\|_c}{2b^2})} 
  D_h\phi(x)(x^{-1}z)  P_{\alpha}(x^{-1}z) dz=\int_{\B_c(0, \frac{\| x\|_c}{2b^2})}  D_h\phi(x)(y) P_{\alpha}(y) dy =0
\end{split}
\end{equation}
holds, by the same argument in \eqref{rel-palla-diffe}.
Hence, from \eqref{continue2}, exploiting \eqref{relcc-alfa}, \eqref{oss1} and \eqref{oss2} and the fact that $\phi \in \mathcal{S}(\G)$, we have
\begin{align}
\label{continue3}
&\int_{\B_c(x, \frac{\| x\|_c}{2b^2})} (\phi(z)-\phi(x)) P_{\alpha}(x^{-1}z) dz  \notag\\
= &\int_{\B_c(x, \frac{\| x\|_c}{2b^2})} (\phi(z)-\phi(x)-D_h \phi(x)(x^{-1}z)) P_{\alpha}(x^{-1}z) dz  \notag \\
 \leq &\int_{\B_c(x, \frac{\| x\|_c}{2b^2})} c \|x^{-1}z\|_c^2 \mathfrak{d}^2_h\phi(\xi) |P_{\alpha}(x^{-1}z)| dz  \notag \\
 \leq & \frac{1}{|\Gamma(\frac{\alpha}{2})|} c C_{\alpha} \int_{\B_c(x, \frac{\| x\|_c}{2b^2})}   \mathfrak{d}^2_h\phi(\xi) \|x^{-1}z\|_c^{\alpha-Q+2} dz\\
= & \frac{1}{|\Gamma(\frac{\alpha}{2})|}c C_{\alpha} \int_{\B_c(x, \frac{\| x\|_c}{2b^2})} \mathfrak{d}^2_h\phi(\xi) \frac{(1+ \| \xi \|_c^k)}{(1+ \| \xi \|_c^{k})} \|x^{-1}z\|_c^{\alpha-Q+2} dz  \notag \\
\leq & \frac{1}{|\Gamma(\frac{\alpha}{2})|}c C_{\alpha} c_k \int_{\B_c(x, \frac{\| x\|_c}{2b^2})} \frac{1}{(1+ \| \xi \|_c^{k})} \|x^{-1}z\|_c^{\alpha-Q+2} dz,  \notag
\end{align}
where $c$ and $c_{\alpha}$ are suitable positive constants.
Now, remarking that $$\|  \xi \|_c \geq \|x\|_c-\|x^{-1}\xi\|_c \geq \|x\|_c - b^{2} \|x^{-1}z\|_c \geq \|x\|_c-b^2 \frac{\|x\|_c}{2b^2}= \frac{\|x\|_c}{2}, $$ we can continue from \eqref{continue3} getting
\begin{align*}
&\int_{\B_c(x, \frac{\| x\|_c}{2b^2})} (\phi(z)-\phi(x)) P_{\alpha}(x^{-1}z) dz \\
 \leq & \ \frac{1}{|\Gamma(\frac{\alpha}{2})|}c C_{\alpha} c_k \int_{\B_c(x, \frac{\| x\|_c}{2b^2})} \frac{1}{(1+ \| \xi \|_c^{k})} \|x^{-1}z\|_c^{\alpha-Q+2} dz\\
\leq & \ \frac{1}{|\Gamma(\frac{\alpha}{2})|}c C_{\alpha} c_k \int_{\B_c(x, \frac{\| x\|_c}{2b^2})} \frac{1}{(1+ \big( \frac{\|x\|_c}{2} \big)^{k})} \frac{1}{\|x^{-1}z\|_c^{Q-\alpha-2}} dz\\
 \leq & \ \frac{1}{|\Gamma(\frac{\alpha}{2})|} c C_{\alpha} c_k \frac{1}{(1+ \big( \frac{\|x\|_c}{2} \big)^{k})} \mathcal{S}^{Q-1}_{\infty}(\partial \B_c(0,1)) \int_{0}^{\frac{\|x\|_c}{2}} r^{\alpha+2-Q+Q-1} dr\\
 = & \ \frac{1}{|\Gamma(\frac{\alpha}{2})|}c C_{\alpha} c_k \frac{1}{(1+ \big( \frac{\|x\|_c}{2} \big)^{k})} \mathcal{S}^{Q-1}_{\infty}(\partial \B_c(0,1)) \frac{\|x\|_c^{\alpha+2}}{2^{\alpha+2}(\alpha+2)}\\
 \leq & \ c(\alpha, k) \| x\|_c^{\alpha+2-k},
\end{align*}
where $c(\alpha,k)$ is a suitable positive constant.

Thus, by choosing $k>Q+2$ we get the following estimate
\begin{equation}
\label{trua}
\begin{split}
\int_{\B_c(0, \frac{\| x\|_c}{2})} (\phi(xy)-\phi(x)) P_{\alpha}(y) \leq \widetilde{c}(\alpha) \| x\|_c^{\alpha-Q},
\end{split}
\end{equation}
for a suitable constant $\widetilde{c}(\alpha)$.
Combining \eqref{an}, \eqref{do} and \eqref{trua} we obtain \eqref{thesis}.
\end{proof}
\subsection{Some properties of $\mathcal{L}^s$, when $-\frac{Q}{2}<s<1$}
\label{section-laplesteso}
Coherently with the continuation constructed in the previous subsections, we show that the definition of the fractional sub-Laplacian $\mathcal{L}^s$ with argument $s \in (-\frac{Q}{2},1)$, that we introduced in Definition \ref{def:fralap} is consistent.\\
First, we observe that for $s=0$, it holds that
$$\psi_{\phi}(x,0)= \phi(x)= \mathcal{L}^0\phi(x),$$
for every $\phi \in \mathcal{S}(\G)$ and $x \in \G$.\\
Moreover, $\mathcal{L}^s$ satisfies the semigroup property stated in the following proposition.

\begin{prop} 
For every $s,m \in (-\frac{Q}{2},1)$ such that $-\frac{Q}{2}<m+s<1 $ the equality
$$\mathcal{L}^s \circ \mathcal{L}^m=\mathcal{L}^{s+m}$$
holds.
\end{prop}

\begin{proof}
Let us consider $\alpha \in (0,Q)$ and $ \beta \in (0,Q)$ such that $\alpha+\beta<Q$. We recall that by \eqref{convolution}, the convolution rule $$P_{\alpha+\beta}=P_{\alpha}\star P_{\beta}$$
holds. For every $x \in \G$ and $\phi \in \mathcal{S}(\G)$, exploiting the relation \eqref{peralfamag0}, we obtain that
\begin{align}\label{assumption}
\widetilde{P}_{\alpha+\beta}^x(\phi)&=P_{\alpha+\beta}(\phi  \circ l_x )= \int_{\G} P_{\alpha+\beta}(y)\phi(xy)dy \notag \\
&= \int_{\G} P_{\alpha}\star P_{\beta}(y)\phi(xy)dy = \int_{\G} \int_{\G} P_{\alpha}(z) P_{\beta}(z^{-1}y) dz \phi(xy)dy \notag \\
&= \int_{\G } P_{\alpha}(z)  \int_{\G} P_{\beta}(z^{-1}y) \phi(xy)dy dz = \int_{\G } P_{\alpha}(z)  \int_{\G} P_{\beta}(z') \phi(xzz')dz' dz\\
&= \int_{\G } P_{\alpha}(z) \psi_{\phi}(xz,\beta) dz =\widetilde{P}_{\alpha}^x(\psi_{\phi}(\cdot,\beta)) =\widetilde{P}_{\alpha}^x(\widetilde{P}_{\beta}^{xz}(\phi)) = \psi_{\psi_{\phi}(\cdot,\beta)}(x,\alpha).  \notag 
\end{align}
We wish to show that the equality in \eqref{assumption} holds for every $\alpha, \beta \in (-2,Q)$ such that $-2<\alpha+\beta<Q$. This conclude the proof by the uniqueness of analytic continuation.\\
First, we observe that, as a consequence of our analytic continuation, the map $\widetilde{P}_{\alpha+\beta}^x(\phi)=\psi_{\phi}(x,\alpha+\beta)$ is analytic with respect to $\alpha$ and $\beta$ on the set of $\alpha$ and $\beta$ such that $-2<\alpha+\beta<Q$. \\
Let us show that the map $\psi_{\psi_{\phi}(\cdot,\beta)}(x,\alpha)$ is analytic as well with respect to $\alpha$ and $\beta$ for $\alpha, \beta \in (-2,Q)$, $\alpha+\beta<Q$.\\
In order to do this, we fix $\beta \in(-2,Q)$ and consider $\alpha \to \psi_{\psi_{\phi}(\cdot,\beta)}(x,\alpha)$.
We need to apply the results of Proposition \ref{prop-psi} and Section \ref{section-psi}. In fact we proved that, if one considers a map $\phi \in C^{\infty}(\G)$ such that $|\phi(z)|\leq K \|z\|_c^{-L}$ when $\|z\|_c \geq S$ for some $L,S,K>0$, then one can carry out the analytic continuation of the map $\psi_{\phi}(x, \alpha)$ for any $\alpha<\min\{L,Q\}$.
We apply this result to the map $\alpha \to \psi_{\psi(\cdot,\beta)}(x,\alpha)$. In fact, by the estimates of Section \ref{estimates-psi} we know that for every $\phi \in \mathcal{S}(\G)$ we have $\psi_{\phi}(z,\beta)=O(\|z\|_c^{\beta-Q})$ as $\|z\|_c \to \infty$. \\
Hence, by Proposition \ref{prop-psi}, we conclude that $\alpha \to \psi_{\psi_{\phi}(\cdot,\beta)}(x,\alpha)$ is analytic up to $-2<\alpha<\min\{Q-\beta,Q\}$, i.e. for every $\alpha \in (-2,Q)$ such that $\alpha+\beta<Q$. \\
Clearly, the role of $\alpha$ and $\beta$ can be exchanged since $\widetilde{P}_{\alpha+\beta}^x(\phi)=\widetilde{P}_{\beta+\alpha}^x(\phi)$, and then $\psi_{\psi_{\phi}(\cdot,\beta)}(x,\alpha)=\psi_{\psi_{\phi}(\cdot,\alpha)}(x,\beta)$. Thus, for every $\alpha \in (-2,Q)$, the map $\beta \to \psi_{\psi(\cdot,\beta)}(x,\alpha)$ is analytic for $\beta \in (-2,Q)$ as $\alpha+\beta <Q$. To sum up, we proved that for every $x \in \G$ and $\phi \in \mathcal{S}(\G)$, 
\begin{equation}
\psi_{\phi}(x,\alpha+\beta)=\widetilde{P}_{\alpha+\beta}^x(\phi)=\psi_{\psi_{\phi}(\cdot,\beta)}(x,\alpha)
\end{equation}
hold for every $\alpha, \beta \in (-2,Q)$ such that $-2<\alpha+\beta<Q$.\\
Let us now consider $x \in \G$ and $\phi \in \mathcal{S}(\G)$ and set $\alpha, \beta \in (-2,Q)$ such that $-2<\alpha+\beta<Q$. Set $\alpha=-2s$ and $\beta=-2m$ for suitable $s,m \in (-\frac{Q}{2},1)$. We have
\begin{align*}
\mathcal{L}^{s+m} \phi(x)&=\widetilde{P}^{x}_{\alpha+\beta}(\phi)=\psi_{\psi_{\phi}(\cdot,\beta)}(x,\alpha)=\psi_{\mathcal{L}^m\phi(\cdot)}(x,\alpha)= \mathcal{L}^s (\mathcal{L}^m \phi)(x).
\end{align*}
This concludes the proof.
\end{proof}

\section{The Heisenberg group $\H^n$}
\label{sect:three}
In this section we recall the setting of the Heisenberg group $\H^n$ and we analyse some suitable families of integrals. This analysis will be an auxiliary tool to construct an analytic continuation of the map $\alpha \to \psi(x,\alpha)$ on $(-\infty, Q)$ in the Heisenberg group.  
\subsection{The Heisenberg group: definition and some properties}
We discuss the Heisenberg group according to Section \ref{sect:prel}, referring to the representation of a Carnot group $\G$ in \eqref{eq:cg}. For every $n \in \N$, independently of the choice of $n$, the Heisenberg group can be represented as a vector space
\[
\mathbb{H}^n = H_1 \oplus H_2,
\]
with dim($H_1$)$=2n$ and dim($H_2$)$=1$, endowed with a symplectic form $\omega$ on $H_1$ and a fixed nonvanishing element \begin{equation}
\label{edueennepiuuno}
e_{2n+1} \in H_2.
\end{equation}
We denote by $\pi_{H_1}$ and $\pi_{H_2}$ the canonical projections on $H_1$ and $H_2$, respectively, associated with the direct sum.
The space $\mathbb{H}^n$ has a structure of Lie algebra given by 
\begin{equation}
[ x,y]= \omega(\pi_{H_1}(x),\pi_{H_1}(y)) \ e_{2n+1}.
\end{equation}
Then, by the Baker-Campbell-Hausdorff formula, the Lie group product is
\begin{equation}\label{lieprod}
x y = x+ y + \frac{[x,y]}{2}
\end{equation}
for every $x,y \in \H^n$.
Notice that $q=2n+1$ and $Q=2n+2$.
We fix a symplectic basis $ (e_1, \dots, e_{2n} )$ of $(H_1,\omega)$, namely
\[
\omega(e_i,e_{n+j})=\delta_{ij}, \qquad \omega(e_i,e_j)=\omega(e_{n+i},e_{n+j})=0,
\]
for every $i,j=1,\ldots,n$, where $\delta_{ij}$ is the Kronecker delta.
Thus, considering the vector $e_{2n+1}$ given in \eqref{edueennepiuuno}, we obtain a basis
\begin{equation}
\label{heisbasis}
\mathcal{B}= ( e_1, \dots, e_{2n+1} ).
\end{equation}
We assume that on $\H^n$ is fixed the scalar product $\langle \cdot, \cdot \rangle$ that makes $\mathcal{B}$ orthonormal.
We can read the product \eqref{lieprod} on $\H^n$ in coordinates with respect to $\cB$ (i.e.\ on $\mathbb{R}^{2n+1}$) as follows:
\begin{small}
\begin{equation}
\label{productheis}
(x_1, \dots, x_{2n+1})  (y_1, \dots, y_{2n+1})= \pa{x_1 + y_1, \ldots, x_{2n+1} + y_{2n+1} + \sum_{i=1}^n \frac{x_iy_{i+n}-x_{i+n} y_i}2},
\end{equation}
\end{small}
for every $(x_1, \ldots, x_{2n+1}), (y_1, \ldots, y_{2n+1}) \in \H^n$.
Notice that $H_2= \mathrm{span}(e_{2n+1})$ is the center of $\H^n$.
Moreover, taking in consideration the product \eqref{productheis}, in our coordinates the Jacobian basis of left invariant vector fields is the following
\begin{equation}
\label{eqbase}
\begin{aligned}
 X_j(x) &= \partial_{x_j}-\frac{1}{2} x_{j+n} \partial_{x_{2n+1}} \ \ \ \ \ \ j=1, \dots, n\\
 Y_j(x) &= \partial_{x_{n+j}} + \frac{1}{2}x_j \partial_{x_{2n+1}} \ \ \ \ \ \ j=1, \dots, n\\
 T(x)&= \partial_{x_{2n+1}},
\end{aligned}
\end{equation}
for every $x=(x_1, \dots, x_{2n+1}) \in \H^n$.
We will also use the following notation: for $i=1, \ldots, 2n+1$ we denote
$$\mathrm{Lie}(\H^n) \ni Z_i:=
\begin{cases}
X_i \qquad \mathrm{if \ }i=1 \ldots n\\
Y_{i-n} \qquad \mathrm{if \ }i=n+1, \ldots, 2n\\
T \qquad \mathrm{if \ }i=2n+1.\\
\end{cases}
$$
With a slight abuse of notation, we denote by $\pi_{H_1}$ and $\pi_{H_2}$ also the corresponding projections read in coordinates with respect to the basis $\mathcal{B}$ i.e.
\begin{align*}
&\pi_{H_1}:\H^n \to \R^{2n},
\pi_{H_1}(x_1, \ldots, x_{2n+1})=(x_1, \ldots, x_{2n}),\\
& \pi_{H_2}:\H^n \to \R, \pi_{H_2}(x_1, \ldots, x_{2n+1})=x_{2n+1},
\end{align*}
for every $(x_1, \ldots, x_{2n+1}) \in \H^n$.\\

The structure of the Lie algebra of $\H^n$ reflects on the explicit form of the $Z$-Taylor polynomials. For example, we consider below the explicit form of the $Z$-Taylor Polynomial of $\G$-degree $n=3$ related to a map $\phi \in C^{\infty}(\H^n)$ in the Heisenberg group $\H^n$. Let $x \in \H^n$, then for every $y \in \H^n$, according to Theorem \ref{explicit}, we have
\begin{equation}\label{Taylor3}
\begin{split}
\mathcal{T}_3(\phi,x)(xy)=&\phi(x)+\sum_{i=1}^{2n}Z_i\phi(x)y_i+ T\phi(x)y_{2n+1}+ \frac{1}{2}\sum_{i,j=1}^{2n} Z_iZ_j\phi(x)y_iy_j\\
&+\frac{1}{2}\sum_{i=1}^{2n}Z_i T\phi(x)y_iy_{2n+1}+\frac{1}{6}\sum_{i,j,k=1}^{2n} Z_iZ_jZ_k\phi(x)y_iy_jy_k.
\end{split}
\end{equation}
Moreover, by applying Theorem \ref{Taylor}, we get the existence of constants $c=c(\H^n)>0$ and $b=b(\H^n)>0$ such that 
\begin{equation}
\label{utile}
|\phi(xy)-\mathcal{T}_3(\phi,x)(xy)| \leq c \|y\|_c^{4} \sup_{\{\|z\|_c \leq b^{4} \|y\|_c\}} \sup_{\{|\beta|=4\}} |(Z_1 \cdots Z_{2n})^{\beta} \phi(xz)|,
\end{equation} 
for every $x,y \in \G$. Here $\beta$ denotes a $2n$-multi-index.\\

We recall also some metric properties on $\H^n$. Referring to Theorem \ref{Eik}, if $\G=\H^n$, then the Carnot-Carath\'eodory distance $d_c(x,0) \in C^1(\H^n \setminus H_2)$ (see for example \cite[Lemma 3.11]{AR02}) and $ | \nabla_H d_c(x, 0)| =1$ for every $x \in \H^n \setminus H_2$.\\

In the following two subsections, we introduce some properties of $h$ and $ \B_c(0,1)$ in $\H^n$.
\subsection{Invariance of the heat kernel under horizontal rotations}
\label{parita-h}
In many papers, like \cite{Gaveau75, Hul76, Gaveau77, Randall96},
the authors studied the heat kernel $h$ in the Heisenberg group $\H^n$ and, in general, in stratified groups of step 2. 
In particular, in the Heisenberg group $\H^n$, for every $t>0$ and $x=(x_1, \dots, x_{2n}, x_{2n+1}) \in \H^n$, the heat kernel $h$ depends on $t$, on the last component $x_{2n+1}$ of $x$ and on the norm $|(x_1, \dots, x_{2n})|=|\pi_{H_1}(x)|$.
This allows to deduce that the map $h$ is invariant under the action of horizontal rotations, i.e. the rotations which do not involve the $x_{2n+1}$-axis $H_2$.\\
As a consequence, for every $t>0$, for every $x=(x_1, \dots, x_{2n+1}) \in \H^n$ and for every $j \in \{1, \dots, 2n \}$ the equality
\begin{equation}
\label{parita-heat}
h(t,(x_1, \dots, x_j, \dots, x_{2n}, x_{2n+1}))=h(t,(x_1, \dots, -x_j, \dots, x_{2n}, x_{2n+1}))
\end{equation}
holds.
Clearly, \eqref{parita-heat} yields that for every $\alpha<Q$, for every $t>0$ and for every $x=(x_1, \dots, x_{2n+1}) \in \H^n$ the following equalities
\begin{align}
\label{parita-kernel}
\|(x_1, \dots, x_j, \dots, x_{2n}, x_{2n+1})\|_{\alpha}^{\alpha-Q}&=\|(x_1, \dots, -x_j, \dots, x_{2n}, x_{2n+1})\|_{\alpha}^{\alpha-Q}, \notag \\
P_{\alpha}(x_1, \dots, x_j, \dots, x_{2n}, x_{2n+1})&=P_{\alpha}(x_1, \dots, -x_j, \dots, x_{2n}, x_{2n+1})
\end{align}
hold for every $j \in \{1, \dots, 2n \}$.
We notice that, a priori, this properties are not proved in arbitrary stratified groups, included the stratified groups of step 2. For the explicit form of the heat kernel we address the reader to \cite{Cygan79}. For further details about this topic, we point out \cite[Remark 2.12]{BMP12}.

\subsection{Invariance of $\B_c(0,1)$ under horizontal rotations}
\label{parita-ball}
In the Heisenberg group the Carnot-Carath\'eodory ball $\B_c(0,R)$, of radius $R$, for $R>0$, centered at $0 \in \H^n$ is invariant under horizontal rotations, i.e. the rotations preserving the center of the group $H_2$ (see for instance \cite[Section 2.3]{LRV12}).\\
As a consequence, let $(x_1, \dots, x_{2n}, x_{2n+1}) \in \B_c(0,R)$. Then, for every point $(y_1, \dots, y_{2n}, x_{2n+1}) \in \H^n$ such that the following equality
$$| (x_1, \dots, x_{2n})|=|(y_1, \dots, y_{2n})|$$
holds, then $(y_1, \ldots, y_{2n}, x_{2n+1}) \in \B_c(0,R)$.
Let us now introduce for every $j \in \{1, \dots, 2n \}$ the following projection
\begin{equation}\label{pij}
\pi_j: H_1 \to \R^{2n-1}, \pi_j\left( \sum_{i=1}^{2n}x_ie_i \right)=(x_1, \dots, x_{j-1}, x_{j+1}, \dots, x_{2n}, x_{2n}).
\end{equation}
Notice that $\pi_{H_2}$ and $\pi_j$, for $j= \{1, \dots, 2n \}$, are Lipschitz map in the Euclidean sense and their (Euclidean) Jacobian is unitary. Hence, for every $r \in \pi_{H_2}(\B_c(0,1)) \subset \R$ and for every $s=(s_1, \dots, s_{2n-1}) \in \pi_j (\pi_{H_2}^{-1}(r) \cap \B_c(0,1)) \subset \R^{2n-1}$, the invariance of $\B_c(0,R)$ with respect to horizontal rotations implies that
the set \begin{equation}
\label{simmetriai}
 \pi_j^{-1}(s) \cap \pi_{H_2}^{-1}(r) \cap \B_c(0,R) 
\end{equation}
is symmetric in the $x_j$-direction with respect to the point $(s_1, \dots, s_{j-1}, 0, s_j, \dots, s_{2n-1},r)$. \\
More precisely, for every $r \in \R$, every $s=(s_1, \dots, s_{2n-1}) \in \R^{2n-1}$ and for every $x_j \in \R$:
\begin{equation*}
(s_1, \dots, s_{j-1}, x_j, s_j, \dots, s_{2n-1},r) \in \B_c(0,R) \mathrm{ \ if \ and \ only \ if \ }(s_1, \dots, s_{j-1}, -x_j, s_j, \dots, s_{2n-1},r) \in \B_c(0,R).
\end{equation*}
In fact, the previous equivalence follows by $$|(s_1, \dots, s_{j-1}, x_j, s_j, \dots, s_{2n-1})|=|(s_1, \dots, s_{j-1}, -x_j, s_j, \dots, s_{2n-1})|,$$ combined with the fact that $\B_c(0,R)$ is invariant under horizontal rotations.\\
In addition, also $\partial \B_c(0,1)$ is invariant under horizontal rotations.
\subsection{Auxiliary tools}
In this subsection we study some specific families of integrals. In particular, we collect some technical results which will be used later on in the paper.
First, we fix $m \in \N$, $\alpha<Q$ and we collect few useful representations of the following integral 
\begin{equation}
\label{diff-taylor}
\int_{\B_c(0,1)}\big(\mathcal{T}_{2m+1}(\phi,x)(xy)-\mathcal{T}_{2m-1}(\phi,x)(xy) \big) \|y\|_{\alpha}^{\alpha-Q} dy.
\end{equation}
We recall that the object $\mathcal{T}_n(\phi,x)$ denotes the $Z$-Taylor Polynomial of $\G$-degree $n$ related to $\phi$ centered at $x$, for every $n \in \N$. 
\subsubsection{Representation, I}
\label{sect:441}
First of all, we exploit the explicit representations of $\mathcal{T}_{2m+1}(\phi,x)(xy)$ and $\mathcal{T}_{2m-1}(\phi,x)(xy)$ given by Theorem \ref{explicit} which allows to write the integral in \eqref{diff-taylor} as
\begin{equation}\label{prima}
 \sum_{h=2m}^{2m+1} \sum_{\substack{k=1, \ldots, h\\ i_1, \ldots,i_k \leq 2n+1 \\ \mathrm{deg}(Z_{i_1})+ \ldots \mathrm{deg}(Z_{i_k})=h}} \frac{ Z_{i_1} \ldots Z_{i_k} \phi(x)}{k!}
 \int_{\B_c(0,1)} y_{i_1} \ldots y_{i_k} \|y\|_{\alpha}^{\alpha-Q} dy.
\end{equation}
Now, for $h \in \{ 2m, 2m+1\}$, $k \in \{1, \ldots, h\}$ and  $i_1, \ldots,i_k \leq 2n+1$ such that $  \mathrm{deg}(Z_{i_1})+ \ldots + \mathrm{deg}(Z_{i_k})=h$, we focus on the integral 
\begin{equation}\label{integral}
\int_{\B_c(0,1)} y_{i_1} \ldots y_{i_k} \|y\|_{\alpha}^{\alpha-Q} dy
\end{equation}
appearing in the sum \eqref{prima}.
We observe one more time that $\B_c(0,1)$ and the map $\|y\|_{\alpha}^{\alpha-Q}$ are symmetric and invariant with respect to the group inversion. In fact, 
\begin{itemize}
\item by the homogeneity of $d_c$, if a point $y \in \B_c(0,1)$, then $y^{-1} \in \B_c(0,1)$. \item by the property of the heat kernel $h$ stated in Theorem \ref{T1}(iii), it holds that $$\|y\|_{\alpha}^{\alpha-Q}=\|y^{-1}\|_{\alpha}^{\alpha-Q},$$
for every $y \in \H^n$, $y \neq 0$.
\end{itemize}
Moreover, for every $y \in \H^n$, $y^{-1}=-y$. Therefore, we know that, if the number $k \in \{1, \ldots, h\}$ of the involved coordinates $y_{i_1} \ldots y_{i_k}$ is odd, then the integral in \eqref{integral} vanishes
\begin{equation}
\int_{\B_c(0,1)} y_{i_1} \ldots y_{i_k} \|y\|_{\alpha}^{\alpha-Q} dy=0.
\end{equation}
As a consequence, we can rewrite \eqref{diff-taylor} as follows
\begin{equation}
 \sum_{h=2m}^{2m+1} \sum_{\substack{k=1, \ldots, h, k  \ \mathrm{even}\\ i_1, \ldots,i_k \leq 2n+1 \\ \mathrm{deg}(Z_{i_1})+ \ldots \mathrm{deg}(Z_{i_k})=h}} \frac{ Z_{i_1} \ldots Z_{i_k} \phi(x)}{k!}
 \int_{\B_c(0,1)} y_{i_1} \ldots y_{i_k} \|y\|_{\alpha}^{\alpha-Q} dy.
\end{equation}
Now, we consider again the integral in \eqref{integral}
and we make the following hypothesis
\begin{equation}\label{hyp}
\mathrm{there \ exists \ } j \in \{1, \dots, k\} \mathrm{ \ such \ that \ }\mathrm{deg}(Z_{i_j})=1 \mathrm{ \ and \ }\sharp \{ \ell \in \{1, \dots, k \} : i_j=i_{\ell} \} \mathrm{ \ is \ odd }.
\end{equation}
Now, if we apply the Euclidean coarea formula to the integral in \eqref{integral} we get
\begin{align*}
&\int_{\B_c(0,1)} y_{i_1} \ldots y_{i_k} \| y \|_{\alpha}^{\alpha-Q} dy =\int_{\pi_{H_2}(\B_c(0,1))} \int_{\pi_{H_2}^{-1}(y_{2n+1}) \cap \B_c(0,1)} y_{i_1} \ldots y_{i_k} \| y \|_{\alpha}^{\alpha-Q} dy_1, \dots dy_{2n} \ dy_{2n+1} \\
=&\int_{\pi_{H_2}(\B_c(0,1))} \int_{\pi_j(\pi_{H_2}^{-1}(y_{2n+1}) \cap \B_c(0,1))} \int_{\pi_j^{-1}(r) \cap \pi_{H_2}^{-1}(y_{2n+1}) \cap \B_c(0,1)} y_{i_1} \ldots y_{i_k} \| y \|_{\alpha}^{\alpha-Q} dy_j dr \ dy_{2n+1}.
\end{align*}
By combining the symmetry of $\B_c(0,1)$ discussed in Section \ref{parita-ball}, the symmetry of $\|y\|_{\alpha}^{\alpha-Q}$ discussed in Section \ref{parita-h} and our hypothesis \eqref{hyp}, we conclude that
\begin{align*}
\int_{\pi_j^{-1}(r) \cap \pi_{H_2}^{-1}(y_{2n+1}) \cap \B_c(0,1)} y_{i_1} \ldots y_{i_k} \| y \|_{\alpha}^{\alpha-Q} dy_j=0,
\end{align*}
for every $y_{2n+1} \in \R$ and $r \in \R^{2n-1}$.
Hence, under the hypothesis \eqref{hyp}, the integral in \eqref{integral} vanishes. Then, the integral in \eqref{integral} can be non-zero only if
\begin{equation}\label{coi-multi}
\int_{\B_c(0,1)} y_{i_1} \ldots y_{i_k} \|y\|_{\alpha}^{\alpha-Q} dy= \int_{\B_c(0,1)} (y_1, \dots, y_{2n})^{\gamma} y_{2n+1}^{\gamma_{2n+1}} \|y\|_{\alpha}^{\alpha-Q} dy,
\end{equation}
for some even $2n$-multi-index $\gamma \in (\N \cup \{0\})^{2n}$ and for some $\gamma_{2n+1} \in \N \cup \{0\}$.\\
Let us now consider again the integral in \eqref{integral} and assume that $h=2m+1$. Hence we have that 
\begin{equation}\label{hyp2}
\mathrm{deg}(Z_{i_1})+ \ldots \mathrm{deg}(Z_{i_k})=2m+1.
\end{equation}
Since we are working in $\H^n$, for every $j=1, \dots, k$, necessarily $\mathrm{deg}(Z_{i_j})$ is 1 or 2. Hence 
\begin{equation}\label{dispari} \sharp\{ j \in \{1, \dots, k\}: \ \mathrm{deg} (Z_{i_j})=1\} = 2m+1- 2 \cdot \sharp\{ j \in \{1, \dots, k\}: \ \mathrm{deg} (Z_{i_j})=2\}.
\end{equation}
Thus, $\sharp\{ j \in \{1, \dots, k\}: \ \mathrm{deg} (Z_{i_j})=1\} $ is odd, and therefore hypothesis \eqref{hyp} is verified, hence, also under the hypothesis \eqref{hyp2}, we have
$$\int_{\B_c(0,1)} y_{i_1} \ldots y_{i_k} \|y\|_{\alpha}^{\alpha-Q} dy=0.$$
Therefore, we can conclude that the integral in \eqref{integral} can be non-zero only if $$ \mathrm{deg}(Z_{i_1})+ \ldots \mathrm{deg}(Z_{i_k})=2m.$$ As a consequence, in \eqref{coi-multi} $\gamma_{2n+1}$ has to be even too. Hence, the integral in \eqref{integral} can be non-zero only if
\begin{equation}\label{dariportare}
\int_{\B_c(0,1)} y_{i_1} \ldots y_{i_k} \|y\|_{\alpha}^{\alpha-Q} dy= \int_{\B_c(0,1)} y^{\gamma} \|y\|_{\alpha}^{\alpha-Q} dy,
\end{equation}
for some even $(2n+1)$-multi-index $\gamma \in (\N \cup \{0\})^{2n+1}$.\\
We observe that our argument allows to state the equality 
\begin{equation}\label{comb1}
\begin{split}
&\int_{\B_c(0,1)}\big(\mathcal{T}_{2m+1}(\phi,x)(xy)-\mathcal{T}_{2m-1}(\phi,x)(xy) \big) \|y\|_{\alpha}^{\alpha-Q} dy\\
=&   \sum_{\substack{k=1, \ldots, 2m, k  \ \mathrm{even}\\ i_1, \ldots,i_k \leq 2n+1 \\ \mathrm{deg}(Z_{i_1})+ \ldots \mathrm{deg}(Z_{i_k})=2m}} \frac{ Z_{i_1} \ldots Z_{i_k} \phi(x)}{k!}
 \int_{\B_c(0,1)} y_{i_1} \ldots y_{i_k} \|y\|_{\alpha}^{\alpha-Q} dy.
\end{split}
\end{equation}
\subsubsection{Representation, II}
According to the previous subsection, we fix $m \in \N$, $k\in \{1, \ldots, 2m \}$ and $ i_1, \ldots,i_k \leq 2n+1$ such that 
$$\mathrm{deg}(Z_{i_1})+ \ldots \mathrm{deg}(Z_{i_k})=2m.$$ Applying the coarea formula of Corollary \ref{cor-coarea}, we obtain
$$\int_{\B_c(0,1)} y_{i_1} \ldots y_{i_k} \|y\|_{\alpha}^{\alpha-Q} dy= \int_0^1 \int_{\partial \B_c(0,r)}  y_{i_1} \ldots y_{i_k} \|y\|_{\alpha}^{\alpha-Q}d\mathcal{S}^{Q-1}_{\infty}(y) dr.$$ By the homogeneity of the norm $\| \cdot \|_{\alpha}^{\alpha-Q}$ and the homogeneity of $\mathcal{S}^{Q-1}_{\infty}$, we get for every $-2m<\alpha<Q$ 
\begin{align*}
&\int_0^1 \int_{\partial \B_c(0,r)}  y_{i_1} \ldots y_{i_k} \|y\|_{\alpha}^{\alpha-Q}d\mathcal{S}^{Q-1}_{\infty}(y) dr\\
= & \int_0^1 \int_{\partial \B_c(0,1)}  (\delta_rz)_{i_1} \ldots (\delta_rz)_{i_k} \|\delta_r z\|_{\alpha}^{\alpha-Q} r^{Q-1} d\mathcal{S}^{Q-1}_{\infty}(z) dr\\
= & \int_0^1 \int_{\partial \B_c(0,1)}  r^{\mathrm{deg}(Z_{i_1})} z_{i_1} \ldots r^{ \mathrm{deg}(Z_{i_k})} z_{i_k} r^{\alpha-Q} \| z\|_{\alpha}^{\alpha-Q} r^{Q-1} d \mathcal{S}^{Q-1}_{\infty}(z)dr\\
= & \int_0^1 r^{\alpha-Q+Q-1+\mathrm{deg}(Z_{i_1})+\ldots \mathrm{deg}(Z_{i_k})} \int_{\partial \B_c(0,1)}  z_{i_1} \ldots  z_{i_k}  \| z\|_{\alpha}^{\alpha-Q} d \mathcal{S}^{Q-1}_{\infty}(z) dr\\
= & \int_0^1 r^{\alpha-1+2m} dr \int_{\partial \B_c(0,1)}  z_{i_1} \ldots  z_{i_k}  \| z\|_{\alpha}^{\alpha-Q} d \mathcal{S}^{Q-1}_{\infty}(z)\\
= & \frac{1}{\alpha+2m} \int_{\partial \B_c(0,1)}  z_{i_1} \ldots  z_{i_k}  \| z\|_{\alpha}^{\alpha-Q} d \mathcal{S}^{Q-1}_{\infty}(z).
\end{align*}

Finally, for every $-2m < \alpha<Q$ we can rewrite \eqref{comb1} as 
\begin{align}\label{comb2}
&\int_{\B_c(0,1)}\big(\mathcal{T}_{2m+1}(\phi,x)(xy)-\mathcal{T}_{2m-1}(\phi,x)(xy) \big) \|y\|_{\alpha}^{\alpha-Q} dy  \notag  \\
=&   \sum_{\substack{k=1, \ldots, 2m, k  \ \mathrm{even}\\ i_1, \ldots,i_k \leq 2n+1 \\ \mathrm{deg}(Z_{i_1})+ \ldots \mathrm{deg}(Z_{i_k})=2m}}  \frac{1}{\alpha+2m} \frac{ Z_{i_1} \ldots Z_{i_k} \phi(x)}{k!}
\int_{\partial \B_c(0,1)}  z_{i_1} \ldots  z_{i_k}  \| z\|_{\alpha}^{\alpha-Q} d \mathcal{S}^{Q-1}_{\infty}(z). 
\end{align}

As a consequence, for every $-2m < \alpha<Q$ we also have that 
\begin{align*}
&\int_{\B_c(0,1)}\big(\mathcal{T}_{2m+1}(\phi,x)(xy)-\mathcal{T}_{2m-1}(\phi,x)(xy) \big) P_{\alpha}(y) dy\\
=&   \sum_{\substack{k=1, \ldots, 2m, k  \ \mathrm{even}\\ i_1, \ldots,i_k \leq 2n+1 \\ \mathrm{deg}(Z_{i_1})+ \ldots \mathrm{deg}(Z_{i_k})=2m}} \frac{1}{\alpha+2m}\frac{1}{\Gamma(\frac{\alpha}{2})}\frac{ Z_{i_1} \ldots Z_{i_k} \phi(x)}{k!}
\int_{\partial \B_c(0,1)}  z_{i_1} \ldots  z_{i_k}  \| z\|_{\alpha}^{\alpha-Q} d \mathcal{S}^{Q-1}_{\infty}(z).
\end{align*}

\subsubsection{Regularity}\label{anal-integrals}
Here, we prove that the following map is analytic
\begin{equation}\label{diff-analitica}
\begin{split}
(-\infty,Q) \ni \alpha \to &   \sum_{\substack{k=1, \ldots, 2m, k  \ \mathrm{even}\\ i_1, \ldots,i_k \leq 2n+1 \\ \mathrm{deg}(Z_{i_1})+ \ldots \mathrm{deg}(Z_{i_k})=2m}} \frac{1}{\alpha+2m}\frac{1}{\Gamma(\frac{\alpha}{2})}\frac{ Z_{i_1} \ldots Z_{i_k} \phi(x)}{k!}
\int_{\partial \B_c(0,1)}  z_{i_1} \ldots  z_{i_k}  \| z\|_{\alpha}^{\alpha-Q} d \mathcal{S}^{Q-1}_{\infty}(z).
\end{split}
\end{equation}
Let us first consider the map $\alpha \to \frac{1}{\alpha+2m}\frac{1}{\Gamma(\frac{\alpha}{2})}$. By the standard properties of the Euler gamma function we have
\begin{align*}
&\lim_{\alpha \to -2m} \frac{1}{\alpha+2m}\frac{1}{\Gamma(\frac{\alpha}{2})}\\
=& \lim_{\alpha \to -2m} \frac{1}{\alpha+2m}\frac{1}{\Gamma(\frac{\alpha+2m}{2})}\left(\frac{\alpha}{2}+(m-1)\right)\left(\frac{\alpha}{2}+(m-2)\right) \ldots \left(\frac{\alpha}{2}\right) = \frac{1}{2}(-1)^m m!,
\end{align*}
then $\alpha \to \frac{1}{\alpha+2m}\frac{1}{\Gamma(\frac{\alpha}{2})}$ is analytic on $\R$.
Indeed, it is enough to show that for every even $k \in \{1, \ldots, 2m\}$ and every $i_1, \ldots i_k \leq 2n+1$ such that $\mathrm{deg}(Z_{i_1})+ \ldots \mathrm{deg}(Z_{i_k}) =2m$, the map
\begin{equation}\label{mappaalfa}
\begin{split}
(-\infty,Q) \ni \alpha \to \int_{\partial \B_c(0,1)}  z_{i_1} \ldots  z_{i_k}  \| z\|_{\alpha}^{\alpha-Q} d \mathcal{S}^{Q-1}_{\infty}(z)
\end{split}
\end{equation}
is analytic. The proof of this claim can be obtained as a generalization of the proof of Item \ref{itemcalfa} in Section \ref{sectioncalfa}, where we showed the analyticity of the map $\sigma(\alpha)$, which in particular equals the map \eqref{mappaalfa} when $m=0$. It is enough to repeat the cited proof combined with the fact that $z \to z_{i_1} \ldots  z_{i_k}$  in \eqref{diff-analitica} is a non-negative continuous map on the compact $\partial{\B_c(0,1)}$. In particular, the representation in \eqref{dariportare} implies the non-negativity of this map. 

\subsubsection{Representation, III}
Here we prove the following lemma.
\begin{lem}\label{lemma-general}
For every even $(2n+1)$-multi-index $\gamma=(\gamma_1, \ldots, \gamma_{2n+1})\in (\N \cup \{ 0 \})^{2n+1}$ and for every $\alpha<Q$ we have
\beq
 \int_{\partial \B_c(0,1)} x^{\gamma}\|x\|_{\alpha}^{\alpha-Q} d \mathcal{S}_{\infty}^{Q-1}(x)= 2\int_{\H^n} x^{\gamma} h(1,x) \|x\|_c^{-\alpha-(\gamma_1+\ldots \gamma_{2n}+2\gamma_{2n+1})} d x.
\eeq
\end{lem}

\begin{proof}
We exploit the explicit form of $\| \cdot \|_{\alpha}$ and the homogeneity property of the heat kernel $h$ getting
\begin{align*}
\int_{\partial \B_c(0,1)} y^{\gamma}\|y\|_{\alpha}^{\alpha-Q} d \mathcal{S}_{\infty}^{Q-1}(y)&=\int_{\partial \B_c(0,1)} y^{\gamma} \int_0^{\infty} t^{\frac{\alpha}{2}-1} h(t,y) d t \  d \mathcal{S}^{Q-1}_{\infty}(y)\\
&= \int_0^{\infty} t^{\frac{\alpha}{2}-1-\frac{Q}{2}} \int_{\partial \B_c(0,1)} y^{\gamma} h(1,\delta_{\frac{1}{\sqrt{t}}}(y)) d \mathcal{S}^{Q-1}_{\infty}(y)d t. 
\end{align*}
Let us now perform the change of variable $r=\frac{1}{\sqrt{t}}$ so that $dt=-2r^{-3}dr$ and we get
\begin{align*}
&\int_0^{\infty} t^{\frac{\alpha}{2}-1-\frac{Q}{2}} \int_{\partial \B_c(0,1)} y^{\gamma}h(1,\delta_{\frac{1}{\sqrt{t}}}(y)) d \mathcal{S}^{Q-1}_{\infty}(y)d t\\
&=-2 \int_{\infty}^0 (r^{-2})^{\frac{\alpha}{2}-1-\frac{Q}{2}} r^{-3} \int_{\partial \B_c(0,1)} y^{\gamma} h(1,\delta_r(y)) d \mathcal{S}^{Q-1}_{\infty}(y)d r.\\
&=2 \int_0^{\infty} r^{-\alpha+Q-1} \int_{\partial \B_c(0,1)} y^{\gamma}h(1,\delta_r(y)) d \mathcal{S}^{Q-1}_{\infty}(y)d r\\
&= 2 \int_0^{\infty} r^{-\alpha+Q-1-(\gamma_1+\ldots + \gamma_{2n}+2\gamma_{2n+1})} \int_{\partial \B_c(0,1)}r^{\gamma_1+\ldots + \gamma_{2n}+2\gamma_{2n+1}} y^{\gamma} h(1,\delta_r(y)) d \mathcal{S}^{Q-1}_{\infty}(y)d r\\
&= 2 \int_0^{\infty} r^{-\alpha+Q-1-(\gamma_1+\ldots + \gamma_{2n}+2\gamma_{2n+1})} \int_{\partial \B_c(0,1)}(\delta_ry)^{\gamma} h(1,\delta_r(y)) d \mathcal{S}^{Q-1}_{\infty}(y)d r.
\end{align*}
Now we apply the change of variables $z=\delta_r(y)$ and taking in account the homogeneity of the metric $d_c$ and of the measure $\mathcal{S}^{Q-1}_{\infty}$, we get
\begin{align*}
&2\int_0^{\infty} r^{-\alpha+Q-1-(\gamma_1+\ldots + \gamma_{2n}+2\gamma_{2n+1})} \int_{\partial \B_c(0,1)} (\delta_r(y))^{\gamma} h(1,\delta_r(y)) d \mathcal{S}^{Q-1}_{\infty}(y)d r\\
=& \ 2\int_0^{\infty} r^{-\alpha-(\gamma_1+\ldots + \gamma_{2n}+2\gamma_{2n+1})} \int_{\partial \B_c(0,r)} z^{\gamma} h(1,z) d \mathcal{S}^{Q-1}_{\infty}(z)d r.
\end{align*}
Therefore, for every $\alpha<Q$ we have
\begin{align*}
&\int_{\partial \B_c(0,1)} y^{\gamma}  \|y\|_{\alpha}^{\alpha-Q} d\mathcal{S}^{Q-1}_{\infty}(y)\\ &=2\int_0^{\infty} r^{-\alpha-(\gamma_1+\ldots + \gamma_{2n}+2\gamma_{2n+1})} \int_{\partial \B_c(0,r)} z^{\gamma} h(1,z) d \mathcal{S}^{Q-1}_{\infty}(z)d r \\
&=2\int_0^{\infty}  \int_{\partial \B_c(0,r)} z^{\gamma} h(1,z) \|z\|_c^{-\alpha-(\gamma_1+\ldots + \gamma_{2n}+2\gamma_{2n+1})} d \mathcal{S}^{Q-1}_{\infty}(z)d r\\
&=2\int_{\H^n} z^{\gamma}  h(1,z) \|z\|_c^{-\alpha-(\gamma_1+\ldots + \gamma_{2n}+2\gamma_{2n+1})} dz,
\end{align*}
where in the last step we exploited the coarea formula of Corollary \ref{cor-coarea}. This concludes the proof.
\end{proof}
We recall that, if we fix $m \in \N$, $k \in \{1, \ldots, 2m\}$ and $ i_1, \ldots,i_k \leq 2n+1$ such that 
$$\mathrm{deg}(Z_{i_1})+ \ldots \mathrm{deg}(Z_{i_k})=2m,$$ as in \eqref{dariportare}, then the existence of some even $(2n+1)$-multi-index $\gamma \in (\N \cup \{0\})^{2n+1}$ such that the following equality holds
\begin{equation}
\int_{\B_c(0,1)} y_{i_1} \ldots y_{i_k} \|y\|_{\alpha}^{\alpha-Q} dy= \int_{\B_c(0,1)} y^{\gamma} \|y\|_{\alpha}^{\alpha-Q} dy
\end{equation}
follows from Section \ref{sect:441}.
Thus, necessarily
$$\mathrm{deg}(Z_{i_1})+ \ldots +\mathrm{deg}(Z_{i_k})=2m = \gamma_1 + \ldots \gamma_{2n}+2 \gamma_{2n+1}.$$ Therefore, exploiting Lemma \ref{lemma-general} we can rewrite \eqref{comb2}  for every $-2m<\alpha<Q$ as follows:
\begin{align*}
&\int_{\B_c(0,1)}\big(\mathcal{T}_{2m+1}(\phi,x)(xy)-\mathcal{T}_{2m-1}(\phi,x)(xy) \big) \|y\|_{\alpha}^{\alpha-Q} dy\\
=&   \sum_{\substack{k=1, \ldots, 2m, k  \ \mathrm{even}\\ i_1, \ldots,i_k \leq 2n+1\\ \mathrm{deg}(Z_{i_1})+ \ldots \mathrm{deg}(Z_{i_k})=2m}} \frac{2}{\alpha+2m}\frac{ Z_{i_1} \ldots Z_{i_k} \phi(x)}{k!}
 \int_{\H^n} y_{i_1} \ldots y_{i_k} h(1,y) \|y\|_c^{-\alpha-2m} d y.
\end{align*}
As a consequence, for every $-2m<\alpha<Q$ we have that
\begin{align}\label{comb3}
&\int_{\B_c(0,1)}\big(\mathcal{T}_{2m+1}(\phi,x)(xy)-\mathcal{T}_{2m-1}(\phi,x)(xy) \big) P_{\alpha}(y) dy  \\
=&   \sum_{\substack{k=1, \ldots, 2m, k  \ \mathrm{even}\\ i_1, \ldots,i_k \leq 2n+1 \\ \mathrm{deg}(Z_{i_1})+ \ldots \mathrm{deg}(Z_{i_k})=2m}} \frac{2}{\alpha+2m}\frac{1}{\Gamma(\frac{\alpha}{2})}\frac{ Z_{i_1} \ldots Z_{i_k} \phi(x)}{k!}
 \int_{\H^n} y_{i_1} \ldots y_{i_k} h(1,y) \|y\|_c^{-\alpha-2m} d y.  \notag 
\end{align}

\subsubsection{The behaviour as $\alpha \to -2\ell, \ \ell \leq m$}\label{sezione-limiti}
Let $m \in \N$. We wish to study the behaviour of the map 
\begin{align}\label{serve5}
\alpha \to \sum_{\substack{k=1, \ldots, 2m, \ k  \ \mathrm{even}\\ i_1, \ldots,i_k \leq 2n+1 \\ \mathrm{deg}(Z_{i_1})+ \ldots \mathrm{deg}(Z_{i_k})=2m}} \frac{2}{\alpha+2m}\frac{1}{\Gamma(\frac{\alpha}{2})}\frac{ Z_{i_1} \ldots Z_{i_k} \phi(x)}{k!}
 \int_{\H^n} y_{i_1} \ldots y_{i_k} h(1,y) \|y\|_c^{-\alpha-2m} d y,
\end{align}
when $\alpha$ is close to an even negative integer $-2k$, $k \geq m$.
First, we compute the limit of \eqref{serve5} as $\alpha \to -2m^+$. We use the properties of the Euler gamma function getting
\begin{align*}
&   \sum_{\substack{k=1, \ldots, 2m, \ k  \ \mathrm{even}\\ i_1, \ldots,i_k \leq 2n+1 \\ \mathrm{deg}(Z_{i_1})+ \ldots \mathrm{deg}(Z_{i_k})=2m}} \frac{2}{\alpha+2m}\frac{1}{\Gamma(\frac{\alpha}{2})}\frac{ Z_{i_1} \ldots Z_{i_k} \phi(x)}{k!}
 \int_{\H^n} y_{i_1} \ldots y_{i_k} h(1,y) \|y\|_c^{-\alpha-2m} d y \\
=&   \sum_{\substack{k=1, \ldots, 2m, \ k  \ \mathrm{even}\\ i_1, \ldots,i_k \leq 2n+1 \\ \mathrm{deg}(Z_{i_1})+ \ldots \mathrm{deg}(Z_{i_k})=2m}} \frac{2}{\alpha+2m}\frac{\left(\frac{\alpha}{2}\right)\left(\frac{\alpha}{2}+1\right)\ldots \left(\frac{\alpha}{2}+(m-1)\right)}{\Gamma(\frac{\alpha+2m}{2})} \frac{ Z_{i_1} \ldots Z_{i_k} \phi(x)}{k!}
 \int_{\H^n} y_{i_1} \ldots y_{i_k} h(1,y)  d y  \\
 \to & \sum_{\substack{k=1, \ldots, 2m, \  k  \ \mathrm{even}\\ i_1, \ldots,i_k \leq 2n+1 \\ \mathrm{deg}(Z_{i_1})+ \ldots \mathrm{deg}(Z_{i_k})=2m}} (-1)^m m! \frac{ Z_{i_1} \ldots Z_{i_k} \phi(x)}{k!}
 \int_{\H^n} y_{i_1} \ldots y_{i_k} h(1,y)  d y
\end{align*}
as $\alpha \to -2m^+$.
Let us now consider $\ell \in \N \cup \{0\}$, $\ell<m$, then we have
\begin{align*}
&  \lim_{\alpha \to -2\ell} \left( \sum_{\substack{k=1, \ldots, 2m, \ k  \ \mathrm{even}\\ i_1, \ldots,i_k \leq 2n+1 \\ \mathrm{deg}(Z_{i_1})+ \ldots \mathrm{deg}(Z_{i_k})=2m}} \frac{2}{\alpha+2m}\frac{1}{\Gamma(\frac{\alpha}{2})}\frac{ Z_{i_1} \ldots Z_{i_k} \phi(x)}{k!}
 \int_{\H^n} y_{i_1} \ldots y_{i_k} h(1,y) \|y\|_c^{-\alpha-2m} d y \right)\\
 =&  \lim_{\alpha \to -2\ell} \left( \sum_{\substack{k=1, \ldots, 2m, \ k  \ \mathrm{even}\\ i_1, \ldots,i_k \leq 2n+1 \\ \mathrm{deg}(Z_{i_1})+ \ldots \mathrm{deg}(Z_{i_k})=2m}} \frac{2}{2m-2\ell}\frac{1}{\Gamma(\frac{\alpha}{2})}\frac{ Z_{i_1} \ldots Z_{i_k} \phi(x)}{k!}
 \int_{\H^n} y_{i_1} \ldots y_{i_k} h(1,y) \|y\|_c^{2\ell-2m} d y \right)=0.
\end{align*}
\subsubsection{A second family of integrals}
We study the regularity of the map 
$\alpha \to \int_{\B_c(0,1)} \Big(\phi(xy)-\mathcal{T}_{2m+1}(\phi,x)(xy) \Big)P_{\alpha}(y) dy.$ The following result holds. 
\begin{lem} \label{analiticity}
Let $\phi \in \mathcal{S}(\H^n)$ and $x \in \G$. For every $m \in \N$, the map
\begin{equation}\label{forma}
\alpha \to \int_{\B_c(0,1)} \Big(\phi(xy)-\mathcal{T}_{2m+1}(\phi,x)(xy) \Big)P_{\alpha}(y) dy
\end{equation}
is analytic on the interval $(-2m-2,Q)$.
\end{lem}
Before proving Lemma \ref{analiticity}, we focus on the convergence of \eqref{forma}.
\begin{Remark}
We point out that
the integral 
$$\int_{\B_c(0,1)} \Big(\phi(xy)-\mathcal{T}_{2m+1}(\phi,x)(xy) \Big)P_{\alpha}(y) dy $$ converges when $-2m-2<\alpha<Q$. In fact, exploiting the coarea formula of Corollary \ref{cor-coarea} and the homogeneity properties of $P_{\alpha}$ and $\mathcal{S}^{Q-1}_{\infty}$ we get
\begin{align*}
&\int_{\B_c(0,1)}\Big(\phi(xy)-\mathcal{T}_{2m+1}(\phi,x)(xy) \Big)P_{\alpha}(y) dy \\
&= \int_0^1 \int_{\partial \B_c(0,r)} \Big(\phi(xy)-\mathcal{T}_{2m+1}(\phi,x)(xy)\Big)P_{\alpha}(y) d \mathcal{S}^{Q-1}_{\infty}(y) dr\\
&= \int_0^1 r^{\alpha-1} \int_{\partial \B_c(0,1)} \Big(\phi(x\delta_r(z))-\mathcal{T}_{2m+1}(\phi,x)(x\delta_r(z))\Big)P_{\alpha}(z) d \mathcal{S}^{Q-1}_{\infty}(z) dr.
\end{align*}
Now, recalling the regularity of $\phi$ and the estimate \eqref{Taylorn} of Theorem \ref{Taylor}, we know that for every $z \in \partial \B_c(0,1)$ $$|(\phi(x\delta_r(z))-\mathcal{T}_{2m+1}(\phi,x)(x\delta_r(z)))| =O(r^{2m+2}) \quad \mathrm{as} \ r \to 0$$holds. Therefore, for a suitable positive constant $c$ we have
\begin{align*}
\int_{\B_c(0,1)}|\phi(xy)-\mathcal{T}_{2m+1}(\phi,x)(xy)||P_{\alpha}(y)| dy & \leq c\int_0^1 r^{\alpha+2m+1} dr \int_{\partial \B_c(0,1)} |P_{\alpha}(z)| d \mathcal{S}^{Q-1}_{\infty}(z) \\
 & = c \ \frac{1}{| \Gamma(\frac{\alpha}{2})|} \frac{1}{\alpha+2m+2}\sigma(\alpha),
\end{align*}
which is finite. Here, we have exploited the fact that $\alpha >-2m-2$.
\end{Remark}

\begin{proof}[Proof of Lemma \ref{analiticity}]
The proof is analogous to the one of Lemma \ref{claim-partedentro}.
Let us consider the $k$-th derivative of the function in \eqref{forma}, for $k \in \N$. In particular, we will show that derivative and integral can be exchanged by estimating the term
\begin{equation}\label{modena}
\left| \int_{\B_c(0,1)} \Big(\phi(xy)-\mathcal{T}_{2m+1}(\phi,x)(xy) \Big)\int_0^{\infty} t^{\frac{\alpha}{2}-1} (\ln(t))^k \left( \frac{1}{2} \right)^k h(t,y)dt dy \right|.
\end{equation}
We split the integral of \eqref{modena} as
\begin{align}
&\left| \int_{\B_c(0,1)} (\phi(xy)-\mathcal{T}_{2m+1}(\phi,x)(xy))\int_0^{\infty} t^{\frac{\alpha}{2}-1} (\ln(t))^k \left( \frac{1}{2} \right)^k h(t,y)dt dy \right| \leq \notag\\
& \left| \int_{\B_c(0,1)} (\phi(xy)-\mathcal{T}_{2m+1}(\phi,x)(xy))\int_0^{1} t^{\frac{\alpha}{2}-1} (\ln(t))^k \left( \frac{1}{2} \right)^k h(t,y)dt dy \right|\\
&+\left| \int_{\B_c(0,1)} (\phi(xy)-\mathcal{T}_{2m+1}(\phi,x)(xy))\int_1^{\infty} t^{\frac{\alpha}{2}-1} (\ln(t))^k \left( \frac{1}{2} \right)^k h(t,y)dt dy \right|. \notag
\end{align}
The second term can be dealt with repeating the steps in \eqref{secondopezzo}
and we get that
\begin{equation}\label{zero1}
\begin{split}
&\left| \int_{\B_c(0,1)} (\phi(xy)-\mathcal{T}_{2m+1}(\phi,x)(xy)) \int_{1}^{\infty} t^{\frac{\alpha}{2}-1} \left( \frac{1}{2} \right)^k \ln^k(t) h(t,y)dt dy \right|\\
&\leq \sup_{\B_c(0,1)}| \phi(x \cdot)-\mathcal{T}_{2m+1}(\phi,x)(x \cdot) | \sup_{\B_c(0,1)} h(1, \cdot) \mu(\B_c(0,1)) \frac{2}{(Q-\alpha)^{k+1}}   k!.
\end{split}
\end{equation}
Then we need to deal with 
$$\left| \int_{\B_c(0,1)} (\phi(xy)-\mathcal{T}_{2m+1}(\phi,x)(xy))\int_0^{1} t^{\frac{\alpha}{2}-1} (\ln(t))^k \left( \frac{1}{2} \right)^k h(t,y)dt dy \right|.$$
In fact, we have
\begin{align}
\label{un2}
&\left| \int_{\B_c(0,1)} \Big(\phi(xy)-\mathcal{T}_{2m+1}(\phi,x)(xy) \Big)\int_0^{1} t^{\frac{\alpha}{2}-1} (\ln(t))^k \left( \frac{1}{2} \right)^k h(t,y)dt dy \right|  \notag \\
 \leq & \int_{\B_c(0,1)} |\phi(xy)-\mathcal{T}_{2m+1}(\phi,x)(xy)|\int_0^{1} t^{\frac{\alpha}{2}-1} |\ln(t)|^k \left( \frac{1}{2} \right)^k h(t,y)dt dy \\
 \leq & \int_0^{1} t^{\frac{\alpha}{2}-1} |\ln(t)|^k \left( \frac{1}{2} \right)^k  \int_{\B_c(0,1)} |\phi(xy)-\mathcal{T}_{2m+1}(\phi,x)(xy)| h(t,y)dy dt.  \notag 
\end{align}
Now, we observe that by the estimate \eqref{Taylorn} of Theorem \ref{Taylor}, we have $|\phi(x y)-\mathcal{T}_{2m+1}(\phi,x)(xy)| =O(\|y \|_c^{2m+2})$ as $\|y \|_c \to 0$, thus for some suitable positive constant $C$ we get
\begin{align*}
& \int_0^{1} t^{\frac{\alpha}{2}-1} |\ln(t)|^k \left( \frac{1}{2} \right)^k  \int_{\B_c(0,1)} |\phi(xy)-\mathcal{T}_{2m+1}(\phi,x)(xy)| h(t,y)dy dt \\
\leq & \ C \int_0^{1} t^{\frac{\alpha}{2}-1} |\ln(t)|^k \left( \frac{1}{2} \right)^k  \int_{\B_c(0,1)} \|y \|_c^{2m+2} h(t,y)dy dt \\
=& \ C \int_0^{1} t^{\frac{\alpha}{2}-1} |\ln(t)|^k \left( \frac{1}{2} \right)^k t^{-\frac{Q}{2}} \int_{\B_c(0,1)} \|y \|_c^{2m+2} h(1,\delta_{\frac{1}{\sqrt{t}}}(y))dy dt ,
\end{align*}
where we exploited the homogeneity property of $h$. Now, we change variables as $z= \delta_{\frac{1}{\sqrt{t}}}(y)$, so that by the homogeneity of the Lebesgue measure and the properties of homogeneous norms we con continue as follows
\begin{align}
\label{dos2}
& C\int_0^{1} t^{\frac{\alpha}{2}-1} |\ln(t)|^k \left( \frac{1}{2} \right)^k t^{-\frac{Q}{2}} \int_{\B_c(0,1)} \|y \|_c^{2m+2} h(1,\delta_{\frac{1}{\sqrt{t}}}(y))dy dt  \notag \\
=& \ C\int_0^{1} t^{\frac{\alpha}{2}-1} |\ln(t)|^k \left( \frac{1}{2} \right)^k  \int_{\B_c(0,\frac{1}{\sqrt{t}})} \|\delta_{\sqrt{t}}(z) \|_c^{2m+2} h(1,z)dz dt \notag  \\
=& \ C\int_0^{1} t^{\frac{\alpha}{2}-1} |\ln(t)|^k \left( \frac{1}{2} \right)^k  \int_{\B_c(0,\frac{1}{\sqrt{t}})} t^{m+1}\| z\|_c^{2m+2} h(1,z)dz dt \\
\leq &\  C\int_0^{1} t^{\frac{\alpha}{2}+m} |\ln(t)|^k \left( \frac{1}{2} \right)^k dt \int_{\H^n}  \|z\|_c^{2m+2} h(1,z)dz .  \notag 
\end{align}
Notice that $\int_{\H^n}  \|z\|_c^{2m+2} h(1,z)dz < \infty$ is a finite constant, since $h(1,z) \in \mathcal{S}(\H^n)$. Now, we observe that, according to \eqref{bk}, $$b_k=\int_0^{1} t^{\frac{\alpha}{2}+m} |\ln(t)|^k \left( \frac{1}{2} \right)^k   dt$$ when $N=m+1$, and $N=m+1 >- \frac{\alpha}{2}$ if and only if $\alpha>-(2m+2)$. Hence, we can exploit the estimate in \eqref{ric2} getting that for every fixed $\alpha >-2m-2$, for every $k \in \N$
\begin{equation}
\label{tres2}
\begin{split}
 \int_0^{1} t^{\frac{\alpha}{2}+m} |\ln(t)|^k \left( \frac{1}{2} \right)^k   dt
\leq C \frac{2}{(\alpha+2m+2)^{k+1}} k!.
\end{split}
\end{equation} 
By combining \eqref{zero1}, \eqref{un2}, \eqref{dos2} and \eqref{tres2} our claim is achieved.
\end{proof}
\begin{Remark}
The results of this subsection are still true if $\phi \in C^{\infty}(\H^n)$ and $\phi(z) \leq K \|z\|_c^{-L}$ as $\|z_c\|>S$ for some $L, S, K>0$.  It can be verified that in this case the map in \eqref{forma} is analytic on $(-2m-2,\min\{Q,L\})$.
\end{Remark}

\section{Analytic continuation in the Heisenberg group}\label{sect-indu}

Here, we prove the following result.
\begin{teo}\label{teo:emilia}
For every $m \in \N$, for every $\phi \in \mathcal{S}(\G)$ and for every $x \in \G$, the map $$\alpha \to \psi(x,\alpha)$$ can be analytically continued to the interval $(-2m-2, -2m]$.  Moreover, the representation 
\begin{align}\label{expression}
\psi(x,\alpha)=&\int_{\H^n \setminus \B_c(0,1)} \phi(xy)P_{\alpha}(y) dy + \int_{\B_c(0,1)} \Big(\phi(xy)-\mathcal{T}_{2m+1}(\phi,x)(xy)\Big)P_{\alpha}(y) dy  \notag  \\
& + \sum_{p=1}^m   \sum_{\substack{k=1, \ldots, 2p, k  \ \mathrm{even}\\ i_1, \ldots,i_k \leq 2n+1 \\ \mathrm{deg}(Z_{i_1})+ \ldots \mathrm{deg}(Z_{i_k})=2p}} \frac{1}{\alpha+2p}\frac{1}{\Gamma(\frac{\alpha}{2})}\frac{ Z_{i_1} \ldots Z_{i_k} \phi(x)}{k!} \int_{\partial \B_c(0,1)}  z_{i_1} \ldots  z_{i_k}  \| z\|_{\alpha}^{\alpha-Q} d \mathcal{S}^{Q-1}_{\infty}(z) \notag \\
&+ \phi(x) \frac{1}{\Gamma(\frac{\alpha}{2})} \frac{1}{\alpha} \sigma(\alpha)
\end{align}
holds for every $\alpha$ such that $-2m-2<\alpha \leq Q$.
\end{teo}
The proof of the previous theorem will be realized by induction in the following two subsections.
\subsection{Base case: continuation of $\alpha \to \psi(x,\alpha)$ on the strip $(-4,-2]$ }
We consider the map $ \alpha \to \psi(x,\alpha)$ which has been analytically continued to $\alpha \in (-2,Q)$ in Theorem \ref{teo:samoggia}. We obtained the expression of $\psi(x,\alpha)$ on $(-2,Q)$ 
\begin{align}\label{serve6}
\psi(x,\alpha)&= \int_{\H^n \setminus \B_c(0,1)} \phi(xy)P_{\alpha}(y) dy + \int_{\B_c(0,1)} \Big(\phi(xy)-\phi(x)\Big)P_{\alpha}(y) dy + \phi(x) \frac{1}{\Gamma(\frac{\alpha}{2})} \frac{1}{\alpha} \sigma(\alpha)\notag\\
&= \int_{\H^n \setminus \B_c(0,1)} \phi(xy)P_{\alpha}(y) dy + \int_{\B_c(0,1)} \Big(\phi(xy)-\mathcal{T}_1(\phi,x)(xy)\Big)P_{\alpha}(y) dy + \phi(x) \frac{1}{\Gamma(\frac{\alpha}{2})} \frac{1}{\alpha} \sigma(\alpha) 
\end{align}
on an arbitrary stratified group $\G$, and then, in particular, it holds on $\H^n$. In this subsection, we wish to continue analytically the map on the strip $-4 < \alpha \leq 2 $, so as to show the base step of the inductive procedure. In corresponds to the value $m=1$.
We consider $\alpha \in (-2,Q)$ and we consider \eqref{serve6} as follows
\begin{align*}
&\psi(x,\alpha)=\int_{\H^n \setminus \B_c(0,1)} \phi(xy)P_{\alpha}(y) dy + \int_{\B_c(0,1)} \Big(\phi(xy)-\mathcal{T}_{1}(\phi,x)(xy)\Big)P_{\alpha}(y) dy\\
&-\int_{\B_c(0,1)}(\mathcal{T}_3(\phi,x)(xy)-\mathcal{T}_1(\phi,x)(xy)) P_{\alpha}(y)dy + \int_{\B_c(0,1)} (\mathcal{T}_{3}(\phi,x)(xy)-\mathcal{T}_1(\phi,x)(xy) )P_{\alpha}(y) dy\\
&+ \phi(x) \frac{1}{\Gamma(\frac{\alpha}{2})} \frac{1}{\alpha} \sigma(\alpha)\\
=&\int_{\H^n \setminus \B_c(0,1)} \phi(xy)P_{\alpha}(y) dy + \int_{\B_c(0,1)} \Big(\phi(xy)-\mathcal{T}_{3}(\phi,x)(xy)\Big)P_{\alpha}(y) dy\\
&+ \int_{\B_c(0,1)} (\mathcal{T}_{3}(\phi,x)(xy)-\mathcal{T}_1(\phi,x)(xy) )P_{\alpha}(y) dy + \phi(x) \frac{1}{\Gamma(\frac{\alpha}{2})} \frac{1}{\alpha} \sigma(\alpha).
\end{align*}
We can exploit the expression obtained in \eqref{comb2} and for every $\alpha \in (-2, Q)$ we get 
\begin{align}\label{eq-basis}
\psi(x,\alpha)=&\int_{\H^n \setminus \B_c(0,1)} \phi(xy)P_{\alpha}(y) dy + \int_{\B_c(0,1)} \Big(\phi(xy)-\mathcal{T}_{3}(\phi,x)(xy)\Big)P_{\alpha}(y) dy \notag \\
&+ \sum_{\substack{i_1, i_2  \leq 2n+1 \\ \mathrm{deg}(Z_{i_1})+\mathrm{deg}(Z_{i_2})=2}} \frac{1}{\alpha+2}\frac{1}{\Gamma(\frac{\alpha}{2})}\frac{ Z_{i_1}Z_{i_2} \phi(x)}{k!}
 \int_{\partial \B_c(0,1)}  z_{i_1} z_{i_2}  \| z\|_{\alpha}^{\alpha-Q} d \mathcal{S}^{Q-1}_{\infty}(z)  \\
 & + \phi(x) \frac{1}{\Gamma(\frac{\alpha}{2})} \frac{1}{\alpha} \sigma(\alpha). \notag 
\end{align}
Let us focus on the right-hand side of \eqref{eq-basis}. We notice that, by Remark \ref{partefuori}, the first term of the sum
$$ \alpha \to \int_{\H^n \setminus \B_c(0,1)} \phi(xy)P_{\alpha}(y) dy$$ is analytic on $(-\infty, Q)$. The last term of the sum
$$ \alpha \to \phi(x) \frac{1}{\Gamma(\frac{\alpha}{2})} \frac{1}{\alpha} \sigma(\alpha)$$
is analytic on $(-\infty, Q)$ too, since in Item \ref{itemcalfa} of Section \ref{sectioncalfa} we proved that $\sigma(\alpha)$ is analytic on $(-\infty, \alpha)$. By the results in Section \ref{anal-integrals}, the third term of the sum 
$$\alpha \to \sum_{\substack{i_1, i_2  \leq 2n+1 \\ \mathrm{deg}(Z_{i_1})+\mathrm{deg}(Z_{i_2})=2}} \frac{1}{\alpha+2}\frac{1}{\Gamma(\frac{\alpha}{2})}\frac{ Z_{i_1}Z_{i_2} \phi(x)}{k!}
 \int_{\partial \B_c(0,1)}  z_{i_1} z_{i_2}  \| z\|_{\alpha}^{\alpha-Q} d \mathcal{S}^{Q-1}_{\infty}(z)$$
is analytic on $(-\infty,Q)$. Finally, by Lemma \ref{analiticity} applied with $m=1$, the second term of the sum 
$$\alpha \to \int_{\B_c(0,1)} \Big(\phi(xy)-\mathcal{T}_{3}(\phi,x)(xy)\Big)P_{\alpha}(y) dy$$
is analytic on $(-4,Q)$. In conclusion, the map $\alpha \to \psi(x,\alpha)$ is analytic on $(-4,Q)$. Then, we have obtained an analytic continuation of $\psi(x,\alpha)$ on $-4<\alpha \leq -2$. On the other hand, we have proved the first step of the inductive procedure, corresponding to $m=1$, aimed in the current section.

\subsection{Induction step: continuation of $\psi(x,\alpha)$ on the strip $(-2m-2, -2m]$} 
Now, we fix $m \in \N$ and we assume that the map $\alpha \to \psi(x,\alpha)$ has been analytically continued on $(-2m,Q)$ and that the representation
\begin{align}\label{map2}
\psi(x,\alpha)&=\int_{\H^n \setminus \B_c(0,1)} \phi(xy)P_{\alpha}(y) dy + \int_{\B_c(0,1)} \Big(\phi(xy)-\mathcal{T}_{2(m-1)+1}(\phi,x)(xy)\Big)P_{\alpha}(y) dy  \notag  \\
& + \sum_{p=1}^{m-1}  \sum_{\substack{k=1, \ldots, 2p, k  \ \mathrm{even}\\ i_1, \ldots,i_k \leq 2n+1 \\ \mathrm{deg}(Z_{i_1})+ \ldots \mathrm{deg}(Z_{i_k})=2p}} \frac{1}{\alpha+2p}\frac{1}{\Gamma(\frac{\alpha}{2})}\frac{ Z_{i_1} \ldots Z_{i_k} \phi(x)}{k!} \int_{\partial \B_c(0,1)}  z_{i_1} \ldots  z_{i_k}  \| z\|_{\alpha}^{\alpha-Q} d \mathcal{S}^{Q-1}_{\infty}(z)  \\
&+ \phi(x) \frac{1}{\Gamma(\frac{\alpha}{2})} \frac{1}{\alpha} \sigma(\alpha)  \notag
\end{align}
holds.
We wish to construct an analytic continuation of the map $\psi(x,\alpha)$ in \eqref{map2} on the strip $(-2m-2,-2m]$. Let us consider
\begin{align}\label{map21}
\psi(x,\alpha)&=\int_{\H^n \setminus \B_c(0,1)} \phi(xy)P_{\alpha}(y) dy + \int_{\B_c(0,1)} \Big(\phi(xy)-\mathcal{T}_{2m+1}(\phi,x)(xy)\Big)P_{\alpha}(y) dy  \notag \\
+& \int_{\B_c(0,1)} \Big( \mathcal{T}_{2m+1}(\phi,x)(xy)-\mathcal{T}_{2(m-1)+1}(\phi,x)(xy) \Big) P_{\alpha}(y) dy\\
+& \sum_{p=1}^{m-1}  \sum_{\substack{k=1, \ldots, 2p, k  \ \mathrm{even}\\ i_1, \ldots,i_k \leq 2n+1 \\ \mathrm{deg}(Z_{i_1})+ \ldots \mathrm{deg}(Z_{i_k})=2p}} \frac{1}{\alpha+2p}\frac{1}{\Gamma(\frac{\alpha}{2})}\frac{ Z_{i_1} \ldots Z_{i_k} \phi(x)}{k!} \int_{\partial \B_c(0,1)}  z_{i_1} \ldots  z_{i_k}  \| z\|_{\alpha}^{\alpha-Q} d \mathcal{S}^{Q-1}_{\infty}(z)  \notag \\
+& \phi(x) \frac{1}{\Gamma(\frac{\alpha}{2})} \frac{1}{\alpha} \sigma(\alpha). \notag 
\end{align}
Now, we apply \eqref{comb2} and we get that for every $\alpha \in (-2m,Q)$, the right-hand side of \eqref{map21} is equal to
\begin{align}\label{eq-indu}
&\int_{\H^n \setminus \B_c(0,1)} \phi(xy)P_{\alpha}(y) dy + \int_{\B_c(0,1)} \Big(\phi(xy)-\mathcal{T}_{2m+1}(\phi,x)(xy)\Big)P_{\alpha}(y) dy  \notag \\
& + \sum_{p=1}^{m}  \sum_{\substack{k=1, \ldots, 2p, k  \ \mathrm{even}\\ i_1, \ldots,i_k \leq 2n+1 \\ \mathrm{deg}(Z_{i_1})+ \ldots \mathrm{deg}(Z_{i_k})=2p}} \frac{1}{\alpha+2p}\frac{1}{\Gamma(\frac{\alpha}{2})}\frac{ Z_{i_1} \ldots Z_{i_k} \phi(x)}{k!} \int_{\partial \B_c(0,1)}  z_{i_1} \ldots  z_{i_k}  \| z\|_{\alpha}^{\alpha-Q} d \mathcal{S}^{Q-1}_{\infty}(z)\\
&+ \phi(x) \frac{1}{\Gamma(\frac{\alpha}{2})} \frac{1}{\alpha} \sigma(\alpha). \notag 
\end{align}
Let us focus on \eqref{eq-indu}. We observe that the first term of the sum is analytic on $(-\infty, Q)$ by Remark \ref{partefuori}; the last term of the sum is analytic on $(-\infty, Q)$ since in Item \ref{itemcalfa} of Section \ref{sectioncalfa} we proved that $\sigma(\alpha)$ is analytic on $(-\infty, \alpha)$; the third term of the sum is analytic on $(-\infty,Q)$ by the results in Section \ref{anal-integrals}; the second term of the sum is analytic on $(-2m-2,Q)$ by Lemma \ref{analiticity}. Then, the map $\alpha \to \psi(x,\alpha)$ in \eqref{eq-indu} is analytic on $(-2m-2,Q)$. Hence, we provided an analytic continuation of $\psi(x,\alpha)$ on $-2m-2<\alpha \leq -2m$ and the representation in \eqref{eq-indu} holds for $-2m-2<\alpha<Q$.
Hence the induction is concluded and Theorem \ref{teo:emilia} is proved.

\begin{Remark}
Notice that, for every $m\in \N$, when $\alpha<-2m$, by explicit calculations based on the coarea formula, we obtain
\begin{equation}\label{int-fuori}
\begin{split}
&\int_{\H^n \setminus \B_c(0,1)}  \Big(\mathcal{T}_{2m+1}(\phi,x)(xy)-\mathcal{T}_{2m-1}(\phi,x)(xy)\Big)P_{\alpha}(y) dy\\
&= -\int_{\B_c(0,1)}  \Big(\mathcal{T}_{2m+1}(\phi,x)(xy)-\mathcal{T}_{2m-1}(\phi,x)(xy)\Big)P_{\alpha}(y) dy,
\end{split}
\end{equation} so that,  when $-2m-2<\alpha <-2m$ we have the following representation of the analytic continuation
\begin{equation}\label{formaglobale}
\begin{split}
\psi(x,\alpha)=  \int_{\H^n} \Big(\phi(xy)-\mathcal{T}_{2m+1}(\phi,x)(xy)\Big)P_{\alpha}(y) dy.
\end{split}
\end{equation}
\end{Remark}
\subsubsection{Analysis of $\psi(x,-2m)$, $m \in \N$}
Let us now compute the value $\psi(x,-2m)$, $m \in \N$. We exploit the representation of $\psi(x,\alpha)$ on $(-2m-2, Q)$ and the results in Section \ref{sezione-limiti} getting:
\begin{align*}
&\psi(x,\alpha)=\int_{\H^n \setminus \B_c(0,1)} \phi(xy)P_{\alpha}(y) dy + \int_{\B_c(0,1)} \Big(\phi(xy)-\mathcal{T}_{2m+1}(\phi,x)(xy)\Big)P_{\alpha}(y) dy\\
& + \sum_{p=1}^m   \sum_{\substack{k=1, \ldots, 2p, k  \ \mathrm{even}\\ i_1, \ldots,i_k \leq 2n+1 \\ \mathrm{deg}(Z_{i_1})+ \ldots \mathrm{deg}(Z_{i_k})=2p}} \frac{1}{\alpha+2p}\frac{1}{\Gamma(\frac{\alpha}{2})}\frac{ Z_{i_1} \ldots Z_{i_k} \phi(x)}{k!} \int_{\partial \B_c(0,1)}  z_{i_1} \ldots  z_{i_k}  \| z\|_{\alpha}^{\alpha-Q} d \mathcal{S}^{Q-1}_{\infty}(z)\\
&+ \phi(x) \frac{1}{\Gamma(\frac{\alpha}{2})} \frac{1}{\alpha} \sigma(\alpha) \\
=& \ \frac{1}{\Gamma(\frac{\alpha}{2})}\int_{\H^n \setminus \B_c(0,1)} \phi(xy)\|y\|_{\alpha}^{\alpha-Q}dy + \frac{1}{\Gamma(\frac{\alpha}{2})} \int_{\B_c(0,1)} \Big(\phi(xy)-\mathcal{T}_{2m+1}(\phi,x)(xy)\Big)\|y\|_{\alpha}^{\alpha-Q} dy\\
& +   \sum_{p=1}^m \sum_{\substack{k=1, \ldots, 2p, k  \ \mathrm{even}\\ i_1, \ldots,i_k \leq 2n+1 \\ \mathrm{deg}(Z_{i_1})+ \ldots \mathrm{deg}(Z_{i_k})=2p}}   \frac{1}{\alpha+2p}\frac{1}{\Gamma(\frac{\alpha}{2})}\frac{ Z_{i_1} \ldots Z_{i_k} \phi(x)}{k!} \int_{\partial \B_c(0,1)}  z_{i_1} \ldots  z_{i_k}  \| z\|_{\alpha}^{\alpha-Q} d \mathcal{S}^{Q-1}_{\infty}(z) \\
&+  \phi(x) \frac{1}{\Gamma(\frac{\alpha}{2})} \frac{1}{\alpha} \sigma(\alpha)\\
\to  & \lim_{\alpha \to -2m}  \sum_{p=1}^m   \sum_{\substack{k=1, \ldots, 2p, k  \ \mathrm{even}\\ i_1, \ldots,i_k \leq 2n+1 \\ \mathrm{deg}(Z_{i_1})+ \ldots \mathrm{deg}(Z_{i_k})=2p}} \frac{1}{\alpha+2p}\frac{1}{\Gamma(\frac{\alpha}{2})}\frac{ Z_{i_1} \ldots Z_{i_k} \phi(x)}{k!} \int_{\partial \B_c(0,1)}  z_{i_1} \ldots  z_{i_k}  \| z\|_{\alpha}^{\alpha-Q} d \mathcal{S}^{Q-1}_{\infty}(z)\\
&= \sum_{\substack{k=1, \ldots, 2m, k  \ \mathrm{even}\\ i_1, \ldots,i_k \leq 2n+1 \\ \mathrm{deg}(Z_{i_1})+ \ldots \mathrm{deg}(Z_{i_k})=2m}} (-1)^m m! \frac{ Z_{i_1} \ldots Z_{i_k} \phi(x)}{k!}
 \int_{\H^n} y_{i_1} \ldots y_{i_k} h(1,y)  d y
\end{align*}
as $\alpha$ goes to $-2m$. We used in the last two steps the results of Section \ref{sezione-limiti}. Hence, we showed that for every $m \in \N$
\begin{align*}
\psi(x, -2m)= \sum_{\substack{k=1, \ldots, 2m, k  \ \mathrm{even}\\ i_1, \ldots,i_k \leq 2n+1 \\ \mathrm{deg}(Z_{i_1})+ \ldots \mathrm{deg}(Z_{i_k})=2m}} (-1)^m m! \frac{ Z_{i_1} \ldots Z_{i_k} \phi(x)}{k!}
 \int_{\H^n} y_{i_1} \ldots y_{i_k} h(1,y)  d y. 
\end{align*}

\subsection{An estimate of $x \to \psi(x,\alpha)$}
We wish to use the analytic continuation of $\psi(x,\alpha)$ to extend the notion of fractional sub-Laplacian $\mathcal{L}^s$ in the spirit of Sections \ref{sec-introddistr} and \ref{section-laplesteso}. In order to do this, we need to generalize in $\H^n$ the estimate on $x \to \psi(x,\alpha)$ established in Section \ref{estimates-psi}. This is the goal of the following proposition.
\begin{prop}
For every $m \in \N$, for every $\alpha \in (-2m-2,2m)$, for every $\phi \in \mathcal{S}(\H^n)$ and for every $x \in \H^n$ the following estimate holds
$$\psi(x, \alpha)=O(\|x\|_c^{\alpha-Q})$$
as $\|x\|_c \to \infty$.
\end{prop}
\begin{proof}
The proof is analogous to the one of Proposition \ref{est-2}. We exploit the expression of $\psi(x,\alpha)$
\begin{align*}
\psi(x,\alpha)= & \int_{\H^n} \Big(\phi(xy)-\mathcal{T}_{2m+1}(\phi,x)(xy) \Big) P_{\alpha}(y) dy
\end{align*}
obtained in \eqref{expression} for $-2m-2<\alpha<-2m$. Thus, the proof of Proposition \ref{est-2} can be essentially repeated verbatim, up to substituting $b^2$ with $b^{2m+2}$ and $\mathfrak{d}_h^2(\xi)$ with 
\begin{equation}\label{mathfrakd}
\mathfrak{d}^{2m+2}_h\phi(\xi)=\max_{p \in \B_c(x,b^{m+2}\|x^{-1}z\|_c)} \mathfrak{d}^{2m+2}_h\phi(p),
\end{equation} where $\mathfrak{d}^{2m+2}_h\phi$ is $\mathfrak{d}^{2m+2}_h\phi(y)=\max_{|\alpha|=2m+2} |(Z_1 \cdots Z_{2n})^{\alpha}\phi(y)|$, for every $y \in \G$. \\
More precisely, in this case the integral variable $z$ belongs to the closed ball $\B_c(x,\frac{\|x\|_c}{2b^{2m+2}})$ and for each of such a $z$, the maximum of $\mathfrak{d}_h^{2m+2}$ in \eqref{mathfrakd} is realized at $\xi=\xi(z) \in \B_c(x,b^{2m+2}\|x^{-1}z\| ) \subset \B_c(x, \frac{\|x\|_c}{2})$.
 Moreover, it is necessary to exploit the fact that every derivative along vector fields of $\phi \in \mathcal{S}(\H^n)$, independently of the number of involved vector fields and on their degree, belongs to $S(\H^n$), hence $\mathfrak{d}_h^{2m+2} \phi \in \mathcal{S}(\H^n)$.
\end{proof}

\subsection{The powers of the fractional sub-Laplacian $\mathcal{L}^s$, $-\frac{Q}{2}<s<m+1$, $m \in \N$}
Coherently with the construction of the analytic continuation of the map $\alpha \to \psi_{\phi}(x,\alpha)$ on the interval $(-2m-2,Q)$, for an arbitrary $m \in \N$, in this section we extend the definition of real powers of the sub-Laplacian $\mathcal{L}$ to $s \in (-\frac{Q}{2},m+1)$. In fact, for every $x \in \H^n$ and $\phi \in \mathcal{S}(\H^n)$ the analytic continuation of $\alpha \to \psi_{\phi}(x,\alpha)$ on the interval $\alpha \in (-2m-2,Q)$ allows to extend the definition of the family of tempered distributions $\{ \widetilde{P}^x_{\alpha} \}_{\alpha \in (-2,Q)}$ that we introduced in Section \ref{sec-introddistr}, setting for every $\alpha \in (-2m-2,Q)$ and $x \in \H^n$ the following
$$\widetilde{P}^{x}_{\alpha}( \phi) :=\psi_{\phi}(x, \alpha),$$
for every $\phi \in \mathcal{S}(\H^n)$.\\
Therefore, we can define for every $s \in (-\frac{Q}{2},m+1)$ the fractional sub-Laplacian
\begin{equation}\label{realpow}
\mathcal{L}^s\phi(x):= \psi_{\phi}(x,-2s)= \widetilde{P}^x_{-2s}(\phi),
\end{equation}
for every $\phi \in \mathcal{S}(\H^n)$ and $x \in \H^n$. \\
In a similar way as in Section \ref{section-laplesteso}, the uniqueness of analytic continuation yields that for every $x \in \H^n$ and $\phi \in \mathcal{S}(\H^n)$ the equality
\begin{equation}\label{compu}
\psi_{\phi}(x,\alpha+\beta)=\widetilde{P}_{\alpha+\beta}^x(\phi)=\psi_{\psi_{\phi}(\cdot,\beta)}(x,\alpha)
\end{equation}
holds for every $\alpha \in (-2m-2,Q), \beta \in (-2m-2,Q)$ such that $-2m-2<\alpha+\beta <Q$. 
Therefore, one more time through analogous computations as in Section \ref{section-laplesteso},  \eqref{compu} allows to prove that if we choose $\alpha, \beta \in (-2m-2,Q)$ and we set $\alpha=-2s$ and $\beta=-2p$ for suitable $s,p \in (-\frac{Q}{2},m+1)$, then the equality
\begin{equation}\label{sg1}
\mathcal{L}^s \circ \mathcal{L}^p=\mathcal{L}^{s+p}
\end{equation}
holds, when $-\frac{Q}{2}<s+p <m+1$.\\
Since for every $x \in \H^n$ and $\phi \in \mathcal{S}(\H^n)$, the analytic continuation of $\alpha \to \psi_{\phi}(x,\alpha)$ has been realized on the interval $(-2m-2,Q)$ for every $m \in \N$, then \eqref{sg1} holds for every $m \in \N$. Therefore, we can deduce that the operator $\mathcal{L}^s$ is defined as in \eqref{realpow} for every $s >-\frac{Q}{2}$ and it satisfies the semigroup property \eqref{semigroup} for every $s,p >-\frac{Q}{2}$ such that $s+q>-\frac{Q}{2}$. We resume this result in the following proposition.
\begin{prop} \label{prop:semigroup}
For every $s,p \in (-\frac{Q}{2}, \infty)$ such that $-\frac{Q}{2}<p+s $ the equality
\begin{equation}\label{semigroup}
\mathcal{L}^s \circ \mathcal{L}^p=\mathcal{L}^{s+p}
\end{equation}
holds.
\end{prop}
Actually, recalling the homogeneous dimension of $\H^n$ is $Q=2n+2$, we have introduced a notion of $\mathcal{L}^s$ on $\mathcal{S}(\H^n)$, for $s >-n-1$, which satisfies the semigroup property \eqref{semigroup} for every $s,p >-n-1$ such that $s+p>-n-1$.

\section{A representation of $\psi(x,\alpha)$ on $(-4,Q)$} 
\label{striscia2}
In this section we obtain a special simplified representation of $\psi(x, \alpha)$ on the interval $(-4,Q)$. Then, we use this representation to deduce the value of the integral $\int_{\H^n}|\pi_{H_1}(x)|^2 h(1,x) dx$.\\
Let us observe that, according to the arguments in Section \ref{sect:441} leading to \eqref{coi-multi} and then to \eqref{dariportare}, for every $\alpha<Q$ the equality
\begin{align}\label{continua}
&\sum_{\substack{i_1, i_2  \leq 2n+1 \\ \mathrm{deg}(Z_{i_1})+\mathrm{deg}(Z_{i_2})=2}} \frac{1}{\alpha+2}\frac{1}{\Gamma(\frac{\alpha}{2})}\frac{ Z_{i_1}Z_{i_2} \phi(x)}{2!}
 \int_{\partial \B_c(0,1)}  z_{i_1} z_{i_2}  \| z\|_{\alpha}^{\alpha-Q} d \mathcal{S}^{Q-1}_{\infty}(z) \notag \\
=&\frac{1}{\alpha+2} \frac{1}{\Gamma(\frac{\alpha}{2})}  \sum_{i=1}^{2n} Z^2_i\phi(x) \int_{\partial \B_c(0,1)}  y_i^2\| y\|_{\alpha}^{\alpha-Q} d \mathcal{S}^{Q-1}_{\infty}(y)
\end{align}
holds.
Now, since the ball $\B_c(0,1)$ and $\| \cdot \|_{\alpha}^{\alpha-Q}$ are invariant under horizontal rotations (see Sections \ref{parita-ball} and \ref{parita-h}), then the relation
\begin{equation}\label{dalfa}
d(\alpha):= \int_{\partial \B_c(0,1)} y_i^2 \|y\|_{\alpha}^{\alpha-Q} d \mathcal{S}^{Q-1}_{\infty}(y)= \frac{1}{2n} \int_{\partial \B_c(0,1)} |\pi_{H_1}(y)|^2 \|y\|_{\alpha}^{\alpha-Q} d \mathcal{S}^{Q-1}_{\infty}(y)
\end{equation}
holds for every $i=1, \dots, 2n$. In fact, by exploiting horizontal rotation invariance for every $i=1, \ldots, 2n$, we get that 
\begin{equation}
 \int_{\partial \B_c(0,1)} y_i^2 \|y\|_{\alpha}^{\alpha-Q} d \mathcal{S}^{Q-1}_{\infty}(y)= \int_{\partial \B_c(0,1)} y_j^2 \|y\|_{\alpha}^{\alpha-Q} d \mathcal{S}^{Q-1}_{\infty}(y),
\end{equation}
for every $i,j \in \{1, \ldots, 2n \}$. Then, we have that for every $i \in \{1, \ldots, 2n \}$
\begin{align*}
2n \int_{\partial \B_c(0,1)} y_i^2 \|y\|_{\alpha}^{\alpha-Q} d \mathcal{S}^{Q-1}_{\infty}(y)
=&\sum_{i=1}^{2n}\int_{\partial \B_c(0,1)} y_i^2 \|y\|_{\alpha}^{\alpha-Q} d \mathcal{S}^{Q-1}_{\infty}(y)\\
=& \int_{\partial \B_c(0,1)} \sum_{i=1}^{2n}y_i^2 \|y\|_{\alpha}^{\alpha-Q} d \mathcal{S}^{Q-1}_{\infty}(y)\\
=& \int_{\partial \B_c(0,1)} |(y_1, \ldots, y_{2n})|^2 \|y\|_{\alpha}^{\alpha-Q} d \mathcal{S}^{Q-1}_{\infty}(y)\\
=& \int_{\partial \B_c(0,1)} |\pi_{H_1}(y)|^2 \|y\|_{\alpha}^{\alpha-Q} d \mathcal{S}^{Q-1}_{\infty}(y).
\end{align*}
As a consequence, the value of $d(\alpha)$ does not depend on the choice of $i$. 
We can continue from \eqref{continua} getting
\begin{align*}
&  \sum_{\substack{i_1, i_2  \leq 2n+1 \\ \mathrm{deg}(Z_{i_1})+\mathrm{deg}(Z_{i_2})=2}} \frac{1}{\alpha+2}\frac{1}{\Gamma(\frac{\alpha}{2})}\frac{ Z_{i_1}Z_{i_2} \phi(x)}{2}
 \int_{\partial \B_c(0,1)}  z_{i_1} z_{i_2}  \| z\|_{\alpha}^{\alpha-Q} d \mathcal{S}^{Q-1}_{\infty}(z)\\
=&\frac{1}{\alpha+2} \frac{1}{\Gamma(\frac{\alpha}{2})}  \sum_{i=1}^{2n} Z^2_i\phi(x) d(\alpha)=-\frac{1}{\alpha+2} \frac{1}{\Gamma(\frac{\alpha}{2})}\frac{1}{2} \mathcal{L}\phi(x) d(\alpha).
\end{align*}
Thus, the analytically continued map $(-4,Q) \ni \alpha \to \psi(x,\alpha)$ can be represented as follows
\begin{equation}
\label{eq-secondastriscia}
\begin{split}
\psi(x,\alpha)= & \int_{\H^n \setminus \B_c(0,1)} \phi(xy)P_{\alpha}(y) dy + \int_{\B_c(0,1)} \Big(\phi(xy)-\phi(x)-\frac{1}{2}\sum_{i=1}^{2n} Z^2_i\phi(x)y^2_i \Big)P_{\alpha}(y) dy \\
&+ \phi(x) \frac{1}{\Gamma(\frac{\alpha}{2})} \frac{1}{\alpha} \sigma(\alpha)- \frac{1}{2}\frac{1}{\alpha+2}\frac{1}{\Gamma(\frac{\alpha}{2})} \mathcal{L}\phi(x)d(\alpha).
\end{split}
\end{equation}

The value of $\psi(x,-2)$ can be deduced by combining Proposition \ref{prop-limauno} and the equality in \eqref{fiuguale} getting that \begin{equation}\label{link-geomcons}
\lim_{\alpha \to -2} \psi(x,\alpha)=\lim_{\alpha \to -2^+} \psi(x,\alpha)=\lim_{\alpha \to -2^+}\mathcal{L}^{-\frac{\alpha}{2}}\phi(x)=\mathcal{L}\phi(x).
\end{equation}

\subsection{A geometric application, I}
We use the representation established in \eqref{eq-secondastriscia} to prove the following proposition.
\begin{prop}\label{Rem-consequence}
Let $n \in \N$ and let $h$ be the heat kernel associated with $\mathcal{L}$ on $\H^n$, then the following equality holds
\begin{equation}
\int_{\H^n} | \pi_{H_1}(x)|^2 h(1, x) dx=4n.
\end{equation}
\end{prop}
\begin{proof}

Let us consider the representation \eqref{eq-secondastriscia} of the map $\psi(x,\alpha)$ for $-4<\alpha<Q$
\begin{align*}
\psi(x,\alpha)= & \int_{\H^n \setminus \B_c(0,1)} \phi(xy)P_{\alpha}(y) dy + \int_{\B_c(0,1)} \Big(\phi(xy)-\phi(x)-\frac{1}{2}\sum_{i=1}^{2n} Z^2_i\phi(x)y^2_i \Big)P_{\alpha}(y) dy \\
&+ \phi(x) \frac{1}{\Gamma(\frac{\alpha}{2})} \frac{1}{\alpha} \sigma(\alpha)- \frac{1}{2}\frac{1}{\alpha+2}\frac{1}{\Gamma(\frac{\alpha}{2})} \mathcal{L}\phi(x)d(\alpha).
\end{align*}
By the arguments carried out in Section \ref{striscia2} we know that
\begin{align*}
\psi(x,-2)&=\lim_{\alpha \to -2} \psi(x,\alpha)= \lim_{\alpha \to -2}\Big( - \frac{1}{2}\frac{1}{\alpha+2}\frac{1}{\Gamma(\frac{\alpha+2}{2})}\frac{\alpha}{2} d(\alpha)\mathcal{L}\phi(x) \Big)\\
&= \lim_{\alpha \to -2} \frac{1}{4} d(\alpha)\mathcal{L}\phi(x)=  \frac{1}{4} \mathcal{L}\phi(x) \lim_{\alpha \to -2}d(\alpha),
\end{align*}
and we know as well that the limit exists since $d(\alpha)$ is continuous with respect to $\alpha$ as $\alpha<Q$.
Now, as we already observed, by Proposition \ref{prop-limauno} and \eqref{fiuguale}, we get\begin{equation}
\psi(x,-2)=\lim_{\alpha \to -2^+} \psi(x,\alpha)=\lim_{\alpha \to -2^+}\mathcal{L}^{-\frac{\alpha}{2}}\phi(x)=\mathcal{L}\phi(x),
\end{equation}
so that we can deduce that 
\begin{equation}\label{limdalpha}
\lim_{\alpha \to -2} d(\alpha)=d(-2)=4.
\end{equation}
Let us now consider for $\alpha<Q$, the map $d(\alpha)$ defined in \eqref{dalfa}. We have
\begin{align*}
d(\alpha)&=\frac{1}{2n} \int_{\partial \B_c(0,1)} |\pi_{H_1}(x)|^2 \|x\|_{\alpha}^{\alpha-Q} d \mathcal{S}^{Q-1}_{\infty}(x)\\
&=\frac{1}{2n}\int_{\partial \B_c(0,1)} |\pi_{H_1}(x)|^2 \int_0^{\infty} t^{\frac{\alpha}{2}-1} h(t,x) d t \  d \mathcal{S}^{Q-1}_{\infty}(x)\\
&=\frac{1}{2n} \int_0^{\infty} t^{\frac{\alpha}{2}-1} \int_{\partial \B_c(0,1)} |\pi_{H_1}(x)|^2h(t,x) d \mathcal{S}^{Q-1}_{\infty}(x)d t  \\
&=\frac{1}{2n} \int_0^{\infty} t^{\frac{\alpha}{2}-1-\frac{Q}{2}} \int_{\partial \B_c(0,1)} |\pi_{H_1}(x)|^2h(1,\delta_{\frac{1}{\sqrt{t}}}(x)) d \mathcal{S}^{Q-1}_{\infty}(x)d t. 
\end{align*}
Performing the change of variable $r=\frac{1}{\sqrt{t}}$, we get
\begin{align*}
&\frac{1}{2n} \int_0^{\infty} t^{\frac{\alpha}{2}-1-\frac{Q}{2}} \int_{\partial \B_c(0,1)} |\pi_{H_1}(x)|^2h(1,\delta_{\frac{1}{\sqrt{t}}}(x)) d \mathcal{S}^{Q-1}_{\infty}(x)d t\\
&=-\frac{1}{2n}2 \int_{\infty}^0 (r^{-2})^{\frac{\alpha}{2}-1-\frac{Q}{2}} r^{-3} \int_{\partial \B_c(0,1)} |\pi_{H_1}(x)|^2h(1,\delta_r(x)) d \mathcal{S}^{Q-1}_{\infty}(x)d r\\
&=\frac{1}{n} \int_0^{\infty} r^{-\alpha+Q-1} \int_{\partial \B_c(0,1)} |\pi_{H_1}(x)|^2h(1,\delta_r(x)) d \mathcal{S}^{Q-1}_{\infty}(x)d r\\
&=\frac{1}{n} \int_0^{\infty} r^{-\alpha+Q-3} \int_{\partial \B_c(0,1)} |r\pi_{H_1}(x)|^2h(1,\delta_r(x)) d \mathcal{S}^{Q-1}_{\infty}(x)d r.
\end{align*}
Let us remark that for every $r>0$ and $x \in \H^n$ we have $|r\pi_{H_1}(x)|^2=|\pi_{H_1}(\delta_r(x))|^2$, thus
\begin{align*}
&\frac{1}{n} \int_0^{\infty} r^{-\alpha+Q-3} \int_{\partial \B_c(0,1)} |r \pi_{H_1}(x)|^2h(1,\delta_r(x)) d \mathcal{S}^{Q-1}_{\infty}(x)d r\\
&=\frac{1}{n} \int_0^{\infty} r^{-\alpha+Q-3} \int_{\partial \B_c(0,1)} |\pi_{H_1}(\delta_r(x))|^2 h(1,\delta_r(x)) d \mathcal{S}^{Q-1}_{\infty}(x)d r.
\end{align*}
By changing one more time the variables $x'=\delta_r(x)$, the homogeneity of the metric $d_c$ and of the measure $\mathcal{S}^{Q-1}_{\infty}$ give
\begin{align*}
&\frac{1}{n} \int_0^{\infty} r^{-\alpha+Q-3} \int_{\partial \B_c(0,1)} |\pi_{H_1}(\delta_r(x))|^2 h(1,\delta_r(x)) d \mathcal{S}^{Q-1}_{\infty}(x)d r\\
&=\frac{1}{n} \int_0^{\infty} r^{-\alpha-2} \int_{\partial \B_c(0,r)} |\pi_{H_1}(x')|^2 h(1,x') d \mathcal{S}^{Q-1}_{\infty}(x')d r.
\end{align*}
Our thesis follows as $\alpha \to -2$. In fact, by \eqref{limdalpha} we know that 
\begin{equation}\label{combining1}
\lim_{\alpha \to -2}d(\alpha)=d(-2)=4
\end{equation} but, at the same time, we know as well that
\begin{align}\label{combining2}
\lim_{\alpha\to -2} d(\alpha)&=\lim_{\alpha \to -2}\frac{1}{n} \int_0^{\infty} r^{-\alpha-2} \int_{\partial \B_c(0,r)} |\pi_{H_1}(x')|^2 h(1,x') d \mathcal{S}^{Q-1}_{\infty}(x')d r \notag \\
&=\frac{1}{n} \int_0^{\infty} \int_{\partial \B_c(0,r)} |\pi_{H_1}(x')|^2 h(1,x') d \mathcal{S}^{Q-1}_{\infty}(x')d r \notag \\
&=\frac{1}{n} \int_{\H^n} |\pi_{H_1}(x')|^2 h(1,x') d (x')\\
&=\frac{1}{n} \int_{\H^n} |\pi_{H_1}(x)|^2 h(1,x) dx,  \notag 
\end{align}
where we used the coarea formula of Corollary \ref{cor-coarea}.
By combining \eqref{combining1} and \eqref{combining2} we conclude the proof.
\end{proof}

\section{A representation of $\psi(x,\alpha)$ on $(-6,Q)$}
\label{striscia3}
Here we discuss a representation of $\psi(x, \alpha)$ for $-6<\alpha \leq -4$ involving the most natural definition of the natural power $\mathcal{L}^2$ of the sub-Laplacian $\mathcal{L}$. 
Given $m \in \N$, we define
\begin{align}
\mathcal{L}^m \phi:&=\mathcal{L} \phi \quad \mathrm{if \ }m=1,\\
\mathcal{L}^m\phi:&=\mathcal{L}(\mathcal{L}^{m-1}\phi)\quad \mathrm{if \ }m>1,
\end{align}
for every $\phi \in C^{\infty}(\G)$.
An explicit computation on the Heisenberg group $\H^n$ shows that
$$\mathcal{L}^m\phi(x)=(-1)^m \sum_{i_1, \ldots, i_m \in \{1, \ldots, 2n\}} Z_{i_1}^2 \ldots  Z_{i_m}^2 \phi(x),$$
for every $m \in \N$ and $\phi \in C^{\infty}(\G)$.\\
According to the continuation described in Sections \ref{sect-indu} and \ref{striscia2} and to the explicit form of the $Z$-Taylor polynomials given by Theorem \ref{explicit}, we recall that $\psi(x,\alpha)$ can be represented for $\alpha \in (-6,Q)$ as follows:
\begin{align}\label{espressione3striscia}
\psi(x,\alpha)= & \int_{\H^n \setminus \B_c(0,1)} \phi(xy)P_{\alpha}(y) dy + \int_{\B_c(0,1)} \Big(\phi(xy)-\mathcal{T}_5(\phi,x)(xy)\Big)P_{\alpha}(y) dy  \notag  \\
&+ \phi(x) \frac{1}{\Gamma(\frac{\alpha}{2})} \frac{1}{\alpha} \sigma(\alpha)- \frac{1}{2}\frac{1}{\alpha+2}\frac{1}{\Gamma(\frac{\alpha}{2})} \mathcal{L}\phi(x)d(\alpha)  \notag  \\
&+ \frac{1}{4!}\frac{1}{\alpha+4} \sum_{\substack{i,j \leq 2n\\i < j}}\Big( Z^2_{i}Z^2_{j} \phi(x)+Z^2_{j}Z^2_{i}\phi(x) +Z_iZ_jZ_iZ_j \phi(x)+ Z_jZ_iZ_jZ_i\phi(x)  \notag \\
&+ Z_iZ_jZ_jZ_i \phi(x)+Z_jZ_iZ_iZ_j\phi(x)\Big)\int_{\partial \B_c(0,1)}  y_{i}^2y_{j}^2P_{\alpha}(y) d\mathcal{S}^{Q-1}_{\infty}(y)\\
&+ \frac{1}{4!} \frac{1}{\alpha+4} \sum_{i \leq 2n} Z^4_{i}\phi(x)\int_{\partial \B_c(0,1)}  y_{i}^4P_{\alpha}(y)  d\mathcal{S}^{Q-1}_{\infty}(y) \notag \\
&+ \frac{1}{2} \frac{1}{\alpha+4}  T^2\phi(x) \int_{\partial \B_c(0,1)}y_{2n+1}^2P_{\alpha}(y)  d\mathcal{S}^{Q-1}_{\infty}(y).  \notag 
\end{align}
For every $i,j \in \{1, \dots, 2n \}$, with $i < j$, there are only two possibilities
\begin{itemize}
\item[(a)] $[Z_i, Z_j]=0$, hence $Z_iZ_j=Z_jZ_i$. In this case we can state the following equality
\begin{align}\label{campicomm}
&Z^2_{i}Z^2_{j} \phi(x)+Z^2_{j}Z^2_{i}\phi(x)+Z_iZ_jZ_iZ_j \phi(x)+ Z_jZ_iZ_jZ_i\phi(x)+ Z_iZ_jZ_jZ_i \phi(x)+Z_jZ_iZ_iZ_j\phi(x) \notag\\
 =& 3Z_i^2Z_j^2\phi(x)+3Z_j^2Z_i^2\phi(x).
\end{align}
\item[(b)] $[Z_i, Z_j]=T$, hence $Z_iZ_j=Z_jZ_i+T$ and $Z_jZ_i=Z_iZ_j-T$.
\end{itemize}
We recall also that $[Z_i,T]=0$ for every $i \in \{1, \dots, 2n \}$. \\
Let us assume that we are in case (b). Thus, we have
\begin{align*}
Z_iZ_jZ_iZ_j\phi(x)=Z_i(Z_iZ_j-T)Z_j\phi(x)&=Z_i^2Z_j^2\phi(x)-Z_iZ_jT\phi(x)\\
Z_jZ_iZ_jZ_i\phi(x)=Z_j(Z_jZ_i+T)Z_i\phi(x)&=Z_j^2Z_i^2\phi(x)+Z_jZ_iT\phi(x)\\
Z_iZ_jZ_jZ_i\phi(x)=(Z_jZ_i+T)Z_jZ_i\phi(x)&=Z_jZ_iZ_jZ_i\phi(x)+Z_jZ_iT\phi(x)=Z_j^2Z_i^2\phi(x)+2Z_jZ_iT\phi(x)\\
Z_jZ_iZ_iZ_j\phi(x)=(Z_iZ_j-T)Z_iZ_j\phi(x)&=Z_iZ_jZ_iZ_j\phi(x)-Z_iZ_jT\phi(x)=Z_i^2Z_j^2\phi(x)-2Z_iZ_jT\phi(x).
\end{align*}
Therefore, 
\begin{align}\label{campiconcomm}
&Z^2_{i}Z^2_{j} \phi(x)+Z^2_{i}Z^2_{j}\phi(x) +Z_iZ_jZ_iZ_j \phi(x)+ Z_jZ_iZ_jZ_i\phi(x) + Z_iZ_jZ_jZ_i \phi(x)+Z_jZ_iZ_iZ_j\phi(x) \notag \\
=&3 Z_i^2Z_j^2\phi(x)+3Z_j^2Z_i^2\phi(x)+3 Z_jZ_iT\phi(x)-3Z_iZ_jT\phi(x) \notag \\
=& 3 Z_i^2Z_j^2\phi(x)+3Z_j^2Z_i^2\phi(x)-3 (Z_iZ_j-Z_jZ_i)T\phi(x)\\
=& 3 Z_i^2Z_j^2\phi(x)+3Z_j^2Z_i^2\phi(x)-3 T^2\phi(x).  \notag 
\end{align}
We can then use \eqref{campicomm} and \eqref{campiconcomm} to rewrite \eqref{espressione3striscia} as
\begin{align}\label{integrale4}
\psi(x,\alpha)= & \int_{\H^n \setminus \B_c(0,1)} \phi(xy)P_{\alpha}(y) dy + \int_{\B_c(0,1)} \Big(\phi(xy)-\mathcal{T}_5(\phi,x)(xy)\Big)P_{\alpha}(y) dy \notag  \\
&+ \phi(x) \frac{1}{\Gamma(\frac{\alpha}{2})} \frac{1}{\alpha} \sigma(\alpha)- \frac{1}{2}\frac{1}{\alpha+2}\frac{1}{\Gamma(\frac{\alpha}{2})} \mathcal{L}\phi(x)d(\alpha)  \notag \\
&+ \frac{1}{4!}\frac{1}{\alpha+4} \sum_{\substack{i,j \leq 2n\\i < j}}\big( Z_i^2Z_j^2\phi(x)+Z_j^2Z_i^2\phi(x)\big)3\int_{\partial \B_c(0,1)}  y_{1}^2y_{2}^2P_{\alpha}(y)  d\mathcal{S}^{Q-1}_{\infty}(y)\\
&+ \frac{1}{4!}\frac{1}{\alpha+4} \sum_{i \leq 2n} Z^4_{i}\phi(x)\int_{\B_c(0,1)}  y_{1}^4P_{\alpha}(y) dy \notag \\
&+   T^2\phi(x) \frac{1}{\alpha+4} \left(\frac{1}{2} \int_{\partial \B_c(0,1)}y_{2n+1}^2P_{\alpha}(y)  d\mathcal{S}^{Q-1}_{\infty}(y) -\frac{3n}{4!} \int_{\B_c(0,1)}y_1^2y_2^2P_{\alpha}(y)  d\mathcal{S}^{Q-1}_{\infty}(y) \right). \notag 
\end{align}
In order to write \eqref{integrale4}, we used the fact that there are $n$ couples of indexes $i,j \in \{1, \ldots, 2n\}$ satisfying case (b). Moreover, we have exploited the invariance by horizontal rotation of $\B_c(0,1)$ and of the $P_{\alpha}$, which ensure that the equalities
\begin{align}\label{serve7} \int_{\partial \B_c(0,1)} y_i^4 P_{\alpha}(y)d\mathcal{S}^{Q-1}_{\infty}(y)= \int_{\partial \B_c(0,1)} y_j^4 P_{\alpha}(y)d\mathcal{S}^{Q-1}_{\infty}(y) \qquad \mathrm{for  \ every \ }i,j \in \{1, &\dots, 2n \},\\
\int_{\partial \B_c(0,1)} y_i^2y_j^2 P_{\alpha}(y)d\mathcal{S}^{Q-1}_{\infty}(y) = \int_{\partial \B_c(0,1)} y_{\ell}^2y_m^2 P_{\alpha}(y) d\mathcal{S}^{Q-1}_{\infty}(y) \qquad\mathrm{for  \ every \ }i \neq j, \ell \neq m \in &\{1, \dots, 2n \} \notag.
\end{align}
hold.

Now, by applying Lemma \ref{lemma-general} for an arbitrary $\alpha < Q$, we have that
\begin{align*}
\int_{\partial \B_c(0,1)}  y_{1}^4 \|y\|_{\alpha}^{\alpha-Q} d\mathcal{S}^{Q-1}_{\infty}(y)&= 2  \int_{\H^n} z_1^4 h(1,z) \|z\|_c^{-\alpha-4} dz , \\
\int_{\partial \B_c(0,1)}y_{2n+1}^2\|y\|_{\alpha}^{\alpha-Q} d\mathcal{S}^{Q-1}_{\infty}(y)&=2\int_{\H^n} |\pi_{H_2}(z)|^2 h(1,z) \|z\|_c^{-\alpha-4} dz;
\end{align*}
therefore, for $\alpha=-4$ we obtain
\begin{align}\label{alfa4-2}
\int_{\partial \B_c(0,1)}  y_{1}^4 \|y\|_{\alpha}^{-4-Q} d\mathcal{S}^{Q-1}_{\infty}(y)&= 2\int_{\H^n} z_1^4 h(1,z) dz\\
\int_{\partial \B_c(0,1)}y_{2n+1}^2\|y\|_{\alpha}^{-4-Q} d\mathcal{S}^{Q-1}_{\infty}(y)&=2\int_{\H^n} |\pi_{H_2}(z)|^2 h(1,z) dz.
\end{align} 
As a consequence, we deduce that for every $-6<\alpha<Q$ the following representation holds
\begin{align}\label{eq:repre}
\psi(x, \alpha)&= \int_{\H^n \setminus \B_c(0,1)} \phi(xy)P_{\alpha}(y) dy + \int_{\B_c(0,1)} \Big(\phi(xy)-\mathcal{T}_5(\phi,x)(xy)\Big)P_{\alpha}(y) dy \notag \\
&+ \phi(x) \frac{1}{\Gamma(\frac{\alpha}{2})} \frac{1}{\alpha} \sigma(\alpha)- \frac{1}{2}\frac{1}{\alpha+2}\frac{1}{\Gamma(\frac{\alpha}{2})} \mathcal{L}\phi(x)d(\alpha) \notag \\
& + \frac{1}{\alpha+4} \frac{1}{\Gamma(\frac{\alpha}{2})} \Bigg( \frac{1}{4!}\sum_{\substack{i,j \leq 2n\\i < j}}\Big( Z_i^2Z_j^2\phi(x)+Z_j^2Z_i^2\phi(x)\Big)6  \int_{\H^n}  y_{1}^2y_{2}^2 h(1,y) \|y\|_c^{-\alpha-4} dy \\
&+ \frac{2}{4!} \sum_{i \leq 2n} Z^4_{i}\phi(x) \int_{\H^n}  y_{1}^4 h(1,y) \|y\|_c^{-\alpha-4} dy \notag\\
&+  T^2\phi(x) \bigg( \int_{\H^n}y_{2n+1}^2 h(1,y) \|y\|_c^{-\alpha-4} dy  -\frac{6n}{4!}\int_{\H^n}  y_{1}^2y_2^2 h(1,y) \|y\|_c^{-\alpha-4} dy  \bigg) \Bigg).\notag
\end{align}
In order to obtain a simpler representation, we need the following result.
\begin{prop}\label{serve8}
For every $\alpha<Q$ the following equality holds
\begin{equation} \label{tesi-bil}
 \int_{\H^n} y_1^4 h(1,y) \|y\|_c^{-\alpha-4}dy= 3 \int_{\H^n} y_1^2y_2^2 h(1,y) \|y\|_c^{-\alpha-4} dy.
\end{equation}
\end{prop}
\begin{proof}
Following \cite[Section 4]{B09} (see also \cite{Gre80}), we first introduce suitable polar coordinates in $\H^n \setminus H_2$. 
\begin{align*}
T: (0,\infty) \times \left(-\frac{\pi}{2},\frac{\pi}{2} \right) \times [0, \pi)^{2n-2} \times [0,2 \pi) &\to \H^n \setminus H_2\\
(\rho, \varphi, \theta_1 , \ldots, \theta_{2n-1} )  &\to T(\rho, \varphi, \theta_1 , \ldots, \theta_{2n-1} ),
\end{align*}
with
\begin{align*}
T(\rho, \varphi, \theta_1 , \ldots, \theta_{2n-1} ) =   \begin{bmatrix} \rho \sqrt{\cos\varphi} \cos \theta_1,\\
 \rho \sqrt{\cos\varphi} \sin \theta_1 \cos \theta_2,\\
  \rho \sqrt{\cos\varphi} \sin \theta_1 \sin \theta_2 \cos \theta_3,\\
  \ldots \\
    \rho \sqrt{\cos\varphi} \sin \theta_1 \sin \theta_2  \ldots \sin \theta_{2n-2} \cos \theta_{2n-1},\\
        \rho \sqrt{\cos\varphi} \sin \theta_1 \sin \theta_2  \ldots \sin \theta_{2n-2} \sin \theta_{2n-1}\\
        \rho^2 \sin \varphi
        \end{bmatrix}^T.
\end{align*}
According to our purposes, we do not need to an explicit computation of the Jacobian of the transformation $T$. We just notice that, by the computations in \cite[Section 4]{B09}, the Jacobian does not depend on $\theta_{2n-1}$ and we denote it by $$JT_{\overline{\theta_{2n-1}}}(\rho, \varphi, \theta_1, \dots, \theta_{2n-2}, \theta_{2n-1}),$$ in order to emphasize its independence of $\theta_{2n-1}$.
By the invariance of $h(1,\cdot)$ with respect to horizontal rotations (see \eqref{serve7}), to prove \eqref{tesi-bil} means to prove that
\begin{equation}\label{tesirota}
 \int_{\H^n} y_{2n}^4 h(1,y) \|y\|_c^{-\alpha-4}dy= 3 \int_{\H^n} y_{2n-1}^2y_{2n}^2 h(1,y) \|y\|_c^{-\alpha-4} dy.
\end{equation}
Let us first perform the change of variables $T$ on the integral appearing on the right-hand side of \eqref{tesirota} as follows
\begin{small}
\begin{align*}
&\int_{\H^n} y_{2n-1}^2y_{2n}^2 h(1,y) \|y\|_c^{-\alpha-4} dy=\int_{\H^n \setminus H_2} y_{2n-1}^2y_{2n}^2 h(1,y) \|y\|_c^{-\alpha-4} dy\\
=& \int_0^{\infty} \int_{-\frac{\pi}{2}}^{\frac{\pi}{2}} \int_0^{\pi} \ldots \int_0^{\pi} \int_0^{2\pi}  \rho^4 \cos^2\varphi \sin^4 \theta_1 \sin^4 \theta_2  \ldots \sin^4 \theta_{2n-2} \cos^2 \theta_{2n-1} \sin^2 \theta_{2n-1}\\
&  h(1, T(\rho, \varphi, \theta_1, \ldots, \theta_{2n-1}) ) \|T(\rho, \varphi, \theta_1, \ldots, \theta_{2n-1})\|_c^{-\alpha-4} JT_{\overline{\theta_{2n-1}}}(\rho, \varphi, \theta_1, \dots, \theta_{2n-2}, \theta_{2n-1}) d \theta_{2n-1} \ldots d \theta_{1} d \varphi d \rho\\
=& \int_0^{\infty} \rho^4 \int_{-\frac{\pi}{2}}^{\frac{\pi}{2}} \cos^2\varphi  \int_0^{\pi}  \sin^4 \theta_1  \ldots \int_0^{\pi} \sin^4 \theta_{2n-2}  JT_{\overline{\theta_{2n-1}}}(\rho, \varphi, \theta_1, \dots, \theta_{2n-2}, \theta_{2n-1})  \\
& \int_0^{2\pi}    \cos^2 \theta_{2n-1} \sin^2 \theta_{2n-1}  h(1, T(\rho, \varphi, \theta_1, \ldots, \theta_{2n-1}) ) \|T(\rho, \varphi, \theta_1, \ldots, \theta_{2n-1})\|_c^{-\alpha-4} d \theta_{2n-1} d \theta_{2n-2}  \ldots d \theta_{1} d \varphi d \rho.
\end{align*}
\end{small}
Now we apply the change of variables $T$ on the integral on the left-hand side of \eqref{tesirota} and we get
\begin{small}
\begin{align*}
&\int_{\H^n} y_{2n}^4 h(1,y) \|y\|_c^{-\alpha-4} dy=\int_{\H^n \setminus H_2} y_{2n}^4 h(1,y) \|y\|_c^{-\alpha-4} dy\\
=& \int_0^{\infty} \int_{-\frac{\pi}{2}}^{\frac{\pi}{2}} \int_0^{\pi}  \ldots \int_0^{\pi} \int_0^{2\pi}  \rho^4 \cos^2\varphi \sin^4 \theta_1 \sin^4 \theta_2  \ldots \sin^4 \theta_{2n-2}  \sin^4 \theta_{2n-1}\\
&   h(1, T(\rho, \varphi, \theta_1, \ldots, \theta_{2n-1}) ) \|T(\rho, \varphi, \theta_1, \ldots, \theta_{2n-1})\|_c^{-\alpha-4} JT_{\overline{\theta_{2n-1}}}(\rho, \varphi, \theta_1, \dots, \theta_{2n-2}, \theta_{2n-1}) d \theta_{2n-1} \ldots d \theta_{1} d \varphi d \rho\\
=& \int_0^{\infty} \rho^4 \int_{-\frac{\pi}{2}}^{\frac{\pi}{2}} \cos^2\varphi  \int_0^{\pi}  \sin^4 \theta_1  \ldots \int_0^{\pi} \sin^4 \theta_{2n-2}  JT_{\overline{\theta_{2n-1}}}(\rho, \varphi, \theta_1, \dots, \theta_{2n-2}, \theta_{2n-1})  \\
& \int_0^{2\pi}    \sin^4 \theta_{2n-1}  h(1, T(\rho, \varphi, \theta_1, \ldots, \theta_{2n-1}) ) \|T(\rho, \varphi, \theta_1, \ldots, \theta_{2n-1})\|_c^{-\alpha-4} d \theta_{2n-1}d \theta_{2n-2}  \ldots d \theta_{1} d \varphi d \rho .
\end{align*}
\end{small}
Therefore, in order to prove \eqref{tesirota}, it is enough to prove that, for fixed $\rho, \varphi, \theta_1, \ldots, \theta_{2n-2} \in (0 ,\infty) \times \left(-\frac{\pi}{2}, \frac{\pi}{2} \right) \times (0, \pi)^{2n-2}$ it holds that
\begin{equation}\label{nuovatesi}
\begin{split}
 &\int_0^{2\pi}    \sin^4 \theta_{2n-1} h(1, T(\rho, \varphi, \theta_1, \ldots, \theta_{2n-1}) ) \|T(\rho, \varphi, \theta_1, \ldots, \theta_{2n-1})\|_c^{-\alpha-4} d \theta_{2n-1} \\
 =& \ 3 \int_0^{2\pi}    \cos^2 \theta_{2n-1} \sin^2 \theta_{2n-1} h(1, T(\rho, \varphi, \theta_1, \ldots, \theta_{2n-1}) ) \|T(\rho, \varphi, \theta_1, \ldots, \theta_{2n-1})\|_c^{-\alpha-4}d \theta_{2n-1}.
 \end{split}
\end{equation}
Notice that $\| \cdot \|_c$ is a $C^1$ map on $\H^n \setminus H_2$ (see for example \cite[Lemma 3.11]{AR02}) and that $h(1,z)$ is $C^{\infty}$ on $\H^n$. Now, exploiting the integration by parts, we get
\begin{align*}
&\int_0^{2\pi}    \cos^2 \theta_{2n-1} \sin^2 \theta_{2n-1} h(1, T(\rho, \varphi, \theta_1, \ldots, \theta_{2n-1}) \|T(\rho, \varphi, \theta_1, \ldots, \theta_{2n-1})\|_c^{-\alpha-4}d \theta_{2n-1}\\
=&\left[ \frac{\sin^3 \theta_{2n-1}}{3} \cos \theta_{2n-1} h(1, T(\rho, \varphi, \theta_1, \ldots, \theta_{2n-1}) \|T(\rho, \varphi, \theta_1, \ldots, \theta_{2n-1})\|_c^{-\alpha-4} \right]_{\theta_{2n-1}=0}^{\theta_{2n-1}=2\pi} \\
&- \int_0^{2\pi}   \frac{\sin^3 \theta_{2n-1}}{3} \Bigg(- \sin \theta_{2n-1}  h(1, T(\rho, \varphi, \theta_1, \ldots, \theta_{2n-1}))\|T(\rho, \varphi, \theta_1, \ldots, \theta_{2n-1})\|_c^{-\alpha-4}\\
&\qquad \qquad + \cos \theta_{2n-1}  \frac{d}{d \theta_{2n-1}} \Big(h(1, T(\rho, \varphi, \theta_1, \ldots, \theta_{2n-1})) \|T(\rho, \varphi, \theta_1, \ldots, \theta_{2n-1})\|_c^{-\alpha-4} \Big) \Bigg)d \theta_{2n-1}\\
=& \int_0^{2\pi}   \frac{\sin^4 \theta_{2n-1}}{3}   h(1, T(\rho, \varphi, \theta_1, \ldots, \theta_{2n-1})) \|T(\rho, \varphi, \theta_1, \ldots, \theta_{2n-1})\|_c^{-\alpha-4} d \theta_{2n-1}\\
& - \int_0^{2\pi}  \frac{\sin^3 \theta_{2n-1}}{3}\cos \theta_{2n-1}  \frac{d}{d \theta_{2n-1}} \Big(h(1, T(\rho, \varphi, \theta_1, \ldots, \theta_{2n-1})) \|T(\rho, \varphi, \theta_1, \ldots, \theta_{2n-1})\|_c^{-\alpha-4} \Big) d \theta_{2n-1}.
\end{align*}
Thus, \eqref{nuovatesi} follows if we prove that
\begin{equation*}
\int_0^{2\pi}  \frac{\sin^3 \theta_{2n-1}}{3}\cos \theta_{2n-1}  \frac{d}{d \theta_{2n-1}} \Big(h(1, T(\rho, \varphi, \theta_1, \ldots, \theta_{2n-1})) \|T(\rho, \varphi, \theta_1, \ldots, \theta_{2n-1})\|_c^{-\alpha-4} \Big)d \theta_{2n-1}=0.
\end{equation*}
We will show the previous equality by showing that 
\begin{equation}\label{nuovatesi3}
 \frac{d}{d \theta_{2n-1}} \Big(h(1, T(\rho, \varphi, \theta_1, \ldots, \theta_{2n-1}))\|T(\rho, \varphi, \theta_1, \ldots, \theta_{2n-1})\|_c^{-\alpha-4}\Big)=0
\end{equation}
for every $
\rho, \varphi, \theta_1 , \ldots, \theta_{2n-1}  \in (0,\infty) \times \left( -\frac{\pi}{2}, \frac{\pi}{2} \right) \times (0,\pi)^{2n-2} \times (0, 2\pi)$. 
In particular, it is enough to notice that, once fixed $\bar{\rho}, \bar{\varphi}, \bar{\theta}_1 , \ldots, \bar{\theta}_{2n-2}  \in (0,\infty) \times \left( -\frac{\pi}{2}, \frac{\pi}{2} \right) \times [0,\pi)^{2n-2}$, the map
$$[0, 2 \pi) \ni \theta_{2n+1} \to T(\bar{\rho}, \bar{\varphi}, \bar{\theta}_1, \ldots, \bar{\theta}_{2n-2}, \theta_{2n-1})$$
is a horizontal rotation. On the other hand, we argue as follows
\begin{itemize}
\item the heat kernel $h(t,y)$ is invariant with respect to horizontal rotations, as stated in Section \ref{parita-h}, thus $h(1, (z_1,\tau_1))=h(1, (z_2, \tau_2))$ if $|z_1|=|z_2|$, for every $(z_1, \tau_1), (z_2, \tau_2) \in \H^n$, with $z_1, z_2 \in H_1 \simeq \R^{2n}$ and $ \tau_1, \tau_2 \in H_2 \simeq \R$. 
\item the norm $\| \cdot \|_c$ is invariant with respect to horizontal rotations, as stated in Section \ref{parita-ball}, thus if we have $(z_1, \tau), (z_2, \tau) \in \H^n$, with $|z_1|=|z_2|$, then $\|(z_1, \tau)\|_c=\|(z_2, \tau)\|_c$.
\end{itemize}
Hence for every $(\bar{\rho}, \bar{\varphi}, \bar{\theta}_1 , \ldots, \bar{\theta}_{2n-2} ) \in (0,\infty) \times \left( -\frac{\pi}{2}, \frac{\pi}{2} \right) \times [0,\pi)^{2n-2}$, the map
$$ (0, 2 \pi) \ni \theta_{2n+1} \to h(1, T(\bar{\rho}, \bar{\varphi}, \bar{\theta}_1, \ldots, \theta_{2n-1})) \|T(\bar{\rho}, \bar{\varphi}, \bar{\theta}_1, \ldots, \theta_{2n-1})\|_c^{-\alpha-4} $$ is a constant map, then \eqref{nuovatesi3} holds. This concludes the proof.
\end{proof}
By combining Proposition \ref{serve8} with Lemma \ref{lemma-general}, the following corollary holds.

\begin{cor}\label{cor-bil}
For every $\alpha<Q$ the following equality holds
\begin{equation} 
 \int_{\partial \B_c(0,1)} y_1^4 \|y\|_{\alpha}^{\alpha-Q}  d\mathcal{S}^{Q-1}_{\infty}(y)= 3  \int_{\partial \B_c(0,1)} y_1^2 y_2^2 \|y\|_{\alpha}^{\alpha-Q} d\mathcal{S}^{Q-1}_{\infty}(y).
\end{equation}
\end{cor}

By combining Corollary \ref{cor-bil} and \eqref{eq:repre} we have that for $\alpha \in (-6,Q)$
\begin{align}\label{exp3}
&\psi(x,\alpha)=    \int_{\H^n \setminus \B_c(0,1)} \phi(xy)P_{\alpha}(y) dy + \int_{\B_c(0,1)} \Big(\phi(xy)-\mathcal{T}_5(\phi,x)(xy)\Big)P_{\alpha}(y) dy \notag  \\
&+ \phi(x) \frac{1}{\Gamma(\frac{\alpha}{2})} \frac{1}{\alpha} \sigma(\alpha)- \frac{1}{2}\frac{1}{\alpha+2}\frac{1}{\Gamma(\frac{\alpha}{2})} \mathcal{L}\phi(x)d(\alpha) \notag   \\
& + \frac{1}{\alpha+4} \frac{1}{\Gamma(\frac{\alpha}{2})} \Bigg( \frac{1}{4!}\sum_{\substack{i,j \leq 2n\\i < j}}( Z_i^2Z_j^2\phi(x)+Z_j^2Z_i^2\phi(x))3  \int_{\partial \B_c(0,1)}  y_{1}^2y_{2}^2\|y\|_{\alpha}^{\alpha-Q} d\mathcal{S}^{Q-1}_{\infty}(y)  \notag  \\
&+ \frac{1}{4!} \sum_{i \leq 2n} Z^4_{i}\phi(x)\int_{\partial \B_c(0,1)}  y_{1}^4 \|y\|_{\alpha}^{\alpha-Q} d\mathcal{S}^{Q-1}_{\infty}(y)  \notag   \\
&+  T^2\phi(x) \bigg(\frac{1}{2} \int_{\partial \B_c(0,1)}y_{2n+1}^2\|y\|_{\alpha}^{\alpha-Q} d\mathcal{S}^{Q-1}_{\infty}(y)  -\frac{3n}{4!} \int_{\partial \B_c(0,1)}y_1^2y_2^2\|y\|_{\alpha}^{\alpha-Q} d\mathcal{S}^{Q-1}_{\infty}(y)  \bigg) \Bigg) \notag  \\
=& \int_{\H^n \setminus \B_c(0,1)} \phi(xy)P_{\alpha}(y) dy + \int_{\B_c(0,1)} \Big(\phi(xy)-\mathcal{T}_5(\phi,x)(xy)\Big)P_{\alpha}(y) dy \notag   \\
&+ \phi(x) \frac{1}{\Gamma(\frac{\alpha}{2})} \frac{1}{\alpha} \sigma(\alpha)- \frac{1}{2}\frac{1}{\alpha+2}\frac{1}{\Gamma(\frac{\alpha}{2})} \mathcal{L}\phi(x)d(\alpha) \notag   \\
&+  \frac{1}{\alpha+4} \frac{1}{\Gamma(\frac{\alpha}{2})} \Bigg( \frac{1}{4!}\sum_{\substack{i,j \leq 2n\\i < j}}\Big( Z_i^2Z_j^2\phi(x)+Z_j^2Z_i^2\phi(x)+ Z^4_{i}\phi(x) + Z_j^4\phi(x)\Big) \int_{\partial \B_c(0,1)}  y_{1}^4 \|y\|_{\alpha}^{\alpha-Q} d\mathcal{S}^{Q-1}_{\infty}(y)  \notag  \\
&+  T^2\phi(x) \bigg(\frac{1}{2}  \int_{\partial \B_c(0,1)}y_{2n+1}^2\|y\|_{\alpha}^{\alpha-Q} d\mathcal{S}^{Q-1}_{\infty}(y) -\frac{n}{4!} \int_{\partial \B_c(0,1)}y_1^4\|y\|_{\alpha}^{\alpha-Q} d\mathcal{S}^{Q-1}_{\infty}(y)  \bigg) \Bigg) \notag  \\
=& \int_{\H^n \setminus \B_c(0,1)} \phi(xy)P_{\alpha}(y) dy + \int_{\B_c(0,1)} \Big(\phi(xy)-\mathcal{T}_5(\phi,x)(xy)\Big)P_{\alpha}(y) dy   \\
&+ \phi(x) \frac{1}{\Gamma(\frac{\alpha}{2})} \frac{1}{\alpha} \sigma(\alpha)- \frac{1}{2}\frac{1}{\alpha+2}\frac{1}{\Gamma(\frac{\alpha}{2})} \mathcal{L}\phi(x)d(\alpha)  \notag  \\
&+  \frac{1}{\alpha+4} \frac{1}{\Gamma(\frac{\alpha}{2})} \Bigg( \frac{1}{4!}\sum_{i,j \leq 2n}\Big( Z_i^2Z_j^2\phi(x)\Big) \int_{\partial \B_c(0,1)}  y_{1}^4 \|y\|_{\alpha}^{\alpha-Q} d\mathcal{S}^{Q-1}_{\infty}(y) \notag   \\
&+  T^2\phi(x) \bigg(\frac{1}{2}  \int_{\partial \B_c(0,1)}y_{2n+1}^2\|y\|_{\alpha}^{\alpha-Q} d\mathcal{S}^{Q-1}_{\infty}(y) -\frac{n}{4!} \int_{\partial \B_c(0,1)}y_1^4\|y\|_{\alpha}^{\alpha-Q} d\mathcal{S}^{Q-1}_{\infty}(y) \bigg) \Bigg). \notag  \\
=& \int_{\H^n \setminus \B_c(0,1)} \phi(xy)P_{\alpha}(y) dy + \int_{\B_c(0,1)} (\phi(xy)-\mathcal{T}_5(\phi,x)(xy))P_{\alpha}(y) dy  \notag   \\
&+ \phi(x) \frac{1}{\Gamma(\frac{\alpha}{2})} \frac{1}{\alpha} \sigma(\alpha)- \frac{1}{2}\frac{1}{\alpha+2}\frac{1}{\Gamma(\frac{\alpha}{2})} \mathcal{L}\phi(x)d(\alpha)  \notag   \\
&+  \frac{1}{\alpha+4} \frac{1}{\Gamma(\frac{\alpha}{2})} \Bigg( \frac{1}{4!} \mathcal{L}^2 \phi(x) \int_{\partial \B_c(0,1)}  y_{1}^4 \|y\|_{\alpha}^{\alpha-Q} d\mathcal{S}^{Q-1}_{\infty}(y) \notag   \\
&+   T^2\phi(x) \bigg(\frac{1}{2}  \int_{\partial \B_c(0,1)}y_{2n+1}^2\|y\|_{\alpha}^{\alpha-Q} d\mathcal{S}^{Q-1}_{\infty}(y)  -\frac{n}{4!} \int_{\partial \B_c(0,1)}y_1^4\|y\|_{\alpha}^{\alpha-Q} d\mathcal{S}^{Q-1}_{\infty}(y)  \bigg) \Bigg). \notag  
\end{align}

\subsection{A geometric application, II}
We use the representation obtained in the previous subsection in \eqref{exp3} to deduce the value of some integrals.

\begin{prop}\label{prop:valint2}
Let $n \in \N$ and let $h$ be the heat kernel associated with $\mathcal{L}$ on $\H^n$, then for every $i \in \{1, \ldots, 2n \}$ the following equalities hold
\begin{equation}\label{val-int2}
\int_{\H^n} x_i^4 h(1, x) dx=12 \qquad \mathrm{and} \qquad \int_{\H^n} x_{2n+1}^2 h(1, x) dx=n.
\end{equation}
\end{prop}
\begin{proof}
According to the representation of $\psi(x,\alpha)$ on $(-6,Q)$ obtained in \eqref{exp3}, exploiting the definition of real power of the fractional sub-Laplacian in \eqref{realpow}, we get that
\begin{equation}\label{L2}
\begin{split}
\mathcal{L}^2\phi(x)&=
\lim_{\alpha \to -4} \psi(x, \alpha)\\
&= \frac{2}{4!} \mathcal{L}^2 \phi(x) \int_{\H^n} y_1^4 h(1,y) dy+  T^2\phi(x) \bigg(  \int_{\H^n}y_{2n+1}^2 h(1,y)dy  -\frac{2n}{4!} \int_{\H^n}y_1^4 h(1,y) dy\bigg) .
\end{split}
\end{equation}
The equalities in \eqref{val-int2} will follow from \eqref{L2} if we show that it does not exist any constant $c \in \R$ such that the equality
\begin{equation}\label{assumption2}
\mathcal{L}^2\phi(x)=c T^2\phi(x) 
\end{equation}
holds for every $\phi \in \mathcal{S}(\H^n)$ and $x \in \H^n$.
Let us assume by contradiction that \eqref{assumption2} holds and let us denote by $d_{\lambda}$, for every $\lambda>0$, the isotropic dilation by $\lambda$, i.e. $$d_{\lambda}(x_1, \ldots, x_{2n+1})=(\lambda x_1, \ldots, \lambda x_{2n+1}),$$
for every $x=(x_1, \ldots, x_{2n+1})\in \H^n$.
Now, we observe that for every $\phi \in C^{\infty}(\H^n)$ and $x \in \H^n$, for every $\lambda>0$ the following homogeneity property holds
$$T^2 (\phi \circ d_{\lambda})(x))=\lambda^2 T^2\phi(d_{\lambda}(x)).$$
Taking in consideration that, if $\phi \in \mathcal{S}(\H^n)$, for every $\lambda>0$ the map $\phi \circ d_{\lambda} \in \mathcal{S}(\H^n)$, then we can apply \eqref{assumption2} to get that
\begin{align}\label{serve9}
\mathcal{L}^2(\phi \circ d_{\lambda})(x)&=cT^2(\phi\circ d_{\lambda})(x)=c \lambda^2 T^2\phi(d_{\lambda}(x))=\lambda^2 \mathcal{L}^2\phi(d_{\lambda}(x)),
\end{align}
for every $x \in \H^n$.
On the other hand, by the homogeneity property of $\mathcal{L}^2$ with respect to the group dilations $\delta_{\lambda}$ we can continue from \eqref{serve9} getting
\begin{align*}
\mathcal{L}^2(\phi \circ d_{\lambda})(x)&=\lambda^2 \mathcal{L}^2\phi(d_{\lambda}(x))\\
&= (\sqrt{\lambda})^4\mathcal{L}^2\phi(\delta_{\sqrt{\lambda}}(\delta_{\frac{1}{\sqrt{\lambda}}}d_{\lambda}(x))\\
&=\mathcal{L}^2(\phi \circ \delta_{\sqrt{\lambda}})(\delta_{\frac{1}{\sqrt{\lambda}}}(d_{\lambda}(x)))\\
&=\mathcal{L}^2(\phi \circ \delta_{\sqrt{\lambda}})(\sqrt{\lambda}\pi_{H_1}(x), \pi_{H_2}(x)).
\end{align*}
Thus, we established the relation 
\begin{equation}\label{castelvetro}
\mathcal{L}^2(\phi \circ d_{\lambda})(x)=\mathcal{L}^2(\phi \circ \delta_{\sqrt{\lambda}})(\sqrt{\lambda}\pi_{H_1}(x), \pi_{H_2}(x)),
\end{equation}
for every $\phi \in \mathcal{S}(\H^n)$, $x \in \H^n$ and $\lambda>0$.
Let us show that the homogeneity property obtained in \eqref{castelvetro} cannot be true. By straightforward computations it is possible to verify that for every $\lambda>0$, $x \in \H^n$ and $\phi \in C^{\infty}(\H^n)$ we have the following representation
\begin{align*}
&\mathcal{L}^2(\phi \circ d_{\lambda})(x) = \mathcal{L}(\mathcal{L}(\phi \circ d_{\lambda}))(x)\\
=& \ \mathcal{L}\Big(\lambda^2(\Delta \phi \circ d_{\lambda})+\lambda^2 \sum_{i=1}^n \big( x_i (\partial_{y_i}\partial_t \phi \circ d_{\lambda})-y_i (\partial_{x_i}\partial_t \phi \circ d_{\lambda})+\frac{1}{4}  (x_i^2+y_i^2)(\partial_t^2 \phi \circ d_{\lambda}) \big) \Big)(x)\\
=& \ \lambda^2\mathcal{L}(\Delta \phi \circ d_{\lambda})(x)+\lambda^2 \sum_{i=1}^n \mathcal{L}\big( x_i (\partial_{y_i}\partial_t \phi \circ d_{\lambda})-y_i (\partial_{x_i}\partial_t \phi \circ d_{\lambda})\big)(x)\\
&+\frac{1}{4} \lambda^2   \sum_{i=1}^n \mathcal{L}((x_i^2+y_i^2)(\partial_t^2 \phi \circ d_{\lambda}))(x)\\
=& \ \lambda^4 \Delta^2 \phi(d_{\lambda}(x))+ n \lambda^2 \partial_t^2  \phi(d_{\lambda}(x))+P_1(\lambda, x, \phi),
\end{align*}
where the second addend in the last equality comes from the computation of the term $\frac{1}{4} \lambda^2   \sum_{i=1}^n \mathcal{L}((x_i^2+y_i^2)(\partial_t^2 \phi \circ d_{\lambda}))(x)$ and $P_1(\lambda,x,\phi)=o(\lambda^2)$ as $\lambda \to 0^+.$ Here, we have denoted by $\Delta$ the operator
$$\Delta= \sum_{i=1}^n (\partial_{x_i}^2+\partial_{y_i}^2).$$
On the other hand, we have as well for every $z \in \H^n$
\begin{align*}
&\mathcal{L}^2(\phi \circ \delta_{\sqrt{\lambda}})(z) = \mathcal{L}(\mathcal{L}(\phi \circ \delta_{\sqrt{\lambda}}))(z)\\
= & \ \mathcal{L}\Big(\lambda(\Delta \phi \circ \delta_{\sqrt{\lambda}})+\lambda \sqrt{\lambda} \sum_{i=1}^n \big( x_i (\partial_{y_i}\partial_t \phi \circ \delta_{\sqrt{\lambda}})-y_i (\partial_{x_i}\partial_t \phi \circ \delta_{\sqrt{\lambda}}))+\frac{1}{4} \lambda^2 \sum_{i=1}^n \big((x_i^2+y_i^2)(\partial_t^2 \phi \circ \delta_{\sqrt{\lambda}})\big) \Big)(z)\\
=& \ \lambda \mathcal{L}(\Delta \phi \circ \delta_{\sqrt{\lambda}})(z)+\lambda \sqrt{\lambda} \sum_{i=1}^n \mathcal{L}\big( x_i (\partial_{y_i}\partial_t \phi \circ \delta_{\sqrt{\lambda}})-y_i (\partial_{x_i}\partial_t \phi \circ d_{\lambda})\big)(z)\\
&+\frac{1}{4} \lambda^2   \sum_{i=1}^n \mathcal{L}((x_i^2+y_i^2)(\partial_t^2 \phi \circ \delta_{\sqrt{\lambda}}))(z)\\
= & \ \lambda^2 \Delta^2 \phi(\delta_{\sqrt{\lambda}}(z))+ n \lambda^2 \partial_t^2  \phi(\delta_{\sqrt{\lambda}}(z))+P_2(\lambda, z, \phi),
\end{align*}
where the second addend in the last equality comes from the computation of the term $\frac{1}{4} \lambda^2   \sum_{i=1}^n \mathcal{L}((x_i^2+y_i^2)(\partial_t^2 \phi \circ \delta_{\sqrt{\lambda}}))(z)$ and $P_2(\lambda,z,\phi)=o(\lambda^2)$ as $\lambda \to 0^+.$ 
When we evaluate the previous operator at a point of the form $z=(\sqrt{\lambda}\pi_{H_1}(x), \pi_{H_2}(x))$, for an arbitrary $x \in \H^n$, we obtain
\begin{equation*}
\mathcal{L}^2(\phi \circ \delta_{\sqrt{\lambda}})(\sqrt{\lambda}\pi_{H_1}(x), \pi_{H_2}(x))=   \lambda^2 \Delta^2 \phi(d_{\lambda}(x))+ n \lambda^2 \partial_t^2  \phi(d_{\lambda}(x))+P_2(\lambda, x, \phi).
\end{equation*}

Therefore, if \eqref{castelvetro} holds, for every $\lambda>0$, $x \in \H^n$ and $\phi \in \mathcal{S}(\H^n)$ we also have
\begin{equation}\label{eq:dadiv}
 \lambda^4 \Delta^2 \phi(d_{\lambda}(x))+ n \lambda^2 \partial_t^2  \phi(\delta_{\lambda}(x))+P_1(\lambda, x, \phi)=\lambda^2 \Delta^2 \phi(d_{\lambda}(x))+ n \lambda^2 \partial_t^2  \phi(d_{\lambda}(x))+P_2(\lambda, x, \phi).
\end{equation}
Now, dividing both sides of \eqref{eq:dadiv} by $\lambda^2$ and letting $\lambda$ go to zero, we deduce that for every $\phi \in \mathcal{S}(\H^n)$
\begin{equation}\label{spilamberto}
\Delta^2 \phi(0)=0,
\end{equation}
but \eqref{spilamberto} does not hold because if $\phi(x)=e^{-\|x\|_K^4} \in \mathcal{S}(\H^n)$, where for $x \in \H^n$ we denoted by $$\| x \|_K=\sqrt[4]{|\pi_{H_1}(x)|^4+|\pi_{H_2}(x)|^2}$$ the norm arising from the Cygan-Kor\'anyi distance \cite{Cygan81}, then by straightforward computations $\phi$ does not satisfy \eqref{spilamberto}. Hence, we reached a contradiction with \eqref{assumption2} and the thesis follows.
\end{proof}

\section{A representation of $\psi(x,\alpha)$ on $(-8, Q)$} 
\label{striscia4}
We give a representation of $\psi(x,\alpha)$ in the strip $(-8,Q)$ involving $\mathcal{L}^3$.

\begin{prop}\label{prop3.6}
If $\B_E(0,1)$ is the Euclidean unit closed ball in $\R^n$ and $\gamma=(\gamma_1, \ldots, \gamma_{n})\in (\N \cup \{ 0 \})^{n}$ is an even $n$-multi-index, then there is a constant $c(n,|\gamma|)>0$ such that
\begin{equation} \label{serve10}
\int_{\B_E(0,1)} x^{\gamma} d x= c(n,|\gamma|)\frac{\gamma!}{\left( \frac{\gamma}{2}\right)!},
\end{equation} 
where by $dx$ we denote the integration with respect to the $n$-dimensional Lebesgue measure on $\R^n$.
\end{prop}

The proof follows from \cite[Proposition 3.6]{B09}.
Let us recall the map $\pi_{H_1}:\H^n \to \R^{2n}, \pi_{H_1}(x_1, \dots, x_{2n+1}) \to (x_1, \dots, x_{2n}), $ and moreover the following proposition holds.
\begin{prop}\label{prop-rel-integrali}
For every even $2n$-multi-index $\gamma \in (\N \cup \{0\})^{2n}$ and for every $\alpha<Q$ the equality
\begin{equation}
\int_{\partial \B_c(0,1)}x_1^{|\gamma|}P_{\alpha}(x)d \mathcal{S}^{Q-1}_{\infty}= \frac{\frac{\gamma}{2}!}{\frac{|\gamma|}{2}!} \frac{|\gamma|!}{\gamma!} \int_{\partial \B_c(0,1)} \pi_{H_1}(x)^{\gamma}P_{\alpha}(x) d \mathcal{S}_{\infty}^{Q-1}(x)
\end{equation}
holds.
\end{prop}
\begin{proof}
Let us recall Lemma \ref{lemma-general} so that for every even $2n$-multi-index $\gamma=(\gamma_1, \ldots, \gamma_{2n})\in (\N \cup \{ 0 \})^{2n}$ and every $\alpha<Q$ we have
\beq
\label{lemma-alfa}
 \int_{\partial \B_c(0,1)} \pi_{H_1}(x)^{\gamma}\|x\|_{\alpha}^{\alpha-Q} d \mathcal{S}_{\infty}^{Q-1}(x)= 2\int_{\H^n} \pi_{H_1}(x)^{\gamma} h(1,x) \|x\|_c^{-\alpha-|\gamma|} d x
 \eeq
 and
 \beq\label{lemma-alfa2}
 \int_{\partial \B_c(0,1)}x_1^{|\gamma|}\|x\|_{\alpha}^{\alpha-Q} d \mathcal{S}^{Q-1}_{\infty}=2 \int_{\H^n}x_1^{|\gamma|}h(1,x) \|x\|_c^{\alpha-|\gamma|} dx.
\eeq
If for every even $2n$-multi-index $\gamma$ and every $\alpha<Q$ the relation
\begin{equation}\label{tesidadim}
\int_{\H^n}x_1^{|\gamma|}h(1,x) \|x\|_c^{\alpha-|\gamma|} dx = \frac{\frac{\gamma}{2}!}{\frac{|\gamma|}{2}!} \frac{|\gamma|!}{\gamma!} \int_{\H^n} \pi_{H_1}(x)^{\gamma}h(1,x) \|x\|_c^{-\alpha-|\gamma|} dx
\end{equation}
holds, then the thesis follows recalling \eqref{lemma-alfa} and \eqref{lemma-alfa2}.
In order to show \eqref{tesidadim}, we recall a consequence of Proposition \ref{prop3.6} applied to the Euclidean space $\R^{2n}$. In fact, denoting by $\B_E(0,1)$ the Euclidean closed unit ball on $\R^{2n}$, and the Euclidean coarea formula, from \eqref{serve10} of Proposition \ref{prop3.6} 
\begin{equation}\label{conseg-lemma}
\begin{split}
&\int_0^1 \int_{\partial \B_E(0,r)} x_1^{|\gamma|} d \mathcal{H}_E^{2n-1}(x)dr= \int_{\B_E(0,1)} x_1^{|\gamma|} d\mathcal{L}^{2n}(x)\\
&= \frac{\frac{\gamma}{2}!}{\frac{|\gamma|}{2}!} \frac{|\gamma|!}{\gamma!}\int_{\B_E(0,1)} x^{\gamma} d\mathcal{L}^{2n}(x)= \frac{\frac{\gamma}{2}!}{\frac{|\gamma|}{2}!} \frac{|\gamma|!}{\gamma!} \int_0^1 \int_{\partial \B_E(0,r)} x^{\gamma} d \mathcal{H}_E^{2n-1}(x)dr
\end{split}
\end{equation}
hold, where $\mathcal{L}^{2n}$ denotes the $2n$-dimensional Lebesgue measure.
Moreover, exploiting the equality \eqref{conseg-lemma} and the homogeneity of the surface measure $\mathcal{H}^{2n-1}_E$ with respect to the Euclidean isotropic dilations, we get
\begin{align*}
\int_0^1 r^{|\gamma|} dr \int_{\partial \B_E(0,1)} x_1^{|\gamma|} d \mathcal{H}_E^{2n-1}(x) = \frac{\frac{\gamma}{2}!}{\frac{|\gamma|}{2}!} \frac{|\gamma|!}{\gamma!} \int_0^1 r^{|\gamma|} dr \int_{\partial \B_E(0,1)} x^{\gamma} d \mathcal{H}_E^{2n-1}(x).
\end{align*}
As a consequence, we have that
\begin{align}\label{serve11}
 \int_{\partial \B_E(0,1)} x_1^{|\gamma|} d \mathcal{H}_E^{2n-1}(x) = \frac{\frac{\gamma}{2}!}{\frac{|\gamma|}{2}!} \frac{|\gamma|!}{\gamma!} \int_{\partial \B_E(0,1)} x^{\gamma} d \mathcal{H}_E^{2n-1}(x).
\end{align}
Hence, by Fubini's theorem we can rewrite the left-hand side of \eqref{tesidadim} as
\beq\label{integrale-fubini}
\int_{\H^n} x_1^{|\gamma|} h(1,x) \|x\|_c^{-\alpha-|\gamma|} dx = \int_{-\infty}^{\infty} \left( \int_{H_1} x_1^{|\gamma|} h(1,x) \|x\|_c^{-\alpha-|\gamma|} dx_1 \dots dx_{2n} \right) dx_{2n+1}.
\eeq
Performing a change of variables on $H_1 \simeq \R^{2n}$, the same introduced in \cite[Section 4]{B09},
\begin{align*}
\Omega_{2n}: (0,\infty)  \times [0, \pi)^{2n-2} \times [0,2 \pi) &\to H_1\\
(\rho, \theta_1 , \ldots, \theta_{2n-1} )  &\to \Omega_{2n}(\rho, \theta_1 , \ldots, \theta_{2n-1} ),
\end{align*}
with 
\begin{align*}
\Omega_{2n}(\rho, \theta_1 , \ldots, \theta_{2n-1} ) =   \begin{bmatrix} \rho  \cos \theta_1,\\
 \rho  \sin \theta_1 \cos \theta_2,\\
  \rho  \sin \theta_1 \sin \theta_2 \cos \theta_3,\\
  \ldots \\
    \rho  \sin \theta_1 \sin \theta_2  \ldots \sin \theta_{2n-2} \cos \theta_{2n-1},\\
        \rho  \sin \theta_1 \sin \theta_2  \ldots \sin \theta_{2n-2} \sin \theta_{2n-1}\\
        \end{bmatrix}^T,
\end{align*}
where the Jacobian of $\Omega_{2n}$ is
$$ J\Omega_{2n}(\rho, \theta_1 , \ldots, \theta_{2n-1} ) =\rho^{2n-1}( \sin \theta_1)^{2n-2}( \sin \theta_2)^{2n-3} \ldots ( \sin \theta_{2n-2}),$$
from \eqref{integrale-fubini} it follows that
\begin{align}
&\int_{-\infty}^{\infty} \left( \int_{H_1} x_1^{|\gamma|} h(1,x) \|x\|_c^{-\alpha-|\gamma|} dx_1 \dots dx_{2n} \right)  dx_{2n+1} \notag \\
=& \int_{-\infty}^{\infty} \int_0^{\infty} \int_0^{\pi} \ldots \int_0^{\pi} \int_0^{2\pi} (\rho \cos(\theta_1))^{|\gamma|} h(1,(\Omega_{2n}(\rho, \theta),x_{2n+1})) \\
&\| (\Omega_{2n}(\rho, \theta),x_{2n+1})\|_c^{-\alpha-|\gamma|}\rho^{2n-1}( \sin \theta_1)^{2n-2}( \sin \theta_2)^{2n-3} \ldots ( \sin \theta_{2n-2})  d\theta_1 \dots d\theta_{2n-1} d \rho  dx_{2n+1} \notag\\
=& \int_{-\infty}^{\infty} \int_0^{\infty} \rho^{|\gamma|+2n-1} \int_0^{\pi} \ldots \int_0^{\pi} \int_0^{2\pi} (\cos(\theta_1))^{|\gamma|} h(1,(\Omega_{2n}(\rho, \theta),x_{2n+1})) \notag  \\
&\| (\Omega_{2n}(\rho, \theta),x_{2n+1})\|_c^{-\alpha-|\gamma|}( \sin \theta_1)^{2n-2}( \sin \theta_2)^{2n-3} \ldots ( \sin \theta_{2n-2})  d\theta_1 \dots d\theta_{2n-1} d \rho  dx_{2n+1} \notag.
\end{align}
We remark now that for every fixed $x_{2n+1} \in \R$ and $\rho \in (0, \infty)$ the function $$ [0,\pi)^{2n-2} \times [0,2\pi) \ni \theta \to h(1,(\Omega_{2n}(\rho, \theta),x_{2n+1})) 
\| (\Omega_{2n}(\rho, \theta),x_{2n+1})\|_c^{-\alpha-|\gamma|}$$ is constant. This is a consequence of the horizontal rotations invariance of the norm $\| \cdot \|_{\alpha}$ and of the heat kernel $h(1, \cdot)$, since for every fixed $x_{2n+1} \in \R$ and $\rho \in (0, \infty)$ the map
$$  [0,\pi)^{2n-2} \times [0,2\pi) \ni \theta  \to (\Omega_{2n}(\rho, \theta), x_{2n+1})$$
is a horizontal rotation.
As a consequence, we obtain
\begin{align*}
&\int_{-\infty}^{\infty} \int_0^{\infty} \rho^{|\gamma|+2n-1} \| (\Omega_{2n}(\rho, \theta), x_{2n+1})\|_c^{-\alpha-|\gamma|} h(1,(\Omega_{2n}(\rho, \theta),x_{2n+1}))  d \rho  dx_{2n+1} \cdot \\
&\int_0^{\pi} \ldots \int_0^{\pi} \int_0^{2\pi}  (\cos(\theta_1))^{|\gamma|}  ( \sin \theta_1)^{2n-2}( \sin \theta_2)^{2n-3} \ldots ( \sin \theta_{2n-2})  d\theta_1 \dots d\theta_{2n-1}\\
=&\int_{-\infty}^{\infty} \int_0^{\infty} \rho^{|\gamma|+2n-1} \| (\Omega_{2n}(\rho, \theta),x_{2n+1})\|_c^{-\alpha-|\gamma|} h(1,(\Omega_{2n}(\rho, \theta), x_{2n+1}))  d \rho  dx_{2n+1} \int_{\partial \B_E(0,1)} x_1^{|\gamma|} d \mathcal{H}^{2n-1}_E(x).
\end{align*}
Repeating the same steps on the right-hand side $ \int_{\H^n} \pi_{H_1}(x)^{\gamma} h(1,x) \|x\|_c^{-\alpha-|\gamma|} dx$ of \eqref{tesidadim} and exploiting \eqref{serve11}, we conclude the proof.
\end{proof}

Now, we wish to exploit Proposition \ref{prop-rel-integrali} in order to deduce a simplified representation of the map $\psi(x,\alpha)$ in \eqref{expression}.
We recall the analytic continuation of $\psi(x,\alpha)$ on the strip $-8<\alpha<Q$ given in \eqref{expression} and introduce the some notation. For every $i,j \in \{1, \ldots, 2n\}, \ i \neq j $, we denote by $\mathcal{P}_{i,j}$ the family of permutations with repetitions of the set of elements $\{i,i,i,i,j,j\}$.
Moreover, for every $\sigma=(i_1, \dots, i_6) \in \mathcal{P}_{i,j}$ we set $Z_{\sigma}=Z_{i_1} \dots Z_{i_6}$. 

From \eqref{exp3}, the explicit expression of $Z$-Taylor polynomials in Theorem \ref{explicit}, the relations characterizing the Lie algebra of $\H^n$ and the result of Proposition \ref{prop-rel-integrali}, we can write $\psi(x,\alpha)$ in the following form for every $-8<\alpha<Q$
\begin{align*}
\psi(x,\alpha)= & \int_{\H^n \setminus \B_c(0,1)} \phi(xy)P_{\alpha}(y) dy + \int_{\B_c(0,1)} \Big(\phi(xy)-\mathcal{T}_7(\phi,x)(xy)\Big)P_{\alpha}(y) dy \\
&+ \phi(x) \frac{1}{\Gamma(\frac{\alpha}{2})} \frac{1}{\alpha} \sigma(\alpha)- \frac{1}{2}\frac{1}{\alpha+2}\frac{1}{\Gamma(\frac{\alpha}{2})} \mathcal{L}\phi(x)d(\alpha)\\
& + \frac{1}{\alpha+4} \frac{1}{\Gamma(\frac{\alpha}{2})} \frac{1}{4!} \mathcal{L}^2 \phi(x) \int_{\partial \B_c(0,1)}y_1^4\|y\|_{\alpha}^{\alpha-Q} d\mathcal{S}^{Q-1}_{\infty}(y)\\
& + \frac{1}{\alpha+4} \frac{1}{\Gamma(\frac{\alpha}{2})} T^2\phi(x) \bigg(\frac{1}{2}  \int_{\partial \B_c(0,1)}y_{2n+1}^2  \|y\|_{\alpha}^{\alpha-Q} dy  -\frac{n}{4!} \int_{\partial \B_c(0,1)}y_1^4\|y\|_{\alpha}^{\alpha-Q} d\mathcal{S}^{Q-1}_{\infty}(y) \bigg)\\
&+\frac{1}{4!}\frac{1}{\alpha+6} \frac{4!}{2!2!}\sum_{i=1}^{2n} T^2 Z_i^2 \phi(x) \int_{\partial \B_c(0,1)}y_i^2 y_{2n+1}^2P_{\alpha}(y) d\mathcal{S}^{Q-1}_{\infty}(y)\\
&+ \frac{1}{6!}\frac{1}{\alpha+6} \sum_{i=1}^{2n} Z_i^6 \phi(x) \int_{\partial \B_c(0,1)}y_i^6P_{\alpha}(y) d\mathcal{S}^{Q-1}_{\infty}(y)\\
&+ \frac{1}{6!}\frac{1}{\alpha+6} \frac{6!}{2! 2! 2!} \frac{1}{3!} \sum_{\substack{ i,j,k=1, \\i \neq j \neq k \neq i,\\ [Z_i,Z_k]=[Z_i,Z_k]=[Z_j,Z_k]=0 }}^{2n} (Z_i^2Z_j^2Z_k^2\phi(x)+Z_i^2Z_k^2Z_j^2\phi(x)+Z_j^2Z_i^2Z_k^2\phi(x)\\
&+Z_j^2Z_k^2Z_i^2\phi(x)+Z_k^2Z_j^2Z_i^2\phi(x) +Z_k^2Z_i^2Z_j^2\phi(x)) \int_{ \partial \B_c(0,1)}y_i^2 y_j^2y_k^2P_{\alpha}(y) d\mathcal{S}^{Q-1}_{\infty}(y)\\
&+ \frac{1}{6!}\frac{1}{\alpha+6} \frac{6!}{4! 2!}   \sum_{\substack{ i,j=1, \\i < j ,\\ [Z_i,Z_j]=T  }}^{2n} \sum_{\substack{k=1 \\,  i,j\neq k ,\\ [Z_i,Z_k]=[Z_j,Z_k]=0}}^{2n}  (Z_i^2Z_j^2+Z_iZ_jZ_iZ_j+Z_iZ_jZ_jZ_i\\
&+ Z_jZ_iZ_iZ_j+Z_jZ_iZ_jZ_i+Z_j^2Z_i^2)Z_k^2\phi(x) \int_{ \partial \B_c(0,1)}y_i^2 y_j^2y_k^2P_{\alpha}(y) d\mathcal{S}^{Q-1}_{\infty}(y)\\
&+ \frac{1}{6!} \frac{1}{\alpha+6} \sum_{\substack{i,j=1, i \neq j }}^{2n} \sum_{\sigma \in \mathcal{P}_{i,j}}Z_{\sigma}\phi(x) \phi(x) \int_{\partial \B_c(0,1)}y_i^4 y_j^2P_{\alpha}(y) d\mathcal{S}^{Q-1}_{\infty}(y)\\
=& \int_{\H^n \setminus \B_c(0,1)} \phi(xy)P_{\alpha}(y) dy + \int_{\B_c(0,1)} \Big(\phi(xy)-\mathcal{T}_7(\phi,x)(xy)\Big)P_{\alpha}(y) dy \\
&+ \phi(x) \frac{1}{\Gamma(\frac{\alpha}{2})} \frac{1}{\alpha} \sigma(\alpha)- \frac{1}{2}\frac{1}{\alpha+2}\frac{1}{\Gamma(\frac{\alpha}{2})} \mathcal{L}\phi(x)d(\alpha)\\
& + \frac{1}{\alpha+4} \frac{1}{\Gamma(\frac{\alpha}{2})} \frac{1}{4!} \mathcal{L}^2 \phi(x) \int_{\partial \B_c(0,1)}y_1^4\|y\|_{\alpha}^{\alpha-Q} d\mathcal{S}^{Q-1}_{\infty}(y)\\
& + \frac{1}{\alpha+4} \frac{1}{\Gamma(\frac{\alpha}{2})} T^2\phi(x) \bigg(\frac{1}{2}  \int_{\partial \B_c(0,1)}y_{2n+1}^2  \|y\|_{\alpha}^{\alpha-Q} dy  -\frac{n}{4!} \int_{\partial \B_c(0,1)}y_1^4\|y\|_{\alpha}^{\alpha-Q} d\mathcal{S}^{Q-1}_{\infty}(y) \bigg)\\
&-\frac{1}{4}\frac{1}{\alpha+6} T^2 \mathcal{L}\phi(x) \int_{\partial \B_c(0,1)}y_1^2 y_{2n+1}^2P_{\alpha}(y) d\mathcal{S}^{Q-1}_{\infty}(y)\\
& + \frac{1}{6!} \frac{1}{\alpha+6} \sum_{i=1}^{2n} Z_i^6 \phi(x) \int_{\partial \B_c(0,1)}y_1^6P_{\alpha}(y) d\mathcal{S}^{Q-1}_{\infty}(y)\\
&+ \frac{1}{6!}\frac{1}{\alpha+6} \frac{6!}{2! 2! 2!} \frac{1}{3!} \sum_{\substack{ i,j,k=1, \\i \neq j \neq k \neq i,\\ [Z_i,Z_k]=[Z_i,Z_k]=[Z_j,Z_k]=0 }}^{2n} (Z_i^2Z_j^2Z_k^2\phi(x)+Z_i^2Z_k^2Z_j^2\phi(x)+Z_j^2Z_i^2Z_k^2\phi(x)\\
&+Z_j^2Z_k^2Z_i^2\phi(x)+Z_k^2Z_j^2Z_i^2\phi(x) +Z_k^2Z_i^2Z_j^2\phi(x)) \frac{2!2!2!}{6!}\frac{3!}{1!1!1!} \int_{ \partial \B_c(0,1)}y_1^6P_{\alpha}(y) d\mathcal{S}^{Q-1}_{\infty}(y)\\
&+ \frac{1}{6!}\frac{1}{\alpha+6} \frac{6!}{4! 2!}   \sum_{\substack{ i,j=1, \\i < j ,\\ [Z_i,Z_j]=T  }}^{2n} \sum_{\substack{k=1 \\,  i,j\neq k ,\\ [Z_i,Z_k]=[Z_j,Z_k]=0}}^{2n}  (3Z_i^2Z_j^2+3Z_j^2Z_i^2-3T^2)Z_k^2\phi(x)\\
& \frac{2!2!2!}{6!}\frac{3!}{1!1!1!}\int_{ \partial \B_c(0,1)}y_1^6P_{\alpha}(y) d\mathcal{S}^{Q-1}_{\infty}(y)\\
&+ \frac{1}{6!} \frac{1}{\alpha+6}  \sum_{\substack{i,j=1, i \neq j }}^{2n} \sum_{\sigma \in \mathcal{P}_{i,j}}Z_{\sigma}\phi(x)  \int_{\partial \B_c(0,1)}y_i^4 y_j^2P_{\alpha}(y) d\mathcal{S}^{Q-1}_{\infty}(y)\\
=& \int_{\H^n \setminus \B_c(0,1)} \phi(xy)P_{\alpha}(y) dy + \int_{\B_c(0,1)} \Big(\phi(xy)-\mathcal{T}_7(\phi,x)(xy)\Big)P_{\alpha}(y) dy \\
&+ \phi(x) \frac{1}{\Gamma(\frac{\alpha}{2})} \frac{1}{\alpha} \sigma(\alpha)- \frac{1}{2}\frac{1}{\alpha+2}\frac{1}{\Gamma(\frac{\alpha}{2})} \mathcal{L}\phi(x)d(\alpha)\\
& + \frac{1}{\alpha+4} \frac{1}{\Gamma(\frac{\alpha}{2})} \frac{1}{4!} \mathcal{L}^2 \phi(x) \int_{\partial \B_c(0,1)}y_1^4\|y\|_{\alpha}^{\alpha-Q} d\mathcal{S}^{Q-1}_{\infty}(y)\\
& + \frac{1}{\alpha+4} \frac{1}{\Gamma(\frac{\alpha}{2})} T^2\phi(x) \bigg(\frac{1}{2}  \int_{\partial \B_c(0,1)}y_{2n+1}^2  \|y\|_{\alpha}^{\alpha-Q} dy  -\frac{n}{4!} \int_{\partial \B_c(0,1)}y_1^4\|y\|_{\alpha}^{\alpha-Q} d\mathcal{S}^{Q-1}_{\infty}(y) \bigg)\\
&-\frac{1}{4}\frac{1}{\alpha+6} T^2 \mathcal{L}\phi(x) \int_{\partial \B_c(0,1)}y_1^2 y_{2n+1}^2P_{\alpha}(y) d\mathcal{S}^{Q-1}_{\infty}(y)\\
& + \frac{1}{6!} \frac{1}{\alpha+6} \sum_{i=1}^{2n} Z_i^6 \phi(x) \int_{\partial \B_c(0,1)}y_1^6P_{\alpha}(y) d\mathcal{S}^{Q-1}_{\infty}(y)\\
&+ \frac{1}{6!}\frac{1}{\alpha+6}  \sum_{\substack{ i,j,k=1, \\i \neq j \neq k \neq i,\\ [Z_i,Z_k]=[Z_i,Z_k]=[Z_j,Z_k]=0 }}^{2n} (Z_i^2Z_j^2Z_k^2\phi(x)+Z_i^2Z_k^2Z_j^2\phi(x)+Z_j^2Z_i^2Z_k^2\phi(x)\\
&+Z_j^2Z_k^2Z_i^2\phi(x)+Z_k^2Z_j^2Z_i^2\phi(x) +Z_k^2Z_i^2Z_j^2\phi(x))  \int_{ \partial \B_c(0,1)}y_1^6P_{\alpha}(y) d\mathcal{S}^{Q-1}_{\infty}(y)\\
&+ \frac{1}{6!}\frac{1}{\alpha+6} 3 \sum_{\substack{ i,j=1, \\i < j ,\\ [Z_i,Z_j]=T  }}^{2n} \sum_{\substack{k=1 \\,  i,j\neq k ,\\ [Z_i,Z_k]=[Z_j,Z_k]=0}}^{2n}  (Z_i^2Z_j^2+Z_j^2Z_i^2-T^2)Z_k^2\phi(x)\\
& \int_{ \partial \B_c(0,1)}y_1^6P_{\alpha}(y) d\mathcal{S}^{Q-1}_{\infty}(y)\\
&+ \frac{1}{6!} \frac{1}{\alpha+6}  \sum_{\substack{i,j=1, i \neq j }}^{2n} \sum_{\sigma \in \mathcal{P}_{i,j}}Z_{\sigma}\phi(x)  \int_{\partial \B_c(0,1)}y_i^4 y_j^2P_{\alpha}(y) d\mathcal{S}^{Q-1}_{\infty}(y)\\
=& \int_{\H^n \setminus \B_c(0,1)} \phi(xy)P_{\alpha}(y) dy + \int_{\B_c(0,1)} \Big(\phi(xy)-\mathcal{T}_7(\phi,x)(xy)\Big)P_{\alpha}(y) dy \\
&+ \phi(x) \frac{1}{\Gamma(\frac{\alpha}{2})} \frac{1}{\alpha} \sigma(\alpha)- \frac{1}{2}\frac{1}{\alpha+2}\frac{1}{\Gamma(\frac{\alpha}{2})} \mathcal{L}\phi(x)d(\alpha)\\
& + \frac{1}{\alpha+4} \frac{1}{\Gamma(\frac{\alpha}{2})} \frac{1}{4!} \mathcal{L}^2 \phi(x) \int_{\partial \B_c(0,1)}y_1^4\|y\|_{\alpha}^{\alpha-Q} d\mathcal{S}^{Q-1}_{\infty}(y)\\
& + \frac{1}{\alpha+4} \frac{1}{\Gamma(\frac{\alpha}{2})} T^2\phi(x) \bigg(\frac{1}{2}  \int_{\partial \B_c(0,1)}y_{2n+1}^2  \|y\|_{\alpha}^{\alpha-Q} dy  -\frac{n}{4!} \int_{\partial \B_c(0,1)}y_1^4\|y\|_{\alpha}^{\alpha-Q} d\mathcal{S}^{Q-1}_{\infty}(y) \bigg)\\
&-\frac{1}{4}\frac{1}{\alpha+6} T^2 \mathcal{L}\phi(x) \int_{\partial \B_c(0,1)}y_1^2 y_{2n+1}^2P_{\alpha}(y) d\mathcal{S}^{Q-1}_{\infty}(y)\\
& + \frac{1}{6!} \frac{1}{\alpha+6} \sum_{i=1}^{2n} Z_i^6 \phi(x) \int_{\partial \B_c(0,1)}y_1^6P_{\alpha}(y) d\mathcal{S}^{Q-1}_{\infty}(y)\\
&+ \frac{1}{6!}\frac{1}{\alpha+6}  \sum_{\substack{ i,j,k=1, \\i \neq j \neq k \neq i,\\ [Z_i,Z_k]=[Z_i,Z_k]=[Z_j,Z_k]=0 }}^{2n} (Z_i^2Z_j^2Z_k^2\phi(x)+Z_i^2Z_k^2Z_j^2\phi(x)+Z_j^2Z_i^2Z_k^2\phi(x)\\
&+Z_j^2Z_k^2Z_i^2\phi(x)+Z_k^2Z_j^2Z_i^2\phi(x) +Z_k^2Z_i^2Z_j^2\phi(x))  \int_{ \partial \B_c(0,1)}y_1^6P_{\alpha}(y) d\mathcal{S}^{Q-1}_{\infty}(y)\\
&+ \frac{1}{6!}\frac{1}{\alpha+6} 3 \Big(\sum_{\substack{ i,j=1, \\i < j ,\\ [Z_i,Z_j]=T  }}^{2n} \sum_{\substack{k=1 \\,  i,j\neq k ,\\ [Z_i,Z_k]=[Z_j,Z_k]=0}}^{2n}  (Z_i^2Z_j^2+Z_j^2Z_i^2)Z_k^2 \phi(x) \\
&- \sum_{\substack{ i,j=1, \\i < j ,\\ [Z_i,Z_j]=T  }}^{2n} \sum_{\substack{k=1 \\,  i,j\neq k ,\\ [Z_i,Z_k]=[Z_j,Z_k]=0}}^{2n} T^2Z_k^2\phi(x) \Big) \int_{ \partial \B_c(0,1)}y_1^6P_{\alpha}(y) d\mathcal{S}^{Q-1}_{\infty}(y)\\
&+ \frac{1}{6!} \frac{1}{\alpha+6}  \sum_{\substack{i,j=1, i \neq j }}^{2n} \sum_{\sigma \in \mathcal{P}_{i,j}}Z_{\sigma}\phi(x) \int_{\partial \B_c(0,1)}y_i^4 y_j^2P_{\alpha}(y) d\mathcal{S}^{Q-1}_{\infty}(y).
\end{align*}
Hence, we have obtained the representation
\begin{align} \label{espressione4striscia}
\psi(x,\alpha)=& \int_{\H^n \setminus \B_c(0,1)} \phi(xy)P_{\alpha}(y) dy + \int_{\B_c(0,1)} \Big(\phi(xy)-\mathcal{T}_7(\phi,x)(xy)\Big)P_{\alpha}(y) dy \notag \\ 
&+ \phi(x) \frac{1}{\Gamma(\frac{\alpha}{2})} \frac{1}{\alpha} \sigma(\alpha)- \frac{1}{2}\frac{1}{\alpha+2}\frac{1}{\Gamma(\frac{\alpha}{2})} \mathcal{L}\phi(x)d(\alpha) \notag\\
& + \frac{1}{\alpha+4} \frac{1}{\Gamma(\frac{\alpha}{2})} \frac{1}{4!} \mathcal{L}^2 \phi(x) \int_{\partial \B_c(0,1)}y_1^4\|y\|_{\alpha}^{\alpha-Q} d\mathcal{S}^{Q-1}_{\infty}(y) \notag \\
& + \frac{1}{\alpha+4} \frac{1}{\Gamma(\frac{\alpha}{2})} T^2\phi(x) \bigg(\frac{1}{2}  \int_{\partial \B_c(0,1)}y_{2n+1}^2  \|y\|_{\alpha}^{\alpha-Q} dy  -\frac{n}{4!} \int_{\partial \B_c(0,1)}y_1^4\|y\|_{\alpha}^{\alpha-Q} d\mathcal{S}^{Q-1}_{\infty}(y) \bigg) \notag \\
&-\frac{1}{4}\frac{1}{\alpha+6} T^2 \mathcal{L}\phi(x) \int_{\partial \B_c(0,1)}y_1^2 y_{2n+1}^2P_{\alpha}(y) d\mathcal{S}^{Q-1}_{\infty}(y) \notag \\
& + \frac{1}{6!} \frac{1}{\alpha+6} \sum_{i=1}^{2n} Z_i^6 \phi(x) \int_{\partial \B_c(0,1)}y_1^6P_{\alpha}(y) d\mathcal{S}^{Q-1}_{\infty}(y)\\
&+ \frac{1}{6!}\frac{1}{\alpha+6}  \sum_{\substack{ i,j,k=1, \\i \neq j \neq k \neq i,\\ [Z_i,Z_k]=[Z_i,Z_k]=[Z_j,Z_k]=0 }}^{2n} (Z_i^2Z_j^2Z_k^2\phi(x)+Z_i^2Z_k^2Z_j^2\phi(x)+Z_j^2Z_i^2Z_k^2\phi(x) \notag \\
&+Z_j^2Z_k^2Z_i^2\phi(x)+Z_k^2Z_j^2Z_i^2\phi(x) +Z_k^2Z_i^2Z_j^2\phi(x))  \int_{ \partial \B_c(0,1)}y_1^6P_{\alpha}(y) d\mathcal{S}^{Q-1}_{\infty}(y) \notag \\
&+ \frac{1}{6!}\frac{1}{\alpha+6}  \Big(\frac{3}{3}\sum_{\substack{ i,j=1, \\i < j ,\\ [Z_i,Z_j]=T  }}^{2n} \sum_{\substack{k=1 \\,  i,j\neq k ,\\ [Z_i,Z_k]=[Z_j,Z_k]=0}}^{2n}  (Z_i^2Z_j^2Z_k^2\phi(x)+Z_i^2Z_k^2Z_j^2\phi(x)+Z_j^2Z_i^2Z_k^2\phi(x) \notag \\
&+Z_j^2Z_k^2Z_i^2\phi(x)+Z_k^2Z_j^2Z_i^2\phi(x) +Z_k^2Z_i^2Z_j^2\phi(x)) \notag \\
&+ 3(n-1)T^2\mathcal{L}\phi(x)) \int_{ \partial \B_c(0,1)}y_1^6P_{\alpha}(y) d\mathcal{S}^{Q-1}_{\infty}(y) \notag \\
&+ \frac{1}{6!} \frac{1}{\alpha+6}  \sum_{\substack{i,j=1, i \neq j }}^{2n} \sum_{\sigma \in \mathcal{P}_{i,j}}Z_{\sigma}\phi(x) \int_{\partial \B_c(0,1)}y_i^4 y_j^2P_{\alpha}(y) d\mathcal{S}^{Q-1}_{\infty}(y). \notag
\end{align}
For every $i,j \in \{1, \dots, 2n\}, i \neq j $ we have two possibilities:
\begin{itemize}
\item[(a)] $[Z_i,Z_j]=0$ and then clearly we can organize as follows
\begin{equation}\label{organizzati}
\sum_{\sigma \in \mathcal{P}_{i,j}}Z_{\sigma}\phi(x)=  5 Z_i^4Z_j^2\phi(x)+5Z_i^2Z_j^2Z_i^2\phi(x)+5Z_j^2Z_i^4 \phi(x),
\end{equation}
\item[(b)] $[Z_i,Z_j] \neq 0$, then we claim that
\begin{equation}\label{claimdaprovare}
\sum_{\sigma \in \mathcal{P}_{i,j}}Z_{\sigma}\phi(x)=  5 Z_i^4Z_j^2\phi(x)+5Z_i^2Z_j^2Z_i^2\phi(x)+5Z_j^2Z_i^4 \phi(x)-25T^2Z_i^2 \phi(x).
\end{equation}
\end{itemize}
Let us explicitly notice that $$\begin{Large}5= \frac{6!}{4!2!}\frac{2!1!}{3!}.\end{Large}$$
If we assume that \eqref{claimdaprovare} is true, then by Proposition \ref{prop-rel-integrali} we have the equality
\begin{align*}
&\frac{1}{6!} \frac{1}{\alpha+6}  \sum_{\substack{i,j=1, i \neq j }}^{2n} \sum_{\sigma \in \mathcal{P}_{i,j}}Z_{\sigma}\phi(x) \int_{\partial \B_c(0,1)}y_i^4 y_j^2P_{\alpha}(y) d\mathcal{S}^{Q-1}_{\infty}(y)\\
=& \  \frac{1}{6!} \frac{1}{\alpha+6} \sum_{\substack{i=1,\\ i \neq j }}^{2n} (Z_i^4Z_j^2\phi(x)+Z_i^2Z_j^2Z_i^2\phi(x)+Z_j^2Z_i^4 \phi(x)) 5  \frac{4!2!}{6!}\frac{3!}{2!1!}\int_{\partial \B_c(0,1)}y_i^6P_{\alpha}(y)  d\mathcal{S}^{Q-1}_{\infty}(y)\\
&-5 \frac{1}{6!} \frac{1}{\alpha+6}\sum_{i=1}^{2n} T^2Z_i^2 \phi(x)  5 \frac{4!2!}{6!}\frac{3!}{2!1!}\int_{\partial \B_c(0,1)}y_i^6P_{\alpha}(y)  d\mathcal{S}^{Q-1}_{\infty}(y),
\end{align*}
hence we can state that 
\begin{align*}
&-\frac{1}{4}\frac{1}{\alpha+6} T^2 \mathcal{L}\phi(x) \int_{\partial \B_c(0,1)}y_1^2 y_{2n+1}^2P_{\alpha}(y) d\mathcal{S}^{Q-1}_{\infty}(y)\\
& + \frac{1}{6!} \frac{1}{\alpha+6} \sum_{i=1}^{2n} Z_i^6 \phi(x) \int_{\partial \B_c(0,1)}y_1^6P_{\alpha}(y) d\mathcal{S}^{Q-1}_{\infty}(y)\\
&+ \frac{1}{6!}\frac{1}{\alpha+6}  \sum_{\substack{ i,j,k=1, \\i \neq j \neq k \neq i,\\ [Z_i,Z_k]=[Z_i,Z_k]=[Z_j,Z_k]=0 }}^{2n} (Z_i^2Z_j^2Z_k^2\phi(x)+Z_i^2Z_k^2Z_j^2\phi(x)+Z_j^2Z_i^2Z_k^2\phi(x)\\
&+Z_j^2Z_k^2Z_i^2\phi(x)+Z_k^2Z_j^2Z_i^2\phi(x) +Z_k^2Z_i^2Z_j^2\phi(x))  \int_{ \partial \B_c(0,1)}y_1^6P_{\alpha}(y) d\mathcal{S}^{Q-1}_{\infty}(y)\\
&+ \frac{1}{6!}\frac{1}{\alpha+6}  \sum_{\substack{ i,j=1, \\i < j ,\\ [Z_i,Z_j]=T  }}^{2n} \sum_{\substack{k=1 \\,  i,j\neq k ,\\ [Z_i,Z_k]=[Z_j,Z_k]=0}}^{2n}  (Z_i^2Z_j^2Z_k^2\phi(x)+Z_i^2Z_k^2Z_j^2\phi(x)+Z_j^2Z_i^2Z_k^2\phi(x)\\
&+Z_j^2Z_k^2Z_i^2\phi(x)+Z_k^2Z_j^2Z_i^2\phi(x) +Z_k^2Z_i^2Z_j^2\phi(x))) \int_{ \partial \B_c(0,1)}y_1^6P_{\alpha}(y) d\mathcal{S}^{Q-1}_{\infty}(y)\\
 \\
&+ 3 \frac{1}{6!} \frac{1}{\alpha+6} (n-1)T^2\mathcal{L}\phi(x)) \int_{ \partial \B_c(0,1)}y_1^6P_{\alpha}(y) d\mathcal{S}^{Q-1}_{\infty}(y)\\
&+ \frac{1}{6!} \frac{1}{\alpha+6}  \sum_{\substack{i,j=1, i \neq j }}^{2n} \sum_{\sigma \in \mathcal{P}_{i,j}}Z_{\sigma}\phi(x)  \int_{\partial \B_c(0,1)}y_i^4 y_j^2P_{\alpha}(y) d\mathcal{S}^{Q-1}_{\infty}(y)\\
=& - \frac{1}{\alpha+6} T^2 \mathcal{L}\phi(x) \left( \frac{1}{4} \int_{\partial \B_c(0,1)}y_1^2 y_{2n+1}^2P_{\alpha}(y)  d\mathcal{S}^{Q-1}_{\infty}(y) - \Big(\frac{5}{6!} +\frac{3}{6!}(n-1) \Big)\int_{\partial \B_c(0,1)}y_1^6P_{\alpha}(y)  d\mathcal{S}^{Q-1}_{\infty}(y) \right)\\
& -  \frac{1}{\alpha+6} \mathcal{L}^3\phi(x) \frac{1}{6!} \int_{\partial \B_c(0,1)}y_1^6P_{\alpha}(y)  d\mathcal{S}^{Q-1}_{\infty}(y).
\end{align*}
Therefore, for $\alpha \in (-8,Q)$ we can rewrite \eqref{espressione4striscia} as 
\begin{align}\label{ultimarap}
\psi(x,\alpha)&= \int_{\H^n \setminus \B_c(0,1)} \phi(xy)P_{\alpha}(y) dy + \int_{\B_c(0,1)} \Big(\phi(xy)-\mathcal{T}_7(\phi,x)(xy)\Big)P_{\alpha}(y) dy \notag \\
&+ \phi(x) \frac{1}{\Gamma(\frac{\alpha}{2})} \frac{1}{\alpha} \sigma(\alpha)- \frac{1}{2}\frac{1}{\alpha+2}\frac{1}{\Gamma(\frac{\alpha}{2})} \mathcal{L}\phi(x)d(\alpha) \notag \\
& + \frac{1}{\alpha+4} \frac{1}{\Gamma(\frac{\alpha}{2})} \frac{1}{4!} \mathcal{L}^2 \phi(x) \int_{\partial \B_c(0,1)}y_1^4\|y\|_{\alpha}^{\alpha-Q} d\mathcal{S}^{Q-1}_{\infty}(y) \notag \\
& + \frac{1}{\alpha+4} \frac{1}{\Gamma(\frac{\alpha}{2})} T^2\phi(x) \bigg(\frac{1}{2}  \int_{\partial \B_c(0,1)}y_{2n+1}^2  \|y\|_{\alpha}^{\alpha-Q} dy  -\frac{n}{4!} \int_{\partial \B_c(0,1)}y_1^4\|y\|_{\alpha}^{\alpha-Q} d\mathcal{S}^{Q-1}_{\infty}(y) \bigg) \\
&-\frac{1}{\alpha+6} \frac{1}{\Gamma(\frac{\alpha}{2})} T^2 \mathcal{L}\phi(x) \Bigg( \frac{1}{4} \int_{\partial \B_c(0,1)}y_1^2 y_{2n+1}^2\|y\|_{\alpha}^{\alpha-Q}  d\mathcal{S}^{Q-1}_{\infty}(y) \notag \\
& -\Big(\frac{5}{6!} +\frac{3}{6!}(n-1) \Big)\int_{\partial \B_c(0,1)}y_1^6\|y\|_{\alpha}^{\alpha-Q}  d\mathcal{S}^{Q-1}_{\infty}(y)  \Bigg) \notag \\
&-  \frac{1}{\alpha+6}  \frac{1}{\Gamma(\frac{\alpha}{2})}\mathcal{L}^3\phi(x) \frac{1}{6!} \int_{\partial \B_c(0,1)}y_1^6\|y\|_{\alpha}^{\alpha-Q}  d\mathcal{S}^{Q-1}_{\infty}(y) . \notag
\end{align}

Finally, in order to establish the representation in \eqref{ultimarap}, we are left to prove \eqref{claimdaprovare}. Let us consider then $i,j \in \{1, \dots, 2n\}, i \neq j$, with $[Z_i,Z_j] \neq 0$. Hence $[Z_i,Z_j]=T$ or $[Z_i,Z_j]=-T$. We assume that $[Z_i,Z_j]=T$ and we denote by $X=Z_i$ and $Y=Z_j$.
The sum on the left-hand side of \eqref{claimdaprovare} involves the following terms
\beq
\begin{matrix}\label{derivate}
XXXXYY\phi(x) & YXXXXY\phi(x)&  YXXXYX\phi(x) \\
XXXYXY\phi(x) & XXYXYX \phi(x)&  YXXYXX \phi(x) \\
XXYXXY\phi(x) & XYXXYX \phi(x)&  XYYXXX \phi(x) \\
XYXXXY\phi(x) & XYXYXX \phi(x) &  YXYXXX \phi(x)\\
XXXYYX\phi(x) & XXYYXX\phi(x)&   YYXXXX \phi(x) \\
\end{matrix}.
\eeq
By exploiting the relations among the vector fields $X, Y$ and $T$, it can be verified that the sum of the terms of the first column of \eqref{derivate} gives 
\begin{equation}\label{serve13}
5 X^4Y^2 \phi(x)-8X^3YT \phi(x),
\end{equation}
the sum of the terms of the third column of \eqref{derivate} is
\begin{equation}\label{serve14}
5 Y^2X^4 \phi(x) +8 YX^3T \phi(x),
\end{equation}
and the sum of the terms of the second column of \eqref{derivate} equals
\begin{equation}\label{serve15}
5 X^2 Y^2X^2 \phi(x) -X^3YT\phi(x)+X^2YXT\phi(x)=5 X^2 Y^2X^2 \phi(x) -X^2T^2\phi(x).
\end{equation}

We observe also that the following equalities hold
\begin{align}\label{serve16} &X^2T^2=XTXT=X(XY-YX)XT=X^2YXT-XYX^2T \notag \\
& X^2T^2=X^2(XY-YX)T=X^3YT-X^2YXT\\
& X^2T^2=(XY-YX)X^2T=XYX^2T-YX^3T ,\notag
\end{align}
hence
$$3X^2T^2=X^3YT-YX^3T.$$

If we sum up the fifteen terms of \eqref{derivate} and we exploit the established relations \eqref{serve13}, \eqref{serve14}, \eqref{serve15} and \eqref{serve16} we get the following result
\begin{align*}
& \ 5(X^4Y^2 \phi(x)+X^2Y^2X^2\phi(x)+Y^2X^4\phi(x)) -8X^3YT \phi(x)+8 YX^3T \phi(x)-X^2T^2\phi(x)\\
=& \ 5(X^4Y^2 \phi(x)+X^2Y^2X^2\phi(x)+Y^2X^4\phi(x)) -8(X^3YT - YX^3T )\phi(x)-X^2T^2\phi(x)\\
=& \ 5(X^4Y^2 \phi(x)+X^2Y^2X^2\phi(x)+Y^2X^4\phi(x)) -8(3X^2T^2 )\phi(x)-X^2T^2\phi(x)\\
=& \ 5(X^4Y^2 \phi(x)+X^2Y^2X^2\phi(x)+Y^2X^4\phi(x)) -24X^2T^2\phi(x)-X^2T^2\phi(x)\\
=& \ 5(X^4Y^2 \phi(x)+X^2Y^2X^2\phi(x)+Y^2X^4\phi(x)) -25X^2T^2\phi(x).
\end{align*}
One can verify that analogous passages can be carried out if $[Z_i,Z_j]=-T$. Thus, \eqref{claimdaprovare} is proved.

\subsection{A geometric application, III}
Here, we apply the representation obtained in \eqref{ultimarap} in order to prove the following proposition.
\begin{prop}
Let $n \in \N$ and let $h$ be the heat kernel associated with $\mathcal{L}$ on $\H^n$, then for every $i \in \{1, \ldots, 2n \}$ the following equalities hold
\begin{equation}\label{val-int3}
\int_{\H^n} x_i^6 h(1, x) dx=5! \qquad \mathrm{and} \qquad \int_{\H^n} x_i^2 x_{2n+1}^2 h(1, x) dx=\frac{2}{3}(5+3(n-1)).
\end{equation}
\end{prop}
\begin{proof}
As a consequence of the representation of $\psi(x,\alpha)$ in \eqref{ultimarap} valid for $\alpha \in (-8,Q)$, exploiting Lemma \ref{lemma-general} and taking into account the definition \eqref{realpow}, we know that
\begin{align}\label{trilap}
&\mathcal{L}^3\phi(x)=\lim_{\alpha \to -6} \psi(x,\alpha) \notag \\
=& \lim_{\alpha \to -6} \Bigg(-\frac{1}{\alpha+6}\frac{1}{\Gamma(\frac{\alpha}{2})} T^2 \mathcal{L}\phi(x) \Bigg( \frac{1}{4} \int_{\partial \B_c(0,1)}y_1^2 y_{2n+1}^2 \|y\|_{\alpha}^{\alpha-Q} d\mathcal{S}^{Q-1}_{\infty}(y) \\
&-\Big(\frac{5}{6!} +\frac{3}{6!}(n-1) \Big) \int_{\partial \B_c(0,1)}y_1^6 \|y\|_{\alpha}^{\alpha-Q} d\mathcal{S}^{Q-1}_{\infty}(y)  \Bigg) - \frac{1}{\alpha+6} \frac{1}{\Gamma(\frac{\alpha}{2})}\mathcal{L}^3\phi(x) \frac{1}{6!} \int_{\partial \B_c(0,1)}y_1^6 \|y\|_{\alpha}^{\alpha-Q} d\mathcal{S}^{Q-1}_{\infty}(y) \Bigg) \notag \\
=&  \ 6 T^2 \mathcal{L}\phi(x) \left( \frac{1}{4} \int_{\H^n } y_1^2 y_{2n+1}^2 h(1,y)dy -\Big(\frac{5}{6!} +\frac{3}{6!}(n-1) \Big) \int_{\H^n} y_1^6 h(1,y) dy \right) \notag \\
& +6 \mathcal{L}^3\phi(x) \frac{1}{6!} \int_{\H^n}y_1^6 h(1,y) dy.  \notag
\end{align}
If we show that there is not any constant $c \in \R$ such that the equality
\begin{equation}\label{assumption3}
\mathcal{L}^3\phi(x)=cT^2 \mathcal{L}\phi(x)
\end{equation}
holds for every $\phi \in \mathcal{S}(\H^n)$ and $x \in \H^n$, then the proof follows from \eqref{trilap}. Let us assume by contradiction that \eqref{assumption3} is true, then the following equality holds
\begin{equation}\label{absurd2}
\mathcal{L}^2(\mathcal{L}\phi)(x)=cT^2 \mathcal{L}\phi(x),
\end{equation}
for every $\phi \in \mathcal{S}(\H^n)$ and $x \in \H^n$. 
According to our reasoning, it would exist a constant $c \in \R$ such that
\begin{equation}\label{vignola}
\mathcal{L}^2\phi(x)=c T^2\phi(x),
\end{equation}
for every $\phi=\mathcal{L}(\psi)$ for some $\psi \in \mathcal{S}(\H^n)$ and for every $x \in \H^n$. Then, by following the same reasoning adopted in the proof of Proposition \ref{prop:valint2}, in order to show a contradiction with \eqref{vignola} it is sufficient to individuate a function $\psi \in \mathcal{S}(\H^n)$ such that
\begin{equation}\label{ceretolo}
\Delta^2(\mathcal{L}\psi)(0) \neq 0.
\end{equation}

For example $\psi(x)=x_1^6 e^{-\|x\|_K^4}$ satisfies \eqref{ceretolo}.
Thus, \eqref{vignola} is not true, hence we reach a contradiction with \eqref{assumption3} and the thesis follows.
\end{proof}

\bibliography{CF-References}

\begin{thebibliography}{10}

\bibitem{AR02}
L.~Ambrosio and S.~Rigot.
\newblock Optimal mass transportation in the {H}eisenberg group.
\newblock {\em J. Funct. Anal.}, 208(2):261--301, 2004.

\bibitem{B09}
A.~Bonfiglioli.
\newblock Pizzetti's formula for {H}-type groups.
\newblock {\em Potential Anal.}, 31(4):311--333, 2009.

\bibitem{BLU07Stratified}
A.~Bonfiglioli, E.~Lanconelli, and F.~Uguzzoni.
\newblock {\em Stratified {L}ie groups and potential theory for their
  sub-{L}aplacians}.
\newblock Springer Monographs in Mathematics. Springer, Berlin, 2007.

\bibitem{B2002}
Andrea Bonfiglioli.
\newblock Expansion of the {H}eisenberg integral mean via iterated {K}ohn
  {L}aplacians: a {P}izzetti-type formula.
\newblock {\em Potential Anal.}, 17(2):165--180, 2002.

\bibitem{B2008}
Andrea Bonfiglioli.
\newblock Taylor formula for homogeneous groups and applications.
\newblock {\em Math. Z.}, 262(2):255--279, 2009.

\bibitem{BH91}
Nicolas Bouleau and Francis Hirsch.
\newblock {\em Dirichlet forms and analysis on {W}iener space}, volume~14 of
  {\em De Gruyter Studies in Mathematics}.
\newblock Walter de Gruyter \& Co., Berlin, 1991.

\bibitem{BMP12}
Marco Bramanti, Michele Miranda, Jr., and Diego Pallara.
\newblock Two characterization of {BV} functions on {C}arnot groups via the
  heat semigroup.
\newblock {\em Int. Math. Res. Not. IMRN}, (17):3845--3876, 2012.

\bibitem{Cygan79}
Jacek Cygan.
\newblock Heat kernels for class {$2$} nilpotent groups.
\newblock {\em Studia Math.}, 64(3):227--238, 1979.

\bibitem{Cygan81}
Jacek Cygan.
\newblock Subadditivity of homogeneous norms on certain nilpotent {L}ie groups.
\newblock {\em Proc. Amer. Math. Soc.}, 83(1):69--70, 1981.

\bibitem{DPValdinoci}
Eleonora Di~Nezza, Giampiero Palatucci, and Enrico Valdinoci.
\newblock Hitchhiker's guide to the fractional {S}obolev spaces.
\newblock {\em Bull. Sci. Math.}, 136(5):521--573, 2012.

\bibitem{FF15}
Fausto Ferrari and Bruno Franchi.
\newblock Harnack inequality for fractional sub-{L}aplacians in {C}arnot
  groups.
\newblock {\em Math. Z.}, 279(1-2):435--458, 2015.

\bibitem{FMPPS18}
Fausto Ferrari, Michele Miranda, Jr., Diego Pallara, Andrea Pinamonti, and
  Yannick Sire.
\newblock Fractional {L}aplacians, perimeters and heat semigroups in {C}arnot
  groups.
\newblock {\em Discrete Contin. Dyn. Syst. Ser. S}, 11(3):477--491, 2018.

\bibitem{Fol73}
Gerald~B. Folland.
\newblock A fundamental solution for a subelliptic operator.
\newblock {\em Bull. Amer. Math. Soc.}, 79:373--376, 1973.

\bibitem{Fol75}
Gerald~B. Folland.
\newblock Subelliptic estimates and function spaces on nilpotent lie groups.
\newblock {\em Ark. Mat.}, 13(2):161--207, 1975.

\bibitem{FS82}
Gerald~B. Folland and Elias~M. Stein.
\newblock {\em Hardy Spaces on Homogeneous groups}.
\newblock Princeton University Press, 1982.

\bibitem{FSSC1}
B.~Franchi, R.~Serapioni, and F.~Serra~Cassano.
\newblock Meyers-{S}errin type theorems and relaxation of variational integrals
  depending on vector fields.
\newblock {\em Houston J. Math.}, 22(4):859--890, 1996.

\bibitem{FSSC5}
Bruno Franchi, Raul Serapioni, and Francesco Serra~Cassano.
\newblock On the structure of finite perimeter sets in step 2 {C}arnot groups.
\newblock {\em J. Geom. Anal.}, 13(3):421--466, 2003.

\bibitem{Gar18}
Nicola Garofalo.
\newblock Some properties of sub-{L}aplaceans.
\newblock In {\em Proceedings of the {I}nternational {C}onference ``{T}wo
  nonlinear days in {U}rbino 2017"}, volume~25 of {\em Electron. J. Differ.
  Equ. Conf.}, pages 103--131. Texas State Univ.--San Marcos, Dept. Math., San
  Marcos, TX, 2018.

\bibitem{Garofalo}
Nicola Garofalo.
\newblock Fractional thoughts.
\newblock In {\em New developments in the analysis of nonlocal operators},
  volume 723 of {\em Contemp. Math.}, pages 1--135. Amer. Math. Soc.,
  [Providence], RI, [2019] \copyright 2019.

\bibitem{GN96}
Nicola Garofalo and Duy-Minh Nhieu.
\newblock Isoperimetric and {S}obolev inequalities for
  {C}arnot-{C}arath\'{e}odory spaces and the existence of minimal surfaces.
\newblock {\em Comm. Pure Appl. Math.}, 49(10):1081--1144, 1996.

\bibitem{GT1}
Nicola Garofalo and Giulio Tralli.
\newblock Feeling the heat in a group of {H}eisenberg type.
\newblock {\em Adv. Math.}, 381:Paper No. 107635, 42, 2021.

\bibitem{Gaveau75}
Bernard Gaveau.
\newblock Principe de moindre action et propagation de la chaleur pour le
  groupe d'{H}eisenberg.
\newblock {\em C. R. Acad. Sci. Paris S\'{e}r. A-B}, 281(10):Aii, A327--A328,
  1975.

\bibitem{Gaveau77}
Bernard Gaveau.
\newblock Principe de moindre action, propagation de la chaleur et estim\'{e}es
  sous elliptiques sur certains groupes nilpotents.
\newblock {\em Acta Math.}, 139(1-2):95--153, 1977.

\bibitem{Gre80}
Peter~C. Greiner.
\newblock Spherical harmonics on the {H}eisenberg group.
\newblock {\em Canad. Math. Bull.}, 23(4):383--396, 1980.

\bibitem{Hormander67}
Lars H\"ormander.
\newblock Hypoelliptic second order differential equations.
\newblock {\em Acta Math.}, 119:147--171, 1967.

\bibitem{Hul76}
A.~Hulanicki.
\newblock The distribution of energy in the {B}rownian motion in the {G}aussian
  field and analytic-hypoellipticity of certain subelliptic operators on the
  {H}eisenberg group.
\newblock {\em Studia Math.}, 56(2):165--173, 1976.

\bibitem{Landkof}
N.~S. Landkof.
\newblock {\em Foundations of modern potential theory}.
\newblock Die Grundlehren der mathematischen Wissenschaften, Band 180.
  Springer-Verlag, New York-Heidelberg, 1972.
\newblock Translated from the Russian by A. P. Doohovskoy.

\bibitem{LRV12}
Gian~Paolo Leonardi, S\'{e}verine Rigot, and Davide Vittone.
\newblock Isodiametric sets in the {H}eisenberg group.
\newblock {\em Rev. Mat. Iberoam.}, 28(4):999--1024, 2012.

\bibitem{Mag6}
Valentino Magnani.
\newblock The coarea formula for real-valued {L}ipschitz maps on stratified
  groups.
\newblock {\em Math. Nachr.}, 278(14):1689--1705, 2005.

\bibitem{Mag31}
Valentino Magnani.
\newblock A new differentiation, shape of the unit ball, and perimeter measure.
\newblock {\em Indiana Univ. Math. J.}, 66(1):183--204, 2017.

\bibitem{MSC2001}
Roberto Monti and Francesco Serra~Cassano.
\newblock Surface measures in {C}arnot-{C}arath\'eodory spaces.
\newblock {\em Calc. Var. Partial Differential Equations}, 13(3):339--376,
  2001.

\bibitem{Randall96}
Jennifer Randall.
\newblock The heat kernel for generalized {H}eisenberg groups.
\newblock {\em J. Geom. Anal.}, 6(2):287--316, 1996.

\bibitem{MarcelRieszA}
Marcel Riesz.
\newblock Int\'egrales de riemann-liouville et potentiels.
\newblock {\em Acta Sci. Math. Szeged}, 9:1--42, 1938.

\bibitem{MarcelRiesz}
Marcel Riesz.
\newblock L'int\'{e}grale de {R}iemann-{L}iouville et le probl\`eme de
  {C}auchy.
\newblock {\em Acta Math.}, 81:1--223, 1949.

\bibitem{Samko}
Stefan~G. Samko, Anatoly~A. Kilbas, and Oleg~I. Marichev.
\newblock {\em Fractional integrals and derivatives}.
\newblock Gordon and Breach Science Publishers, Yverdon, 1993.
\newblock Theory and applications, Edited and with a foreword by S. M.
  Nikol'ski\u{\i}, Translated from the 1987 Russian original, Revised by the
  authors.

\bibitem{SerraCassano2016}
Francesco Serra~Cassano.
\newblock Some topics of geometric measure theory in {C}arnot groups.
\newblock In {\em Geometry, Analysis and Dynamics on sub-{R}iemannian
  manifolds. {V}ol. 1}, EMS Ser. Lect. Math., pages 1--121. Eur. Math. Soc.,
  Z\"urich, 2016.

\bibitem{Yosida}
K\={o}saku Yosida.
\newblock {\em Functional analysis}.
\newblock Classics in Mathematics. Springer-Verlag, Berlin, 1995.
\newblock Reprint of the sixth (1980) edition.

\end{thebibliography}
\bibliographystyle{plain}

\end{document}